\documentclass[a4j,12pt,leqno]{amsart}
\usepackage{amsmath}
\usepackage{amssymb}
\usepackage{amsthm}
\usepackage{graphicx}
\usepackage{bm}
\usepackage{mathrsfs}
\usepackage{tinywncy}
\usepackage{comment}

\theoremstyle{plain} 
\newtheorem{theorem}{Theorem}[section] 
\newtheorem{lemma}[theorem]{Lemma}
\newtheorem{corollary}[theorem]{Corollary}
\newtheorem{proposition}[theorem]{Proposition}
\theoremstyle{definition} 
\newtheorem{definition}[theorem]{Definition}
\newtheorem{example}[theorem]{Example}
\theoremstyle{remark}
\newtheorem{remark}[theorem]{Remark}

\newcommand{\jump}[1]{\ensuremath{[\![#1]\!]}}
\newcommand{\Q}{\mathop{\mathbb{Q}}\nolimits}

\newcommand{\li}{\mathop{\mathrm{li}}\nolimits}

\newcommand{\wt}{\mathop{\mathrm{wt}}\nolimits}
\newcommand{\dep}{\mathop{\mathrm{dep}}\nolimits}
\newcommand{\pr}{\mathop{\mathrm{pr}}\nolimits}
\newcommand{\Li}{\mathop{\mathrm{Li}}\nolimits}

\begin{document}

\title{\uppercase{On functional equations of finite multiple polylogarithms}} 

\author{
	\textsc{Kenji Sakugawa and Shin-ichiro Seki} 
}
\date{}

\begin{abstract}
Recently, several people study finite multiple zeta values (FMZVs) and finite polylogarithms (FPs). In this paper, we introduce finite multiple polylogarithms (FMPs), which are natural generalizations of FMZVs and FPs, and we establish functional equations of FMPs. As applications of these functional equations, we calculate special values of FMPs containing generalizations of congruences obtained by Me\v{s}trovi\'c, Z. W. Sun, L. L. Zhao, Tauraso, and J. Zhao. We show supercongruences for certain generalized Bernoulli numbers and the Bernoulli numbers as an appendix.
\end{abstract}

\maketitle

\section{Introduction}
\label{sec:Introduction}
From the end of twentieth century to the beginning of twenty-first century, Hoffman and J. Zhao had started research about mod $p$ multiple harmonic sums, which are motivated by various generalizations of classical Wolstenholme's theorem.
Recently, Kaneko and Zagier introduced a new ``ad\'elic'' framework to describe the pioneer works by Hoffman and Zhao and they defined finite multiple zeta values (FMZVs). Let $k_1, \dots, k_m$ be positive integers and $\Bbbk := (k_1, \dots, k_m)$.
\begin{definition}[Kaneko and Zagier \cite{K, KZ}]
	{\em The finite multiple zeta value} $\zeta_{\mathcal{A}}(\Bbbk )$ is defined by
	\[
	\zeta_{\mathcal{A}}(\Bbbk ) := \Biggl( \sum_{p> n_1 > \cdots > n_m >0}\frac{1}{n_1^{k_1}\cdots n_m^{k_m}} \bmod{p} \Biggr)_p \in \mathcal{A}\hspace{4mm}
	\]
	and {\em the finite multiple zeta-star value} $\zeta_{\mathcal{A}}^{\star}(\Bbbk)$ is defined by
	\[
	\zeta_{\mathcal{A}}^{\star}(\Bbbk ) := \Biggl( \sum_{p-1\geq n_1 \geq \cdots \geq n_m \geq 1}\frac{1}{n_1^{k_1}\cdots n_m^{k_m}} \bmod{p} \Biggr)_p \in \mathcal{A}.
	\] 
	Here, the $\Q$-algebra $\mathcal{A}$ is defined by
	\[
	\mathcal{A} := \Biggl( \prod_{p}\mathbb{F}_p \Biggr) \left/ \Biggl( \bigoplus_{p}\mathbb{F}_p \Biggr) \right. ,
	\]
	where $p$ runs over all prime numbers.
\end{definition}
In this framework, Kaneko and Zagier established a conjecture, which states that there is an isomorphism between the $\Q$-algebra spanned by FMZVs and the quotient $\mathbb{Q}$-algebra modulo the ideal generated by $\zeta (2)$ of 
the $\Q$-algebra spanned by the usual multiple zeta values. 

On the other hand,  Kontsevich \cite{Ko}, Elbaz-Vincent and Gangl \cite{EG} introduced finite version of polylogarithms and studied functional equations of them. Based on their works, Mattarei and Tauraso \cite{MT1} calculated special values of finite polylogarithms.

Inspired by these studies, we introduce a finite version of multiple polylogarithms in the framework of Kaneko and Zagier:
\begin{definition}[See Definition \ref{def of FMP}]
	{\em The finite multiple polylogarithms} (FMPs) $\text{\rm \pounds}_{\mathcal{A}, \Bbbk}(t)$, $\text{\rm \pounds}_{\mathcal{A}, \Bbbk}^{\star}(t)$, $\widetilde{\text{\rm \pounds}}_{\mathcal{A}, \Bbbk}(t)$, and $\widetilde{\text{\rm \pounds}}_{\mathcal{A}, \Bbbk}^{\star}(t)$ are defined by
	\begin{align*}
	\text{\rm \pounds}_{\mathcal{A}, \Bbbk}(t) &:= \Biggl( \sum_{p> n_1 > \cdots > n_m >0}\frac{t^{n_1}}{n_1^{k_1}\cdots n_m^{k_m}} \bmod{p} \Biggr)_p \in \mathcal{A}_{\mathbb{Z} [t]},\\
	\text{\rm \pounds}_{\mathcal{A}, \Bbbk}^{\star}(t) &:= \Biggl( \sum_{p-1\geq n_1 \geq \cdots \geq n_m \geq 1}\frac{t^{n_1}}{n_1^{k_1}\cdots n_m^{k_m}} \bmod{p} \Biggr)_p \in \mathcal{A}_{\mathbb{Z} [t]},\\
	\widetilde{\text{\rm \pounds}}_{\mathcal{A}, \Bbbk}(t) &:= \Biggl( \sum_{p> n_1 > \cdots > n_m >0}\frac{t^{n_m}}{n_1^{k_1}\cdots n_m^{k_m}} \bmod{p} \Biggr)_p \in \mathcal{A}_{\mathbb{Z} [t]},\\
	\widetilde{\text{\rm \pounds}}_{\mathcal{A}, \Bbbk}^{\star}(t) &:= \Biggl( \sum_{p-1\geq n_1 \geq \cdots \geq n_m \geq 1}\frac{t^{n_m}}{n_1^{k_1}\cdots n_m^{k_m}} \bmod{p} \Biggr)_p \in \mathcal{A}_{\mathbb{Z} [t]}.
	\end{align*}
	Here, the $\Q$-algebra $\mathcal{A}_{\mathbb{Z}[t]}$ is defined by
	\[
	\mathcal{A}_{\mathbb{Z} [t]} = \Biggl( \prod_{p}\mathbb{F}_p[t] \Biggr) \left/ \Biggl( \bigoplus_{p}\mathbb{F}_p [t] \Biggr) \right. , 
	\]
	where $p$ runs over all prime numbers.
\end{definition}
The symbol $\text{\rm \pounds}$ is used for the finite polylogarithms by Elbaz-Vincent and Gangl in their paper \cite{EG}.

The main purpose of this paper is to establish functional equations of FMPs. Special cases of the main results are as follows:
\begin{theorem}[Main Theorem]
	The following functional equations hold in $\mathcal{A}_{{\mathbb Z}[t]}$:
	\begin{equation}
	\widetilde{\text{\rm \pounds}}_{\mathcal{A}, \Bbbk}^{\star}(t) = \widetilde{\text{\rm \pounds}}_{\mathcal{A}, \Bbbk^{\vee}}^{\star}(1-t)-\zeta_{\mathcal{A}}^{\star}(\Bbbk^{\vee}),
	\label{fn eq 1}\end{equation}
	\begin{equation}
	(-1)^{m-1}\text{\rm \pounds}_{\mathcal{A}, \Bbbk}(t)=\widetilde{\text{\rm \pounds}}_{\mathcal{A}, \overline{\Bbbk}}^{\star}(t)+\sum_{j=1}^{m-1}(-1)^j\text{\rm \pounds}_{\mathcal{A}, (k_1, \dots, k_j)}(t)\zeta_{\mathcal{A}}^{\star}(k_m, \dots, k_{j+1}).
	\label{fn eq 2}\end{equation}
	Here, $\overline{\Bbbk}$ is the reverse index and $\Bbbk^{\vee}$ the Hoffman dual of $\Bbbk$ $($see Subsection \ref{subsec:Notations for indices and the Hoffman dual}$)$.
	\label{introthm}\end{theorem}
The equality (\ref{fn eq 1}) is a generalization of the Hoffman duality \cite[Theorem 4.6]{Ho}. The precise version of the main results are Theorem \ref{MTA}, Corollary \ref{MTB}, Remark \ref{1-var remark}, and Theorem \ref{MTC} which consist of the multiple variable cases of (\ref{fn eq 1}) and (\ref{fn eq 2}), the $\mathcal{A}_2$-version of (\ref{fn eq 1}), and the $\mathcal{A}_n$-version of (\ref{fn eq 2}) for any positive integer $n$ (see Theorem \ref{gen of sun conj intro} below and Subsection \ref{subsec:The ring A} for the definition of $\mathcal{A}_n$). Main results are obtained as $\bmod$ $p$ reductions of generalizations (= Theorem \ref{theoremA} and Theorem \ref{theoremB}) of classical Euler's identity:
\begin{equation}
\sum_{n=1}^{N} (-1)^{n-1}\binom{N}{n}\frac{1}{n} = \sum_{n=1}^{N}\frac{1}{n},
\label{Euler's identity}\end{equation}
where $N$ is a positive integer (\cite{E}). The functional equation (\ref{fn eq 2}) and its generalization also hold for the usual multiple polylogarithms (Theorem \ref{UMP}).

As applications of the functional equations, we will calculate some special values of the finite multiple polylogarithms by using Tauraso and J. Zhao's results for the alternating multiple harmonic sums in Subsection \ref{subsec:Special values of F(S)MPs}.

Several people study supercongruences involving the harmonic numbers (Z. W. Sun, L. L. Zhao, Me\v{s}trovi\'c, and so on).  For instance, Z. W. Sun and L. L. Zhao proved the following congruence:
\begin{theorem}[Z. W. Sun and L. L. Zhao {\cite[Theorem 1.1]{SZ}}]
	Let $p$ be a prime number greater than  $3$. Then
	\[
	\sum_{k=1}^{p-1}\frac{H_k}{k2^k} \equiv \frac{7}{24}pB_{p-3} \pmod{p^2},
	\]
	where $H_k = \sum_{j=1}^k1/j$ is the $k$-th harmonic number and $B_{p-3}$ is the $(p-3)$-rd Bernoulli number.
	\label{Sun conjecture}\end{theorem}
We can regard such congruences as explicit formulas of special values of FMPs and we will give generalizations of some of them. As an application of main results,
we obtain the following theorem which is a generalization of Theorem \ref{Sun conjecture} (= the case $m=2$ in Theorem \ref{gen of sun conj intro}):
\begin{theorem}[cf. Theorem \ref{A_2 theorem}]
	Let $m$ be an even positive integer. Then we have
	\[
	\text{\rm \pounds}_{\mathcal{A}_2, \{ 1 \}^m}^{\star} (1/2)  = \left( \frac{2^{m+1}-1}{2^{m+1}}\frac{B_{p-m-1}}{m+1}p \bmod{p^2} \right)_p \ \text{in} \ \mathcal{A}_2,
	\]
	where $B_{p-m-1}$ is the $(p-m-1)$-st Bernoulli number and
	\[
	\text{\rm \pounds}_{\mathcal{A}_2, \{1 \}^m}^{\star}(1/2) := \left( \sum_{p-1 \geq n_1 \geq \cdots \geq n_m \geq 1}\frac{1}{n_1\cdots n_m2^{n_1}} \bmod{p^2} \right)_p \ \text{in} \ \mathcal{A}_2.
	\]
	Here, the $\Q$-algebra $\mathcal{A}_2$ is defined by
	\[
	\mathcal{A}_2 := \Biggl( \prod_{p}{\mathbb Z} /p^2{\mathbb Z} \Biggr) \left/ \Biggl( \bigoplus_{p}{\mathbb Z} /p^2{\mathbb Z} \Biggr) \right. .
	\]
	\label{gen of sun conj intro}\end{theorem}

Independently of us, Ono and Yamamoto gave another definition of FMPs and  established the shuffle relation in their preprint \cite{OY}. We will investigate a relation between our FMPs and Ono-Yamamoto's FMPs and calculate its special values.

This paper is organized as follows:

\noindent In Section \ref{sec:Generalizations of Euler's identity}, we prove generalizations of Euler's identity involving binomial coefficients by introducing some formal truncated integral operators.
In Section \ref{sec:Functional equations of finite multiple polylogarithms}, we define the ring $\mathcal{A}_{n, R}^{\Sigma}$ and recall some known results on FMZVs. Then we introduce FMPs and prove the functional equations by applying identities obtained in Section \ref{sec:Generalizations of Euler's identity}.
In Section \ref{sec:Special values of finite multiple polylogarithms},  as applications of the functional equations, we calculate special values of FMPs. We also study the relation between our FMPs and the ones defined by Ono and Yamamoto.
There are two appendices in this paper. 
In \ref{sec:An elementary proof of the congruences between Bernoulli numbers and generalized Bernoulli numbers}, we give an elementary proof of supercongruences for the generalized Bernoulli number for powers of the Teichm\"uller character and the Bernoulli numbers. \ref{sec:Table of sufficient conditions for congruences} is a table of sufficient conditions for special values of FMPs.

\section{Generalizations of Euler's identity}
\label{sec:Generalizations of Euler's identity}
\subsection{Notations for indices and the Hoffman dual}
\label{subsec:Notations for indices and the Hoffman dual}
Here, we recall an involution introduced by Hoffman on the set of indices.

We define the set $I$ by
\[
I := \coprod_{m \in \mathbb{Z}_{> 0}}(\underbrace{\mathbb{Z}_{>0} \times \cdots \times \mathbb{Z}_{>0}}_{m})
\]
and we call an element of $I$ {\em an index}. For an index $\Bbbk = (k_1, \ldots , k_m) \in I$ we define {\em the weight} (resp. {\em the depth}) of $\Bbbk$ to be $k_1+\cdots +k_m$ (resp. $m$) and we denote it by $\wt (\Bbbk)$ (resp. $\dep (\Bbbk)$).

For a non-negative integer $k$, the symbol $\{ k \}^m$ denotes $m$ repetitions $(k, \ldots , k) \in \mathbb{Z}_{\geq0}^m$ of $k$. Let $\mathbf{e}_i:=(\{0\}^{i-1}, 1, \{0\}^{m-i})$ when $m$ is clear from the context. Throughout this paper, we use the following operators for indices:

For three indices $\Bbbk =(k_1, \dots, k_m), \Bbbk_1 = (k_1', \dots, k_m')$, and $\Bbbk_2 = (k_1'', \dots, k_{m'}'')$, we define $\overline{\Bbbk}, \Bbbk_1\sqcup \Bbbk_2$, and $\Bbbk \oplus \Bbbk_1$ by $\overline{\Bbbk} := (k_m, \dots, k_1), \Bbbk_1 \sqcup \Bbbk_2:=(k_1', \dots, k_m', k_1'', \dots, k_{m'}''),$ and $\Bbbk \oplus \Bbbk_1 := (k_1+k_1', \dots, k_m+k_m')$, respectively.

Let $W$ be the free monoid generated by the set $\{ 0, 1 \}$. We denote by $W_1$ the set of words in $W$ of the form $\cdots 1$. We see that the correspondence \[
(k_1, \ldots , k_m) \mapsto \underbrace{0\cdots 0}_{k_1-1}1\underbrace{0\cdots 0}_{k_2-1}1\cdots 1\underbrace{0\cdots 0}_{k_m-1}1
\]
induces a bijection $w \colon I\xrightarrow{\sim}W_1$.
\begin{definition}[{cf.\ \cite[Section 3]{Ho}}]
	Let $\tau \colon W \xrightarrow{\sim} W$ be a monoid homomorphism defined by $\tau (0) =1$ and $\tau (1)=0$. Then we define an involution ${}^{\vee} \colon I \to I$ by the equality $w(\Bbbk^{\vee}) = \tau (w(\Bbbk)1^{-1})1$. We call this involution {\em the Hoffman dual}.
	\label{def of hof dual}\end{definition}
By the definition of the Hoffman dual, we see that $\Bbbk^{\vee \vee} = \Bbbk$ holds for any index $\Bbbk$. 
We can use the notation $\overline{\Bbbk}^{\vee}$ since the Hoffman dual and the reversal operator $\Bbbk \mapsto \overline{\Bbbk}$ commute. 
\begin{example}
	We have the following equalities:
	\begin{align*}
	&m^{\vee} = \{ 1 \}^{m}, \ (k_1, k_2)^{\vee} = (\{ 1 \}^{k_1-1}, 2, \{ 1 \}^{k_2-1}),
	\\
	&(k_1, k_2, k_3)^{\vee} = (\{ 1 \}^{k_1-1}, 2, \{ 1 \}^{k_2-2}, 2, \{ 1 \}^{k_3-1}),
	\\
	&(k_1, \{ 1 \}^{k_2-1})^{\vee} = (\{ 1 \}^{k_1-1}, k_2).
	\end{align*}
	Here, $m, k_1, k_2$ and $k_3$ are positive integers and the third equality holds when $k_2$ is greater than or equal to $2$.
\end{example}
The following lemma is useful for inductive arguments on weight:
\begin{lemma}
	Let $\Bbbk$ be an index and $\Bbbk^{\vee}$ its dual. Then we have $(\Bbbk \oplus \mathbf{e}_1)^{\vee}=\{1\} \sqcup \Bbbk^{\vee}$ and  $(\{1\} \sqcup \Bbbk)^{\vee}$ $=$ $(\Bbbk^{\vee}\oplus \mathbf{e}_1)$.
	\label{Induction lemma for dual}\end{lemma}
\begin{proof}
	These are obvious by the definition of the Hoffman dual.
\end{proof}
We can prove that the equalities $\wt (\Bbbk^{\vee})=\wt (\Bbbk)$ and $\dep (\Bbbk) + \dep (\Bbbk^{\vee}) = \wt (\Bbbk)+1$ hold for any index $\Bbbk$ by using the above lemma.
\begin{remark}
	Hoffman defined the dual $\Bbbk \mapsto \Bbbk^{\vee}$ by another way (see \cite[Section 3]{Ho}). Let $I_w := \{ \Bbbk \in I \mid \wt (\Bbbk)=w \}$ for a positive integer $w$ and $\mathcal{P} (\{ 1, 2, \dots w-1 \} )$ the power set of $\{ 1, 2, \dots , w-1 \}$. Then there is a bijection $\psi \colon I_w \xrightarrow{\sim} \mathcal{P} ( \{ 1, 2, \dots , w-1 \} )$ defined by the correspondence 
	\[
	\Bbbk = (k_1, \ldots , k_m) \mapsto \{ k_1 , k_1+k_2, \cdots , k_1+\cdots +k_{m-1} \}
	\]
	and the original Hoffman dual $\Bbbk^{\vee}$ is defined by
	\[
	\Bbbk^{\vee} := \psi^{-1}(\{ 1, 2, \dots , w-1 \} \setminus \psi (\Bbbk))
	\]
	for $\Bbbk \in I_w$. This definition is equivalent to the first definition since the identities in Lemma \ref{Induction lemma for dual} also hold for this dual.
\end{remark}
\subsection{Generalizations of Euler's identity and their corollaries}
\label{subsec:Generalizations of Euler's identity and their corollaries}
In this subsection, we state polynomial identities which will be used for the proof of our main results in Subsection \ref{subsec:Definitions and functional equations of finite multiple polylogarithms}. We will give the proofs of Theorem \ref{theoremA} and Theorem \ref{theoremB} in Subsection \ref{subsec:Proofs of Theorem A and Theorem B}. Through this subsection, let $R$ be a commutative ring including the field of rational numbers. 
\begin{theorem}
	Let $\Bbbk = (k_1, \dots , k_m)$ be an index of weight $w$ and $N$ a positive integer. Then the following polynomial identity holds in $R[t_1, \dots, t_m]:$
	{\small \begin{equation}
	\begin{split}
	&\sum_{N\geq n_1 \geq \cdots \geq n_m \geq 1}(-1)^{n_1}\binom{N}{n_1}\frac{t_1^{n_1-n_2}\cdots t_{m-1}^{n_{m-1}-n_m}t_m^{n_m}}{n_1^{k_1} \cdots n_m^{k_m}}= \\
	& \sum_{N\geq n_1 \geq \cdots \geq n_w \geq 1}\frac{(1-t_1)^{n_{l_1}-n_{l_1+1}} \cdots (1-t_{m-1})^{n_{l_{m-1}}-n_{l_{m-1}+1}}\{ (1-t_m)^{n_{l_m}}-1 \}}{n_1\cdots n_w},\hspace{20mm}
	\end{split}
	\label{maltivar identity}\end{equation}
	\begin{equation}
	\begin{split}
	&\sum_{N\geq n_1 \geq \cdots \geq n_m \geq 1}\frac{t_1^{n_1-n_2}\cdots t_{m-1}^{n_{m-1}-n_m}t_m^{n_m}}{n_1^{k_1} \cdots n_m^{k_m}}= \\
	& \sum_{N\geq n_1 \geq \cdots \geq n_w \geq 1}(-1)^{n_1}\binom{N}{n_1}\frac{(1-t_1)^{n_{l_1}-n_{l_1+1}} \cdots (1-t_{m-1})^{n_{l_{m-1}}-n_{l_{m-1}+1}}\{ (1-t_m)^{n_{l_m}}-1 \}}{n_1\cdots n_w},
	\end{split}
	\label{maltivar identity2}\end{equation}
}where $l_1=k_1, l_2=k_1+k_2, \dots, l_m=k_1+\cdots +k_m(=w)$.
	\label{theoremA}\end{theorem}
When we substitute $1$ for some of $t_1, \dots, t_{m-1}$ in the left hand side of (\ref{maltivar identity}), some terms vanish and some $0^0$ appear in the summation of the right hand side. We consider $0^0$ as $1$ since $t^{n-n}=t^0$ is $1$ as a constant function for a variable $t$ and a positive integer $n$. In such situations, we can rewrite (\ref{maltivar identity}) as follows:  
\begin{corollary}
	We use the same notations as in Theorem \ref{theoremA}. Let $\{ i_1, \dots , i_h \}$ be any subset of $\{ 1, \dots, w-1\}$ with the complement $\{ j_1, \dots, j_{h'} \}$ and put $(k_1', \dots, k_{m'}')$ $:=$ $(k_1+\cdots +k_{i_1}, k_{i_1+1}+\cdots +k_{i_2}, \dots, k_{i_h+1}+\cdots +k_{m})^{\vee}$. Then we have
	\begin{equation*}
	\begin{split}
	&\left. \sum_{N\geq n_1 \geq \cdots \geq n_m \geq 1}(-1)^{n_1}\binom{N}{n_1}\frac{t_1^{n_1-n_2}\cdots t_{m-1}^{n_{m-1}-n_m}t_m^{n_m}}{n_1^{k_1} \cdots n_m^{k_m}}\right|_{t_{i_1}=\cdots = t_{i_h}=1} = \\
	& \sum_{N\geq n_1 \geq \cdots \geq n_{m'} \geq 1}\frac{(1-t_{j_1})^{n_{L_1}-n_{L_1+1}}(1-t_{j_2})^{n_{L_2}-n_{L_2+1}} \cdots (1-t_{j_{h'}})^{n_{L_{h'}}-n_{L_{h'}+1}}\{ (1-t_m)^{n_{m'}}-1 \}}{n_1^{k_1'}\cdots n_{m'}^{k_{m'}'}},
	\end{split}
	\end{equation*}
	where $L_i:=l_{j_i}-j_i+i$ for $i=1, \dots, m'$. 
	\label{substitution corollary}\end{corollary}
\begin{proof}
	If the condition
	\begin{equation}
	n_{l_{i_1}}=n_{l_{i_1}+1}, \
	n_{l_{i_2}}=n_{l_{i_2}+1}, \
	\cdots, \ n_{l_{i_h}}=n_{l_{i_h}+1} 
	\tag{$\ast$}
	\end{equation}
	holds, then we have
	\[
	\begin{split}
	&(1-t_1)^{n_{l_1}-n_{l_1+1}}(1-t_2)^{n_{l_2}-n_{l_2+1}} \cdots (1-t_{m-1})^{n_{l_{m-1}}-n_{l_{m-1}+1}}\{ (1-t_m)^{n_{l_m}}-1 \} \\
	&=(1-t_{j_1})^{n_{l_{j_1}}-n_{l_{j_1}+1}}(1-t_{j_2})^{n_{l_{j_2}}-n_{l_{j_2}+1}}\cdots (1-t_{j_{h'}})^{n_{l_{j_{h'}}}-n_{l_{j_{h'}}+1}}\{ (1-t_m)^{n_{w}}-1 \}
	\end{split}
	\]
	as a polynomial equality. On the other hand, if $(n_1, \dots, n_w)$ does not satisfy the condition ($\ast$), then the term for $(n_1, \dots, n_w)$ of the right hand side of Theorem \ref{theoremA} (\ref{maltivar identity}) vanishes when $t_{i_1}=\cdots =t_{i_h}=1$.
	
	By the definition of the Hoffman dual, we have
	\[
	\{ l_1', l_2', \dots, l_{m'-1}' \} = \{ 1, \dots, w-1 \} \setminus \{ l_{i_1}, \dots, l_{i_h} \}
	\]
	where $l_1'=k_1', l_2'=k_1'+k_2', \dots, l_{m'-1}'=k_1'+\cdots +k_{m'-1}'$. Therefore we can rewrite the condition ($\ast$) as follows:
	\[
	n_1=\cdots =n_{l_1'}, \ n_{l_1'+1}=\cdots =n_{l_2'}, \ \cdots, \ n_{l_{m'-1}'+1}=\cdots =n_w.
	\]
	Hence we have the desired formula.
\end{proof}
In particular, we have the following corollary:
\begin{corollary}
	Let $\Bbbk = (k_1, \dots , k_m)$ be an index, $\Bbbk^{\vee} = (k_1', \dots , k_{m'}')$, and $N$ a positive integer. Then we have the polynomial identity
	\begin{equation}
	\sum_{N\geq n_1\geq \dots \geq n_m\geq 1}(-1)^{n_1}\binom{N}{n_1}\frac{t^{n_m}}{n_1^{k_1}\cdots n_m^{k_m}}
	=\sum_{N\geq n_1\geq \dots \geq n_{m'}\geq 1}\frac{(1-t)^{n_{m'}}-1}{n_1^{k_1'}\cdots n_{m'}^{k_{m'}'}},
	\label{1-var. Hoffman identity}\end{equation}
	in $R[t]$.
	\label{cor of thmA}\end{corollary}
\begin{remark}
	The case $\Bbbk = m \in \mathbb{Z}_{>0}$ of Corollary \ref{cor of thmA} (\ref{1-var. Hoffman identity}) gives Tauraso-Zhao's identity \cite[Lemam 5.5 (42)]{TZ} and the case $t=1$ of Corollary \ref{cor of thmA} (\ref{1-var. Hoffman identity}) gives Hoffman's identity \cite[Theorem 4.2]{Ho}. Dilcher's identity \cite{D} and Hern\'andez's identity \cite{He} are special cases of Hoffman's identity. All these are generalizations of Euler's identity (\ref{Euler's identity}).
\end{remark}
\begin{remark}
	Theorem \ref{theoremA} is also deduced from Kawashima-Tanaka's formula \cite[Theorem 2.6]{KT}, which is a generalization of the identity
	\begin{equation*}
	\sum_{n=0}^N (-1)^n \binom{N}{n}\frac{1}{n+1} = \frac{1}{N+1}
	\label{simple version}\end{equation*}
	(cf. \cite[Woord's solution]{He}). Our proof of Theorem \ref{theoremA} in Subsection \ref{subsec:Proofs of Theorem A and Theorem B} is quite different from the proof by Kawashima and Tanaka.
\end{remark}
\begin{theorem}
	Let $\Bbbk = (k_1, \dots , k_m)$ be an index of weight $w$ and $N$ a positive integer. Then the following identity holds in $R[t_1, t_2^{\pm 1}, \dots, t_m^{\pm 1}]:$
	{\footnotesize \begin{equation}
	\begin{split}
	&\sum_{N+1 > n_1 > \cdots > n_m > 0}(-1)^{n_m}\binom{N}{n_m}\frac{(t_1/t_2)^{n_1}\cdots (t_{m-1}/t_m)^{n_{m-1}}t_m^{n_m}}{n_1^{k_1}\cdots n_m^{k_m}} \\
	&= (-1)^{m-1}\sum_{N \geq n_1 \geq \cdots \geq n_w \geq 1}\frac{(1-t_m)^{n_{l_1}-n_{l_1+1}}\cdots (1-t_2)^{n_{l_{m-1}}-n_{l_{m-1}+1}}\{(1-t_1)^{n_{l_m}}-1\}}{n_1\cdots n_w}\\
	&\hspace{4mm}+\sum_{j=1}^{m-1}(-1)^{m-j-1}\left(\sum_{N+1>n_1>\cdots >n_j>0}\frac{(t_1/t_2)^{n_1}\cdots (t_j/t_{j+1})^{n_j}}{n_1^{k_1}\cdots n_j^{k_j}}\right) \times \\
	&\hspace{8mm}\left( \sum_{N\geq n_1\geq \cdots \geq n_{l_{m-j}}\geq 1}\frac{(1-t_m)^{n_{l_1}-n_{l_1+1}}\cdots (1-t_{j+2})^{n_{l_{m-j-1}}-n_{l_{m-j-1}+1}}\{(1-t_{j+1})^{n_{l_m-j}}-1\}}{n_1\cdots n_{l_{m-j}}} \right)
	\end{split}
	\label{multivar identity3}\end{equation}
}and
	{\tiny \begin{equation}
		\begin{split}
		&\sum_{N+1 > n_1 > \cdots > n_m > 0}\frac{(t_1/t_2)^{n_1}\cdots (t_{m-1}/t_m)^{n_{m-1}}t_m^{n_m}}{n_1^{k_1}\cdots n_m^{k_m}} \\
		&= (-1)^{m-1}\sum_{N \geq n_1 \geq \cdots \geq n_w \geq 1}(-1)^{n_1}\binom{N}{n_1}\frac{(1-t_m)^{n_{l_1}-n_{l_1+1}}\cdots (1-t_2)^{n_{l_{m-1}}-n_{l_{m-1}+1}}\{(1-t_1)^{n_{l_m}}-1\}}{n_1\cdots n_w}\\
		&+\sum_{j=1}^{m-1}(-1)^{m-j-1}\left(\sum_{N+1>n_1>\cdots >n_j>0}\frac{(t_1/t_2)^{n_1}\cdots (t_j/t_{j+1})^{n_j}}{n_1^{k_1}\cdots n_j^{k_j}}\right) \times \\
		&\left( \sum_{N\geq n_1\geq \cdots \geq n_{l_{m-j}}\geq 1}(-1)^{n_1}\binom{N}{n_1}\frac{(1-t_m)^{n_{l_1}-n_{l_1+1}}\cdots (1-t_{j+2})^{n_{l_{m-j-1}}-n_{l_{m-j-1}+1}}\{(1-t_{j+1})^{n_{l_m-j}}-1\}}{n_1\cdots n_{l_{m-j}}} \right),
		\end{split}
		\label{multivar identity4}\end{equation}
	}where $l_1=k_m, l_2=k_m+k_{m-1}, \dots, l_m=k_m+\cdots +k_1(=w)$.
	\label{theoremB}\end{theorem}
\begin{theorem}
	Let $\Bbbk = (k_1, \dots , k_m)$ be an index and $N$ a positive integer. Then the following identity holds in $R[t_1, \dots, t_m]:$
	\begin{multline*}
	(-1)^{m-1}\sum_{N+1>n_1>\cdots >n_m>0}\frac{t_1^{n_1}\dots t_m^{n_m}}{n_1^{k_1}\cdots n_m^{k_m}}
	=\sum_{N\geq n_1\geq \cdots \geq n_m\geq 1}\frac{t_m^{n_1}\cdots t_1^{n_m}}{n_1^{k_m}\cdots n_m^{k_1}}\\
	+\sum_{j=1}^{m-1}(-1)^j\left( \sum_{N+1>n_1>\cdots >n_j>0}\frac{t_1^{n_1}\cdots t_j^{n_j}}{n_1^{k_1}\cdots n_j^{k_j}}\right) \left( \sum_{N\geq n_1\geq \cdots \geq n_{m-j}\geq 1}\frac{t_m^{n_1}\cdots t_{j+1}^{n_{m-j}}}{n_1^{k_m}\cdots n_{m-j}^{k_{j+1}}} \right).
	\end{multline*}
	\label{A_n fneq}\end{theorem}
\begin{proof}
	By combining Theorem \ref{theoremA} (\ref{maltivar identity2}) and Theorem \ref{theoremB} (\ref{multivar identity4}), we have
	{\scriptsize \begin{equation*}
		\begin{split}
		&(-1)^{m-1}\sum_{N+1 > n_1 > \cdots > n_m > 0}\frac{(t_1/t_2)^{n_1}\cdots (t_{m-1}/t_m)^{n_{m-1}}t_m^{n_m}}{n_1^{k_1}\cdots n_m^{k_m}}
		=\sum_{N\geq n_1 \geq \cdots \geq n_m \geq 1}\frac{t_m^{n_1-n_2}\cdots t_{2}^{n_{m-1}-n_m}t_1^{n_m}}{n_1^{k_m} \cdots n_m^{k_1}}\\
		&+\sum_{j=1}^{m-1}(-1)^{j}\left(\sum_{N+1>n_1>\cdots >n_j>0}\frac{(t_1/t_2)^{n_1}\cdots (t_j/t_{j+1})^{n_j}}{n_1^{k_1}\cdots n_j^{k_j}}\right)
		\left( \sum_{N\geq n_1 \geq \cdots \geq n_{m-j} \geq 1}\frac{t_m^{n_1-n_2}\cdots t_{j+2}^{n_{m-j-1}-n_{m-j}}t_{j+1}^{n_{m-j}}}{n_1^{k_m} \cdots n_{m-j}^{k_{j+1}}} \right)
		\end{split}
		\end{equation*}
	}in $R[t_1, t_2^{\pm 1}, \dots, t_m^{\pm 1}]$. By replacing $t_1/t_2 \mapsto t_1, \dots, t_{m-1}/t_m \mapsto t_{m-1}$, we obtain the desired identity.
\end{proof}
As applications of Theorem \ref{A_n fneq}, we can prove not only a formula for the finite multiple polylogarithms (= Theorem \ref{MTC}) but also a formula for the usual multiple polylogarithms by taking the limit $N \to \infty$. We recall the definition of the multiple polylogarithms:
\begin{definition}
	Let $\Bbbk = (k_1, \dots, k_m)$ be an index and $z_1, \dots, z_m$ complex numbers satisfying at least one of the following conditions for absolute convergence:
	
	\vspace{2mm}
	\noindent (i) $|z_1|<1$ and $|z_i| \leq 1 \ (2 \leq i \leq m)$, \hspace{8mm} (ii) $|z_i| \leq 1 \ (1 \leq i \leq m)$ and $k_1 \geq 2$.
	\vspace{2mm}
	
	Then, we define {\em the multiple polylogarithms} by
	\begin{equation*}
	\Li_{\Bbbk}(z_1, \dots, z_m) := \sum_{n_1>n_2>\cdots>n_m\geq 1}\frac{z_1^{n_1}\cdots z_m^{n_m}}{n_1^{k_1}\cdots n_m^{k_m}}, \hspace{3mm} \Li_{\Bbbk}^{\star}(z_1, \dots, z_m) := \sum_{n_1\geq n_2\geq \cdots\geq n_m\geq 1}\frac{z_1^{n_1}\cdots z_m^{n_m}}{n_1^{k_1}\cdots n_m^{k_m}}.
	\end{equation*}
	If $k_1\geq 2$, then we define {\em the multiple zeta$($-star$)$ values} $\zeta(\Bbbk)$ and $\zeta^{\star}(\Bbbk)$ by
	$\zeta(\Bbbk):=\Li_{\Bbbk}(\{1\}^m)$ and $\zeta^{\star}(\Bbbk):=\Li_{\Bbbk}^{\star}(\{1\}^m)$, respectively.
\end{definition}
\begin{theorem}
	Let $\Bbbk = (k_1, \dots, k_m)$ be an index and $z_1, \dots, z_m$ complex numbers satisfying at least one of the following conditions$:$
	
	\vspace{2mm}
	\noindent (i) $|z_1|<1$, $|z_i| \leq 1 \ (2 \leq i \leq m-1)$, and $|z_m|<1$,
	
	\noindent(ii) $|z_i| \leq 1 \ (1 \leq i \leq m)$, $k_1 \geq 2$, and $k_m \geq 2$.
	\vspace{2mm}
	
	\noindent Then we have
	\begin{equation*}
	\sum_{j=0}^{m}(-1)^j\Li_{(k_1, \dots, k_j)}(z_1, \dots , z_j)\Li_{(k_m, \dots, k_{j+1})}^{\star}(z_m, \dots, z_{j+1})=0.
	\end{equation*}
	Here, we consider $\Li_{(k_1, \dots, k_j)}(z_1, \dots , z_j)$ $($resp. $\Li_{(k_m, \dots, k_{j+1})}^{\star}(z_m, \dots, z_{j+1}))$ as $1$ when $j=0$ $($resp. $j=m)$.
	\label{UMP}\end{theorem}
If $k_1 \geq 2$ and $k_m \geq 2$, then we have
\begin{equation}
\sum_{j=0}^m(-1)^j\zeta(k_1, \dots, k_j)\zeta^{\star}(k_m, \dots, k_{j+1})=0.
\label{usual non-star and star}\end{equation}
See Remark \ref{antipode}.
\subsection{Truncated integral operators}
\label{subsec:Truncated integral operators}
In this subsection, we introduce {\em truncated integral operators} to prove the theorems in Subsection \ref{subsec:Generalizations of Euler's identity and their corollaries}. Through this subsection, let $R$ be a commutative ring including the field of rational numbers $\mathbb{Q}$ and $N$ a positive integer. Let $t$ and $s$ be indeterminates.

Let $\int \! \ast dt \colon R\jump{t} \to R\jump{t}$ be the formal indefinite integral operator satisfying the condition that the constant term with respect to $t$ is equal to $0$, that is,
\[
\int \left( \sum_{n=0}^{\infty}a_nt^n \right) dt := \sum_{n=0}^{\infty}\frac{a_n}{n+1}t^{n+1}
.
\]
We prepare the following five $R$-linear operators:
\begin{align*}
&I_{t; R} \colon tR[t] \longrightarrow tR[t]\hspace{13mm}I_{t, s; R} \colon R[t, s] \longrightarrow R\jump{s/t}[t] \hspace{13mm}\tau_{t ; R}^{\leq N} \colon R\jump{t} \longrightarrow R[t]\\
&\hspace{10mm} f(t) \longmapsto \int \frac{f(t)}{t}dt, \hspace{15mm} f(t, s) \longmapsto \int \frac{f(t, s)}{t-s}ds, \hspace{17mm} \sum_{n=0}^{\infty}a_nt^n \longmapsto \sum_{n=0}^{N}a_nt^n,\\
&\pr_{t ; R} \colon R(\!(t^{-1})\!) \longrightarrow R[t] \hspace{66mm} \pr_{t ; R}^{-} \colon R (\!(t^{-1})\!) \to t^{-1}R\jump{t^{-1}}\\
&\hspace{10mm} \sum_{n=-\infty}^{n_0}a_nt^n \longmapsto \begin{cases} \sum_{n=0}^{n_0}a_nt^n & \text{if $n_0$ is non-negative,} \\ 0 & \text{otherwise,} \end{cases} \hspace{15mm}\sum_{n=n_0}^{\infty}a_nt^{-n} \longmapsto \sum_{n=1}^{\infty}a_nt^{-n}.\\
\end{align*}
Here, we consider the formal integral operator in the definition of $I_{t, s ; R}$ as an operator on $R(\!(t^{-1})\!)[s]$. For instance, we have
\begin{equation}
I_{t, s; R}(s^n) = \sum_{j=1}^{\infty}\frac{s^{n+j}t^{-j}}{j} 
\label{I_t,s}\end{equation}
for a non-negative integer $n$.
\begin{definition}
	We define {\em the truncated integral operators} $J_{t, s; R}^{\star}$ and $J_{t, s; R}^{N}$ by
	\begin{align*}
	&J_{t, s; R}^{\star} :=\pr_{t; R\jump{s}} \circ I_{t, s; R} \colon R[t, s] \longrightarrow R\jump{s}[t], \\ 
	&J_{t, s; R}^{N} := \tau_{s; R\jump{t^{-1}}}^{\leq N} \circ \pr_{t, R\jump{s}}^{-} \circ I_{t, s; R} \colon R[t, s] \longrightarrow t^{-1}R\jump{t^{-1}}[s].
	\end{align*}
	We can check easily that the image of $J_{t, s; R}^{\star}$ (resp. $J_{t, s; R}^{N}$) is included in $R[t, s]$ (resp. $t^{-1}R[t^{-1}, s])$.
\end{definition}
For simplicity, we omit the ring $R$ from our notations. The following Lemma \ref{I_0-I_1} and Lemma \ref{August31} are fundamental for the proofs of Theorem \ref{theoremA} and Theorem \ref{theoremB}.
\begin{lemma}Let $n$ be a positive integer. Then we have  the following identities$:$
	\begin{align}
	&I_t(t^n)=\frac{t^n}{n},
	\label{I_0'}\\
	&I_t((1-t)^n-1)=\sum_{j=1}^{n}\frac{(1-t)^j-1}{j}
	\label{I_0},\\
	&J_{t,s}^{\star}(t^n)=\sum_{j=1}^n\frac{t^{n-j}s^j}{j},
	\label{J1}\\
	&J_{t,s}^{\star}((1-t)^n-1)=\sum_{j=1}^n\frac{(1-t)^{n-j}\{ (1-s)^j-1 \}}{j}.
	\label{J2}
	\end{align}
	\label{I_0-I_1}\end{lemma}
\begin{proof}The equality (\ref{I_0'}) can be easily checked.
	We show the equality (\ref{I_0}). Set $T:=1-t$. Then the left hand side of (\ref{I_0}) equals to $$-\int\frac{T^k-1}{1-T}dT
	=\int\sum_{j=0}^{k-1} T^jdT=\sum_{j=1}^k\frac{T^j-1}{j}.$$
	By the definition of $T$, we obtain the equality (\ref{I_0}).
	Let us show the equality (\ref{J1}). By the equality (\ref{I_t,s}), the following equalities hold:
	\begin{equation*}
	J_{t,s}^{\star}(t^n)=\pr_{t, R\jump{s}}(t^nI_{t, s; R}(1))=\pr_{t, R\jump{s}}\left(\sum_{j=1}^{\infty}\frac{s^jt^{n-j}}{j} \right) =\sum_{j=1}^{n}\frac{s^jt^{n-j}}{j}.
	\end{equation*}
	Finally, we show the equality (\ref{J2}). Note that the following equalities hold:
	\begin{equation*}
	\begin{split}
	\frac{(1-t)^n}{t-s}=-\frac{(1-t)^{n-1}}{1-\frac{s-1}{t-1}} &=-(1-t)^{n-1}\left(\frac{1-(\frac{s-1}{t-1})^n}{1-\frac{s-1}{t-1}}+\frac{(\frac{s-1}{t-1})^n}{1-\frac{s-1}{t-1}}\right)\\
	&=-(1-t)^{n-1}\sum_{j=0}^{n-1}\left(\frac{s-1}{t-1}\right)^j +\frac{(1-s)^n}{t-s}\\
	&=-\sum_{j=0}^{n-1}(1-s)^j(1-t)^{n-j-1}+\frac{(1-s)^n}{t-s}.
	\end{split}
	\end{equation*}
	As $J_{t,s}^{\star}(f(s))=\pr_{t, R\jump{s}}(\int\frac{f(s)ds}{t-s})=0$ for each $f(s)\in R[s]$ by the equality (\ref{I_t,s}), we have
	{\footnotesize \begin{equation*}
	J_{t,s}^{\star}((1-t)^n-1) =J_{t,s}^{\star}((1-t)^n) =\int\left(-\sum_{j=0}^{n-1}(1-s)^j(1-t)^{n-j-1}\right)ds=\sum_{j=1}^n\frac{(1-t)^{n-j}\{ (1-s)^j-1\}}{j}.
	\end{equation*}
}This completes the proof of the lemma.
\end{proof}
Before we give the lemma for $J_{t,s}^{N}$, we prepare the following two auxiliary lemmas:
\begin{lemma}[{cf. \cite[Proof of Lemma 4.1]{ZHS2}}]Let $N$ be a positive integer.
	We have the following polynomial identities in $R[t]:$
	\begin{equation}
	\sum_{n=1}^{N}(-1)^{n}\binom{N}{n}\frac{t^n}{n} = \sum_{n=1}^{N}\frac{(1-t)^n-1}{n},
	\label{left start}\end{equation}
	\begin{equation}
	\sum_{n=1}^{N}\frac{t^n}{n} = \sum_{n=1}^{N}(-1)^{n}\binom{N}{n}\frac{(1-t)^n-1}{n}.
	\label{right start}\end{equation}
	\label{starting identity}\end{lemma}
\begin{proof}
	First, we remark that $(1-t)^N-1 = \sum_{n=1}^{N}\binom{N}{n}(-t)^n$. Then by applying $I_t$ to both sides and using Lemma \ref{I_0-I_1} (\ref{I_0'}) and (\ref{I_0}), we obtain the identity (\ref{left start}). The identity (\ref{right start}) is obtained by the substitution $t \mapsto 1-t$ and the Euler's identity (\ref{Euler's identity}), which is a special case of the identity (\ref{left start}).
\end{proof}
\begin{lemma}
	Let $j$ and $n$ be non-negative integers satisfying $j \leq n$. Then we have the following polynomial identity in $R[t]:$
	\[
	\sum_{k=j}^n\binom{n}{k}\binom{k}{j}t^k = \binom{n}{j}t^j(1+t)^{n-j}.
	\]	
	\label{double binomial coefficient}\end{lemma}
\begin{proof}
	By the binomial expansion, we have
	\[
	(t+s+ts)^n = \{t+(1+t)s\}^n=\sum_{j=0}^n\binom{n}{j}t^j(1+t)^{n-j}s^{n-j}.
	\]	
	On the other hand, we have
	\[
	\begin{split}
	(t+s+ts)^n &= \{ t(1+s)+s\}^n = \sum_{k=0}^n\binom{n}{k}t^k(1+s)^ks^{n-k}\\
	&=\sum_{k=0}^n\binom{n}{k}t^k\sum_{j=0}^k\binom{k}{j}s^{k-j}s^{n-k}=\sum_{j=0}^n\left\{\sum_{k=j}^n\binom{n}{k}\binom{k}{j}t^k\right\}s^{n-j}
	\end{split}
	\]
	Compare the coefficient of $s^{n-j}$.
\end{proof}
\begin{lemma}
	Let $N$ and $n$ be positive integers. Then we have the following identities in $R[t^{\pm1}, s]:$
	\begin{align}
	&J_{t, s}^N(t^n) = \sum_{j=n+1}^N\frac{s^jt^{n-j}}{j},\label{J_3}\\
	&J_{t, s}^N((1-t)^n-1) = -\sum_{j=1}^n\frac{(1-t)^{n-j}\{(1-s)^j-1\}}{j}+\left( \sum_{j=1}^N\frac{(s/t)^j}{j}\right)\{(1-t)^n-1\}.\label{J_4}
	\end{align}
	Here, we understand the summation in the right hand side of the equality $(\ref{J_3})$ as $0$ if $n+1$ is greater than $N$.
	\label{August31}\end{lemma}
\begin{proof}
	The equality (\ref{J_3}) is an immediate consequence of the equality (\ref{I_t,s}). We show the equality (\ref{J_4}).
	By the equality
	\[
	I_{t, s; R}((1-t)^n)=(1-t)^n\sum_{j=1}^{\infty}\frac{s^jt^{-j}}{j} = \sum_{k=0}^n(-1)^k\binom{n}{k}\left( \sum_{j=1}^{\infty}\frac{s^jt^{k-j}}{j}\right),
	\]
	we have
	{\footnotesize \begin{equation}
	\begin{split}
	J_{t, s}^{N}((1-t)^n) = \sum_{k=0}^n(-1)^k\binom{n}{k}\left( \sum_{j=k+1}^N\frac{s^jt^{k-j}}{j} \right)
	&=\sum_{k=0}^n(-1)^k\binom{n}{k}\left( \sum_{j=1}^N\frac{s^jt^{k-j}}{j}-\sum_{j=1}^{k}\frac{s^jt^{k-j}}{j}\right) \\
	&=\sum_{j=1}^N\frac{(s/t)^{j}}{j}(1-t)^n-\sum_{n \geq k \geq j \geq 1}(-1)^{k}\binom{n}{k}\frac{s^jt^{k-j}}{j}.
	\end{split}
	\label{(1-t)^n}\end{equation}
}Since the equality
	\[
	\sum_{j=1}^k\frac{(s/t)^j}{j}=\sum_{j=1}^k(-1)^j\binom{k}{j}\frac{(1-s/t)^j-1}{j}
	\]
	holds by Lemma \ref{starting identity} (\ref{right start}), we have
	\[
	\sum_{n \geq k \geq j \geq 1}(-1)^{k}\binom{n}{k}\frac{s^jt^{k-j}}{j} = \sum_{n \geq k \geq j \geq 1}(-1)^{j+k}\binom{n}{k}\binom{k}{j} \frac{(1-s/t)^j-1}{j}t^k.
	\]
	Furthermore, by Lemma \ref{double binomial coefficient}, we have
	\begin{equation*}
	\begin{split}
	\sum_{n \geq k \geq j \geq 1}(-1)^{j+k}\binom{n}{k}\binom{k}{j} \frac{(1-s/t)^j-1}{j}t^k
	&=\sum_{j=1}^n\binom{n}{j}t^j(1-t)^{n-j}\frac{(1-s/t)^j-1}{j}\\
	&=\sum_{j=1}^n\binom{n}{j}(1-t)^{n-j}\frac{(t-s)^j-t^j}{j}\\
	&=(1-t)^n\sum_{j=1}^n\binom{n}{j}\frac{1}{j}\left\{\left(\frac{t-s}{1-t}\right)^j-\left(\frac{t}{1-t}\right)^j\right\}.
	\end{split}
	\end{equation*}
	Therefore, according to Lemma \ref{starting identity} (\ref{left start}), we can delete the binomial coefficients completely as follows: 
	{\small \begin{equation*}
	\begin{split}
	\sum_{n \geq k \geq j \geq 1}(-1)^{j+k}\binom{n}{k}\binom{k}{j} \frac{(1-s/t)^j-1}{j}t^k
	&=(1-t)^n\sum_{j=1}^n\frac{1}{j}\left\{ \left( 1-\frac{s-t}{1-t}\right)^j-\left( 1-\frac{-t}{1-t}\right)^j\right\}\\
	&=(1-t)^n\sum_{j=1}^n\frac{1}{j}\left\{\left(\frac{1-s}{1-t}\right)^j-\left( \frac{1}{1-t}\right)^j\right\}\\
	&=\sum_{j=1}^n\frac{(1-t)^{n-j}\left\{(1-s)^j-1\right\}}{j}.
	\end{split}
	\end{equation*}
}Hence, we have the desired identity by the equalities (\ref{(1-t)^n}) and $J_{t, s}^{N}(1)=\sum_{j=1}^N\frac{(s/t)^j}{j}$.
\end{proof}
\subsection{Proofs of Theorem \ref{theoremA} and Theorem \ref{theoremB}}
\label{subsec:Proofs of Theorem A and Theorem B}
\begin{proof}[Proof of Theorem $\ref{theoremA}$]
	We show this theorem by the induction on the weight $w$ of the index. If $w=1$, the assertion of the theorem is nothing but Lemma \ref{starting identity}. We show only the equality (\ref{maltivar identity}) because the proof of the equality (\ref{maltivar identity2}) is completely the same. Now, we assume that the assertion holds for an index $\Bbbk = (k_1, \dots , k_m)$. Then it is sufficient to show that the assertions also hold
	for the indices $\Bbbk_1:=(k_1, \dots , k_m+1)$ and $\Bbbk_2 := (k_1, \dots , k_m, 1)$. 
	
	First, we consider the case $\Bbbk_1$. By Lemma \ref{I_0-I_1} (\ref{I_0'}), we have
	\begin{equation*}
	I_{t_m}\Biggl( \sum\nolimits' (-1)^{n_1}\binom{N}{n_1}\frac{t_1^{n_1-n_2}\cdots t_{m-1}^{n_{m-1}-n_m}t_m^{n_m}}{n_1^{k_1} \cdots n_m^{k_m}}\Biggr) 
	=\sum\nolimits' (-1)^{n_1}\binom{N}{n_1}\frac{t_1^{n_1-n_2}\cdots t_{m-1}^{n_{m-1}-n_m}t_m^{n_m}}{n_1^{k_1} \cdots n_m^{k_m+1}},
	\end{equation*}
	where the summation $\sum'$ runs over $N\geq n_1 \geq \cdots \geq n_m \geq 1$ and $I_{t_m}:=I_{t_m; R[t_1, \dots, t_{m-1}]}$. On the other hand, by Lemma \ref{I_0-I_1} (\ref{I_0}), we have
	\begin{equation*}
	\begin{split}
	&I_{t_m}\left( \sum_{N\geq n_1 \geq \cdots \geq n_w\geq 1} \frac{(1-t_1)^{n_{l_1}-n_{l_1+1}}\cdots (1-t_{m-1})^{n_{l_{m-1}}-n_{l_{m-1}+1}}\{ (1-t_m)^{n_{l_m}}-1 \}}{n_1\cdots n_w}\right) \\
	=& \sum_{N \geq n_1 \geq \cdots \geq n_{w+1} \geq 1} \frac{(1-t_1)^{n_{l_1}-n_{l_1+1}}\cdots (1-t_{m-1})^{n_{l_{m-1}}-n_{l_{m-1}+1}}\{ (1-t_m)^{n_{l_m+1}}-1 \}}{n_1\cdots n_{w+1}},
	\end{split}
	\end{equation*}
	where $l_1=k_1, l_2=k_1+k_2, \dots , l_m=k_1+\cdots +k_m(=w)$. Thus, the equality (\ref{maltivar identity}) in the theorem also holds for $\Bbbk_1$ by the induction hypothesis.
	
	Next, we check the equality for the index $\Bbbk_2$. By Lemma \ref{I_0-I_1} (\ref{J1}), we have
	{\footnotesize \begin{equation*}
	J_{t_m, t_{m+1}}^{\star}\left( \sum\nolimits'(-1)^{n_1}\binom{N}{n_1}\frac{t_1^{n_1-n_2}\cdots t_{m-1}^{n_{m-1}-n_m}t_m^{n_m}}{n_1^{k_1} \cdots n_m^{k_m}}\right)=\sum\nolimits''(-1)^{n_1}\binom{N}{n_1}\frac{t_1^{n_1-n_2}\cdots t_{m}^{n_{m}-n_{m+1}}t_{m+1}^{n_{m+1}}}{n_1^{k_1} \cdots n_m^{k_m}n_{m+1}},
	\end{equation*}
}where the summation $\sum''$ runs over $N\geq n_1 \geq \cdots \geq n_{m+1} \geq 1$ and $J_{t_{m}, t_{m+1}}^{\star}:=J_{t_m, t_{m+1}; R[t_1, \dots, t_{m-1}]}^{\star}$. On the other hand, by Lemma \ref{I_0-I_1} (\ref{J2}), we have
	\begin{equation*}
	\begin{split}
	&J_{t_m, t_{m+1}}^{\star}\left( \sum_{N\geq n_1 \geq \cdots \geq n_w \geq 1} \frac{(1-t_1)^{n_{l_1}-n_{l_1+1}}\cdots (1-t_{m-1})^{n_{l_{m-1}}-n_{l_{m-1}+1}}\{ (1-t_m)^{n_{l_m}}-1 \}}{n_1\cdots n_w}\right) \\
	=& \sum_{N \geq n_1 \geq \cdots \geq n_{w+1} \geq 1} \frac{(1-t_1)^{n_{l_1}-n_{l_1+1}}\cdots (1-t_m)^{n_{l_{m}}-n_{l_{m}+1}}\{ (1-t_{m+1})^{n_{l_m+1}}-1 \}}{n_1\cdots n_{w+1}}.
	\end{split}
	\label{J_2'}\end{equation*}
	Using the induction hypothesis, the assertion of the equality (\ref{maltivar identity}) holds for the index $\Bbbk_2$. This completes the proof of Theorem \ref{theoremA}.
\end{proof}
\begin{proof}[Proof of Theorem $\ref{theoremB}$]
	We show this theorem by the induction on the weight $w$ of the index. If $w=1$, the assertion of the theorem is nothing but Lemma \ref{starting identity}. We show only the equality (\ref{multivar identity3}) because the proof of the equality (\ref{multivar identity4}) is completely the same. Now, we assume that the assertion holds for an index $\Bbbk = (k_1, \dots , k_m)$. Then it is sufficient to show that the assertions also hold for the indices $\Bbbk_1:=(k_1+1, \dots , k_m)$ and $\Bbbk_2 := (1, k_1, \dots , k_m)$. 
	
	First, we consider the case $\Bbbk_1$. By Lemma \ref{I_0-I_1} (\ref{I_0'}), we have
	{\scriptsize \begin{equation*}
		I_{t_1}\Biggl( \sum\nolimits' (-1)^{n_m}\binom{N}{n_m}\frac{(t_1/t_2)^{n_1}\cdots (t_{m-1}/t_m)^{n_{m-1}}t_m^{n_m}}{n_1^{k_1}\cdots n_m^{k_m}}\Biggr) 
		=\sum\nolimits'(-1)^{n_m}\binom{N}{n_m}\frac{(t_1/t_2)^{n_1}\cdots (t_{m-1}/t_m)^{n_{m-1}}t_m^{n_m}}{n_1^{k_1+1}\cdots n_m^{k_m}},
		\end{equation*}
	}where the summation $\sum'$ runs over $N+1 > n_1 > \cdots > n_m > 0$ and $I_{t_1}:=I_{t_1; R[t_2^{\pm1}, \dots, t_{m}^{\pm1}]}$. On the other hand, by Lemma \ref{I_0-I_1} (\ref{I_0}), we have
	{\small \begin{equation*}
	\begin{split}
	&I_{t_1}(\text{R. H. S. of the equality (\ref{multivar identity3})})=\\
	&(-1)^{m-1}\sum_{N \geq n_1 \geq \cdots \geq n_{w+1} \geq 1}\frac{(1-t_m)^{n_{l_1}-n_{l_1+1}}\cdots (1-t_2)^{n_{l_{m-1}}-n_{l_{m-1}+1}}\{(1-t_1)^{n_{l_m+1}}-1\}}{n_1\cdots n_{w+1}}\\
	&+\sum_{j=1}^{m-1}(-1)^{m-j-1}\left(\sum_{N+1>n_1>\cdots >n_j>0}\frac{(t_1/t_2)^{n_1}\cdots (t_j/t_{j+1})^{n_j}}{n_1^{k_1+1}\cdots n_j^{k_j}}\right) \times \\
	&\hspace{8mm}\left( \sum_{N\geq n_1\geq \cdots \geq n_{l_{m-j}}\geq 1}\frac{(1-t_m)^{n_{l_1}-n_{l_1+1}}\cdots (1-t_{j+2})^{n_{l_{m-j-1}}-n_{l_{m-j-1}+1}}\{(1-t_{j+1})^{n_{l_m-j}}-1\}}{n_1\cdots n_{l_{m-j}}} \right)
	\end{split}
	\end{equation*}
}where $l_1=k_m, l_2=k_m+k_{m-1}, \dots , l_m=k_m+\cdots +k_1(=w)$. Thus, the equality (\ref{multivar identity3}) in the theorem also holds for the index $\Bbbk_1$ by the induction hypothesis.
	
	Next, we check the equality for the index $\Bbbk_2$. By Lemma \ref{August31} (\ref{J_3}), we have
	{\scriptsize \begin{equation*}
		J_{t_1, t_{0}}^{N}\left(\sum\nolimits' (-1)^{n_m}\binom{N}{n_m}\frac{(t_1/t_2)^{n_1}\cdots (t_{m-1}/t_m)^{n_{m-1}}t_m^{n_m}}{n_1^{k_1}\cdots n_m^{k_m}}\right)=\sum\nolimits''(-1)^{n_m}\binom{N}{n_m}\frac{(t_0/t_1)^{n_0}\cdots (t_{m-1}/t_m)^{n_{m-1}}t_m^{n_m}}{n_0n_1^{k_1}\cdots n_m^{k_m}},
		\end{equation*}
	}where the summation $\sum''$ runs over $N+1>n_0> n_1 > \cdots > n_{m}>0$ and $J_{t_{1}, t_{0}}^{N}:=J_{t_1, t_{0}; R[t_2^{\pm1}, \dots, t_{m}^{\pm1}]}^{N}$. On the other hand, by Lemma \ref{August31} (\ref{J_4}), we have
	{\footnotesize \begin{equation*}
	\begin{split}
	&(-1)^mJ_{t_1, t_0}^N(\text{R. H. S. of the equality (\ref{multivar identity3})})=\\
	&\sum_{N \geq n_1 \geq \cdots \geq n_{w+1} \geq 1}\frac{(1-t_m)^{n_{l_1}-n_{l_1+1}}\cdots (1-t_1)^{n_{l_{m}}-n_{l_{m}+1}}\{(1-t_0)^{n_{l_m+1}}-1\}}{n_1\cdots n_{w+1}}\\
	&-\left(\sum_{n_0=1}^{N}\frac{(t_0/t_1)^{n_0}}{n_0}\right)
	\left(\sum_{N \geq n_1 \geq \cdots \geq n_{w} \geq 1}\frac{(1-t_m)^{n_{l_1}-n_{l_1+1}}\cdots (1-t_2)^{n_{l_{m-1}}-n_{l_{m-1}+1}}\{(1-t_1)^{n_{l_m}}-1\}}{n_1\cdots n_{w}}\right)\\
	&+\sum_{j=1}^{m-1}(-1)^{j-1}\left(\sum_{N+1>n_0>n_1>\cdots >n_j>0}\frac{(t_0/t_1)^{n_0}\cdots (t_j/t_{j+1})^{n_j}}{n_0n_1^{k_1}\cdots n_j^{k_j}}\right) \times \\
	&\hspace{8mm}\left( \sum_{N\geq n_1\geq \cdots \geq n_{l_{m-j}}\geq 1}\frac{(1-t_m)^{n_{l_1}-n_{l_1+1}}\cdots (1-t_{j+2})^{n_{l_{m-j-1}}-n_{l_{m-j-1}+1}}\{(1-t_{j+1})^{n_{l_m-j}}-1\}}{n_1\cdots n_{l_{m-j}}} \right)
	\end{split}
	\end{equation*}
}Using the induction hypothesis, the assertion of the equality (\ref{multivar identity3}) holds for the index $\Bbbk_2$. This completes the proof of Theorem \ref{theoremB}.
\end{proof}
\section{Functional equations of finite multiple polylogarithms}
\label{sec:Functional equations of finite multiple polylogarithms}
\subsection{The ring $\mathcal{A}_{n, R}^{\Sigma}$}
\label{subsec:The ring A}
Kaneko and Zagier defined finite multiple zeta(-star) values as elements of the ${\mathbb Q}$-algebra $\mathcal{A} = (\prod_p \mathbb{F}_p) \otimes_{{\mathbb Z}} {\mathbb Q}$ where $p$ runs over the set of all rational primes. See \cite{K} and their forthcoming paper \cite{KZ}. The ring $\mathcal{A}$ has been used in a different context by Kontsevich \cite[2.2]{K02}. Now, we introduce more general rings of such a type for the definition of finite multiple polylogarithms.
\begin{definition}
	Let $R$ be a commutative ring and $\Sigma$ a family of ideals of $R$. Then we define the ring $\mathcal{A}_{n, R}^{\Sigma}$ for each positive integer $n$ as follows:
	\[
	\mathcal{A}_{n, R}^{\Sigma} := \Biggl( \prod_{ I\in \Sigma} R/I^n \Biggr) \left/ \Biggl( \bigoplus_{I \in \Sigma} R/I^n \Biggr) \right. .
	\] 
\end{definition}
Since we use only the case $\Sigma = \{ pR \mid p \ \text{is a rational prime} \}$, we omit the notation $\Sigma$ in the rest of this paper. We denote $\mathcal{A}_{n, {\mathbb Z}} = \mathcal{A}_n$ and $\mathcal{A}_{1, R} = \mathcal{A}_R$. Then the ring $\mathcal{A}_{1, {\mathbb Z}}$ coincides with $\mathcal{A}$. We will define the finite multiple polylogarithms as elements of $\mathcal{A}_{n, {\mathbb Z} [t_1, \dots , t_m]}$ (that is not equal to the polynomial ring $\mathcal{A}_n[t_1, \dots , t_m]$). Note that $\mathcal{A}_{n, R}$ is a $\mathbb{Q}$-algebra in some important case even if $R$ is not a $\mathbb{Q}$-algebra. 
\begin{example}
	We often use 
	\[
	\mathcal{A}_{{\mathbb Z} [t]} = \Biggl( \prod_{p}\mathbb{F}_p[t] \Biggr) \left/ \Biggl( \bigoplus_{p}\mathbb{F}_p[t] \Biggr) \right.
	\]
	in the next section. Here, the direct product and the direct sum run over the all rational primes and $\mathcal{A}_{{\mathbb Z} [t]}$ coincides with $\mathcal{B}$ defined by Ono and Yamamoto in their paper \cite{OY}.
\end{example}
We denote each element of $\mathcal{A}_{n, R}$ as $(a_p)_p$ where $a_p \in R/p^nR$, so $(a_p)_p = (b_p)_p$ holds if and only if $a_p = b_p$ for all but finitely many rational primes $p$. We may use the notation $(a_p)_p \in \mathcal{A}_{n, R}$ even if $a_p$ is not defined for finitely many rational primes $p$. See Example \ref{example} (3). We define the element $\mathbf{p}_n \in \mathcal{A}_{n, R}$ to be $(p \bmod{p^n})_p$ and we denote it by $\mathbf{p}$ when $n$ is clear from the context. Note that $\mathcal{A}_n/\mathbf{p}^m\mathcal{A}_n = \mathcal{A}_{m}$ for $n \geq m$.
\begin{example}
	We give some typical examples of elements of $\mathcal{A}_{n, R}$.
	\begin{itemize}
		\item[(1)] 
		$t^{\mathbf{p}} := (t^p)_p \in \mathcal{A}_{n, {\mathbb Z} [t]}$. 
		\item[(2)]
		For any rational number $r$, we define $r^{\mathbf{p}}$ (resp. $r$) to be $(r^p)_p$ (resp. $(r)_p$) in $\mathcal{A}$. Then $r^{\mathbf{p}}=r \in \mathcal{A}$ holds by Fermat's little theorem.
		\item[(3)]
		Let $k$ be a positive integer greater than 2. We use the notation $B_{\mathbf{p}-k}$ as
		\[
		B_{\mathbf{p}-k} = (B_{p-k} \bmod{p^n})_p \in \mathcal{A}_n ,
		\]
		where $B_m$ is the $m$-th Bernoulli number defined by
		\[
		\frac{te^t}{e^t-1} = \sum_{n=0}^{\infty}B_m\frac{t^m}{m!}.
		\]
		By the von Staudt-Clausen theorem \cite[Theorem 5.10]{Wa}, $B_{p-k}$ is a $p$-adic integer for a rational prime greater than $k$. Therefore, the element $B_{\mathbf{p}-k}$ is well-defined as an element of $\mathcal{A}_n$ and conjecturally, it is non-zero by the conjecture that there are infinitely many regular primes. For later use, we define $\widehat{B}_m$ to be $\frac{B_m}{m}$ for a positive integer $m$. For example, we use the notation $\widehat{B}_{\mathbf{p}-k}$ as ($\frac{B_{\mathbf{p}-k}}{\mathbf{p}-k} \bmod{p^n})_p  \in \mathcal{A}_n$.
	\end{itemize} 
	\label{example}\end{example}
The following lemma will be used for deducing the functional equations of the finite multiple polylogarithms from the generalizations of Euler's identity obtained in Section \ref{sec:Generalizations of Euler's identity}.
\begin{lemma}
	Let $n$ be a positive integer. Then we have
	\begin{equation}
	(-1)^n \binom{\mathbf{p}-1}{n} = 1 -\mathbf{p}H_n \ \text{in} \ \mathcal{A}_2,
	\label{A_2 vanish}\end{equation}
	where $H_n = \sum_{j=1}^n1/j$ is the $n$-th harmonic number.
	
	Let $p$ be an odd prime number greater than $n$. Then the following congruence is satisfied$:$
	\begin{equation}
	H_{p-n-1} \equiv H_{n} \pmod{p}.
	\label{H_{p-n}}\end{equation}
	Here, we define $H_0$ to be $1$.
	\label{vanish lemma}\end{lemma}
\begin{proof}
	For each prime number $p>n$, we have the following congruence:
	\[
	(-1)^n\binom{p-1}{n} = \prod_{j=1}^n\left( 1-\frac{p}{j} \right) \equiv 1-pH_n \pmod{p^2}.
	\]
	Therefore the first assertion holds. Let $p>n+1$. By the substitution $n \mapsto p-n$, we have
	\begin{equation}
	H_{p-n-1} = \sum_{k=1}^{p-n-1}\frac{1}{k} = \sum_{k=n+1}^{p-1}\frac{1}{p-k} \equiv -\sum_{k=n+1}^{p-1}\frac{1}{k} = -(H_{p-1} - H_{n}) \equiv H_n \pmod{p}.
	\end{equation}
\end{proof}
\subsection{Definitions and functional equations of finite multiple polylogarithms}
\label{subsec:Definitions and functional equations of finite multiple polylogarithms}
First, we recall the definition of the finite multiple zeta(-star) values (FMZ(S)Vs).
\begin{definition}
	Let $n$ be a positive integer and $\Bbbk = (k_1, \dots , k_m)$ an index. Then we define {\em the finite multiple zeta value} $\zeta_{\mathcal{A}_n}(\Bbbk )$ by
	\[
	\zeta_{\mathcal{A}_n}(\Bbbk ) := \Biggl( \sum_{p> n_1 > \cdots > n_m >0}\frac{1}{n_1^{k_1}\cdots n_m^{k_m}} \bmod{p^n} \Biggr)_p \in \mathcal{A}_n\hspace{4mm}
	\]
	and we define {\em the finite multiple zeta-star value} $\zeta_{\mathcal{A}_n}^{\star}(\Bbbk )$ by
	\[
	\zeta_{\mathcal{A}_n}^{\star}(\Bbbk ) := \Biggl( \sum_{p-1\geq n_1 \geq \cdots \geq n_m \geq 1}\frac{1}{n_1^{k_1}\cdots n_m^{k_m}} \bmod{p^n} \Biggr)_p \in \mathcal{A}_n.
	\] 
	\label{def of FMZV}\end{definition}
\begin{remark}
	Several people use the notation $\zeta_{\mathcal{A}_n}^{\bullet} (\Bbbk )$ instead of $\zeta_{\mathcal{A}_n}^{\bullet} (\overline{\Bbbk})$ for $\bullet \in \{\emptyset, \star\}$. Therefore we have to be careful when we read other papers.
\end{remark}
For later use, we summarize known results about FMZ(S)Vs.
\begin{proposition}
	Let $m, k, k_1, k_2, k_3$ be positive integers, $\Bbbk = (k_1, \dots , k_m)$ an index and $\bullet \in \{ \emptyset , \star \}$. Then the following equalities hold$:$  
\begin{align}
	&\zeta_{\mathcal{A}_2}(\{ k \}^m )=(-1)^{m-1}\zeta_{\mathcal{A}_2}^{\star}(\{ k \}^m )=(-1)^{m-1}k\frac{B_{\mathbf{p}-mk-1}}{mk+1}\mathbf{p},
	\label{Zhou-Cai}\\
	&\zeta_{\mathcal{A}_3}(k) = \begin{cases}
	\binom{k+1}{2}\widehat{B}_{\mathbf{p}-k-2}\mathbf{p}^2 & \text{if $k$ is odd,}\\
	k\left( \widehat{B}_{2\mathbf{p}-k-2}-2\widehat{B}_{\mathbf{p}-k-1} \right) \mathbf{p} & \text{if $k$ is even,}
	\end{cases}
	\label{RZA3}
	\end{align}
	{\small \begin{equation}
	\zeta_{\mathcal{A}_4}(k) = \begin{cases}
	-\binom{k+1}{2}\left( \widehat{B}_{2\mathbf{p}-k-3}-2\widehat{B}_{\mathbf{p}-k-2} \right) \mathbf{p}^2 & \text{if $k$ is odd,}\\
	-k\left( \widehat{B}_{3\mathbf{p}-k-3}-3\widehat{B}_{2\mathbf{p}-k-2}+3\widehat{B}_{\mathbf{p}-k-1}\right) \mathbf{p} - \binom{k+2}{3}\widehat{B}_{\mathbf{p}-k-3}\mathbf{p}^3 & \text{if $k$ is even,}
	\end{cases}
	\label{RZA4}
	\end{equation}}
	\begin{align}
	&\zeta_{\mathcal{A}}^{\bullet}(k_1, k_2) = (-1)^{k_1}\binom{k_1+k_2}{k_1}\frac{B_{\mathbf{p}-k_1-k_2}}{k_1+k_2},
	\label{dep2 for FMZV}\\
	&\zeta_{\mathcal{A}_2}(k_1, k_2) = \frac{1}{2}\left\{ (-1)^{k_2}k_1\binom{w+1}{k_2}-(-1)^{k_1}k_2\binom{w+1}{k_1}-w \right\} \frac{B_{\mathbf{p}-w-1}}{w+1}\mathbf{p}
	\label{dep2 A_2}\end{align}
	if $w=k_1+k_2$ is even,
	\begin{equation}
	\zeta_{\mathcal{A}_2}^{\star}(k_1, k_2) = \frac{1}{2}\left\{ (-1)^{k_2}k_1\binom{w+1}{k_2}-(-1)^{k_1}k_2\binom{w+1}{k_1}+w \right\} \frac{B_{\mathbf{p}-w-1}}{w+1}\mathbf{p}
	\label{dep2 A_2 star}\end{equation}
	if $w=k_1+k_2$ is even,
	\begin{equation}
	\zeta_{\mathcal{A}}(k_1, k_2, k_3) = -\zeta_{\mathcal{A}}^{\star}(k_1, k_2, k_3) = \frac{1}{2}\left\{ (-1)^{k_3}\binom{w'}{k_3}-(-1)^{k_1}\binom{w'}{k_1} \right\}\frac{B_{\mathbf{p}-w'}}{w'}
	\label{dep3 for FMZV}\end{equation}
	if $w'=k_1+k_2+k_3$ is odd,
	\begin{equation}
	\sum_{\sigma \in S_m}\zeta_{\mathcal{A}}^{\bullet}(\sigma (\Bbbk )) = 0,
	\label{sym sum}\end{equation}
	where $S_m$ denotes the $m$-th symmetric group and $\sigma (\Bbbk ) := (k_{\sigma (1)}, \dots , k_{\sigma (m)})$,
	\begin{align}
	&\zeta_{\mathcal{A}}^{\bullet}(\Bbbk ) = (-1)^{\wt (\Bbbk )}\zeta_{\mathcal{A}}^{\bullet}(\overline{\Bbbk}),
	\label{reverse for FMZV}\\
	&\zeta_{\mathcal{A}}^{\star}(\Bbbk ) = -\zeta_{\mathcal{A}}^{\star}(\Bbbk^{\vee}),
	\label{Hoffman's duality for FMZV}\\
	&\zeta_{\mathcal{A}}^{\star}(k_1, \{ 1 \}^{k_2-1}) = (-1)^{k_1}\zeta_{\mathcal{A}}(k_1, \{ 1 \}^{k_2-1}).
	\label{A-S for FMZV}\end{align}
	\label{FMZV's properties}\end{proposition}
\begin{proof}
	The equalities (\ref{Zhou-Cai}), (\ref{dep2 for FMZV}), (\ref{dep2 A_2}), (\ref{dep2 A_2 star}), and (\ref{dep3 for FMZV}) are \cite[Theorem 1.6]{Z}, \cite[Theorem 3.1]{Z} (cf. \cite[Theorem 6.1]{Ho}), \cite[Theorem 3.2]{Z}, \cite[Theorem 3.2]{Z}, and \cite[Theorem 3.5]{Z} (cf. \cite[Theorem 6.2]{Ho}), respectively. The equalities (\ref{RZA3}) and (\ref{RZA4}) are \cite[Theorem 5.1]{ZHS1} and \cite[Remark 5.1]{ZHS1}, respectively. The equalities (\ref{sym sum}), (\ref{reverse for FMZV}), (\ref{Hoffman's duality for FMZV}), and (\ref{A-S for FMZV}) are \cite[Theorem 4.4]{Ho}, \cite[Theorem 4.5]{Ho}, \cite[Theorem 4.6]{Ho}, and \cite[Theorem 5.1]{Ho}, respectively.
\end{proof}
Now, we define the finite multiple polylogarithms which are main objects in this paper.
\begin{definition}
	Let $n$ be a positive integer and $\Bbbk = (k_1, \dots , k_m)$ an index. Then we define the various multi-variable finite multiple polylogarithms as follows:
	
	\noindent{\em The finite harmonic multiple polylogarithm $(${\rm FHMP}$)$}:
	\[
	\text{\rm \pounds}_{\mathcal{A}_n, \Bbbk}^{\ast}(t_1, \dots , t_m) := \Biggl( \sum_{p>n_1> \dots >n_m>0}\frac{t_1^{n_1}\cdots t_m^{n_m}}{n_1^{k_1}\cdots n_m^{k_m}} \bmod{p^n} \Biggr)_p \in \mathcal{A}_{n, {\mathbb Z} [t_1, \dots , t_m]},\hspace{24mm}
	\]
	{\em The finite harmonic star-multiple polylogarithm $(${\rm FHSMP}$)$}:
	\[
	\text{\rm \pounds}_{\mathcal{A}_n, \Bbbk}^{\ast , \star}(t_1, \dots , t_m) := \Biggl( \sum_{p-1\geq n_1\geq \dots \geq n_m\geq 1}\frac{t_1^{n_1}\cdots t_m^{n_m}}{n_1^{k_1}\cdots n_m^{k_m}} \bmod{p^n} \Biggr)_p \in \mathcal{A}_{n, {\mathbb Z} [t_1, \dots , t_m]},\hspace{21mm}
	\]
	{\em The finite shuffle multiple polylogarithm $(${\rm FSMP}$)$}:
	\[
	\text{\rm \pounds}_{\mathcal{A}_n, \Bbbk}^{\cyr sh}(t_1, \dots , t_m) := \Biggl( \sum_{p>n_1> \dots >n_m>0}\frac{t_1^{n_1-n_2}\cdots t_{m-1}^{n_{m-1}-n_m}t_m^{n_m}}{n_1^{k_1}\cdots n_m^{k_m}} \bmod{p^n} \Biggr)_p \in \mathcal{A}_{n, {\mathbb Z} [t_1, \dots , t_m]},\hspace{4mm}
	\]
	{\em The finite shuffle star-multiple polylogarithm $(${\rm FSSMP}$)$}:
	\[
	\text{\rm \pounds}_{\mathcal{A}_n, \Bbbk}^{\cyr sh , \star}(t_1, \dots , t_m) := \Biggl( \sum_{p-1\geq n_1\geq \dots \geq n_m\geq 1}\frac{t_1^{n_1-n_2}\cdots t_{m-1}^{n_{m-1}-n_m}t_m^{n_m}}{n_1^{k_1}\cdots n_m^{k_m}} \bmod{p^n} \Biggr)_p \in \mathcal{A}_{n, {\mathbb Z} [t_1, \dots , t_m]}.
	\]
	We also define 1-variable finite (star-)multiple polylogarithms (F(S)MP) as follows:
	\begin{equation*}
	\begin{split}
	\text{\rm \pounds}_{\mathcal{A}_n, \Bbbk}^{\bullet}(t) &:= \text{\rm \pounds}_{\mathcal{A}_n, \Bbbk}^{\ast , \bullet} (t, \{ 1 \}^{m-1}) = \text{\rm \pounds}_{\mathcal{A}_n, \Bbbk}^{\cyr sh , \bullet} (\{ t \}^{m}) \in \mathcal{A}_{n, {\mathbb Z} [t]},\\
	\widetilde{\text{\rm \pounds}}_{\mathcal{A}_n, \Bbbk}^{\bullet}(t) &:= \text{\rm \pounds}_{\mathcal{A}_n, \Bbbk}^{\ast , \bullet} (\{ 1 \}^{m-1}, t) = \text{\rm \pounds}_{\mathcal{A}_n, \Bbbk}^{\cyr sh , \bullet} (\{ 1 \}^{m-1}, t) \in \mathcal{A}_{n, {\mathbb Z} [t]},
	\end{split}
	\end{equation*}
	where $\bullet \in \{ \emptyset , \star \}$. We call $\text{\rm \pounds}_{\mathcal{A}_n, m}^{\bullet}(t) = \widetilde{\text{\rm \pounds}}_{\mathcal{A}_n, m}^{\bullet}(t)$ {\em the $m$-th finite polylogarithm}.
	\label{def of FMP}\end{definition}
\begin{remark}
	The original definition of the $m$-th finite polylogarithm by Elbaz-Vincent and Gangl is the $p$-component of $\text{\rm \pounds}_{\mathcal{A}, m}(t)$ in $\mathbb{F}_p[t]$ (cf. \cite[Definition 5.1]{EG}).
\end{remark}
\begin{remark}
	Let $R$ be a commutative ring. For any subset $\{ i_1, \dots , i_h \}$ of $\{ 1, \dots , m \}$ and $r_1, \dots , r_h \in R$, the substitution mapping
	\[
	\mathcal{A}_{n, {\mathbb Z} [t_1, \dots , t_m]} \longrightarrow \mathcal{A}_{n, R[t_{j_1}, \dots , t_{j_{h'}}]}
	\]
	defined by
	\[
	(f_p(t_1, \cdots , t_m))_p \mapsto (\left. f_p(t_1, \dots , t_m) \right|_{t_{i_1}=r_1, \dots , t_{i_h}=r_h})_p
	\]
	where $\{ j_1, \dots , j_{h'} \}$ is the complement of $\{ i_1, \dots , i_h \}$ with respect to $\{ 1, \dots , m \}$. For example, we have
	\[
	\text{\rm \pounds}_{\mathcal{A}_n, \Bbbk}^{\bullet}(1) = \widetilde{\text{\rm \pounds}}_{\mathcal{A}_n, \Bbbk}^{\bullet}(1) = \zeta_{\mathcal{A}_n}^{\bullet}(\Bbbk ) \in \mathcal{A}_n
	\] 
	for $\bullet \in \{ \emptyset , \star \}$. Our definition of FMPs is natural in this sense.
\end{remark}
The following proposition is a generalization of Proposition \ref{FMZV's properties} (\ref{reverse for FMZV}) (cf. \cite[Lemma 5.4]{TZ}):
\begin{proposition}[Reversal formulas]
	Let $\Bbbk =(k_1, \dots, k_m)$ be an index and $\bullet \in \{ \emptyset , \star \}$. Then we have the following equality in $\mathcal{A}_{2, {\mathbb Z} [t_1^{\pm1}, \dots, t_m^{\pm1}]}:$
	{\footnotesize \begin{equation}
	\text{\rm \pounds}_{\mathcal{A}_2, \Bbbk}^{\ast , \bullet}(t_1, \dots , t_m)= (-1)^{\wt (\Bbbk )}(t_1\cdots t_m)^{\mathbf{p}}\Bigl( \text{\rm \pounds}_{\mathcal{A}_2, \overline{\Bbbk}}^{\ast , \bullet}(t_m^{-1}, \dots , t_1^{-1}) + \mathbf{p}\sum_{i=1}^mk_i\text{\rm \pounds}_{\mathcal{A}_2, \overline{\Bbbk \oplus \mathbf{e}_i}}^{\ast, \bullet}(t_m^{-1}, \dots, t_1^{-1})\Bigr).
	\label{reverse A_2}
	\end{equation}}
	In particular, we have
	\begin{align}
	&\text{\rm \pounds}_{\mathcal{A}, \Bbbk}^{\ast , \bullet}(t_1, \dots , t_m) = (-1)^{\wt (\Bbbk )}(t_1\cdots t_m)^{\mathbf{p}}\text{\rm \pounds}_{\mathcal{A}, \overline{\Bbbk}}^{\ast , \bullet}(t_m^{-1}, \dots , t_1^{-1}),\hspace{30mm}
	\label{reverse1}\\
	&\text{\rm \pounds}_{\mathcal{A}, {\Bbbk}}^{\bullet}(t) =(-1)^{\wt (\Bbbk )}t^{\mathbf{p}}\widetilde{\text{\rm \pounds}}_{\mathcal{A}, \overline{\Bbbk}}^{\bullet}(t^{-1}) , \ \ \widetilde{\text{\rm \pounds}}_{\mathcal{A}, {\Bbbk}}^{\bullet}(t) =(-1)^{\wt (\Bbbk )}t^{\mathbf{p}}\text{\rm \pounds}_{\mathcal{A}, \overline{\Bbbk}}^{\bullet}(t^{-1}).
	\label{reverse2}\end{align}
	\label{reverse prop}
\end{proposition}
\begin{proof}
	We show only the case $\bullet = \emptyset$ since the proof for the case $\bullet = \star$ is similar. By using the substitution trick $n_i \mapsto p-n_i$, we have
	\begin{equation*}
	\begin{split}
	&\text{\rm \pounds}_{\mathcal{A}_2, \Bbbk}^{\ast}(t_1, \dots , t_m)= \Biggl( \sum_{p>n_1>\cdots > n_m>0}\frac{t_1^{n_1}\cdots t_m^{n_m}}{n_1^{k_1}\cdots n_m^{k_m}} \bmod{p^2} \Biggr)_p \\
	&= \Biggl( \sum_{p>p-n_1>\cdots > p-n_m>0}\frac{t_1^{p-n_1}\cdots t_m^{p-n_m}}{(p-n_1)^{k_1}\cdots (p-n_m)^{k_m}} \bmod{p^2} \Biggr)_p \\
	&= (-1)^{\wt (\Bbbk )}(t_1\cdots t_m)^{\mathbf{p}}\Biggl( \sum_{p>n_m>\cdots > n_1>0}\frac{(p+n_m)^{k_m}\cdots (p+n_1)^{k_1}t_m^{-n_m}\cdots t_1^{-n_1}}{n_m^{2k_m}\cdots n_1^{2k_1}} \bmod{p^2} \Biggr)_p.
	\end{split}
	\end{equation*}
	Since $(p+n_m)^{k_m}\cdots (p+n_1)^{k_1} = n_m^{k_m}\cdots n_1^{k_1}+p\sum_{i=1}^mk_in_m^{k_m}\cdots n_i^{k_i-1}\cdots n_1^{k_1} \pmod{p^2}$, we have the equality (\ref{reverse A_2}).
\end{proof}
Our main results in this paper are Theorem \ref{MTA}, Corollary \ref{MTB} and Theorem \ref{MTC} below:
\begin{theorem}
	Let $\Bbbk =(k_1, \dots, k_m)$ be an index of weight $w$. Then we have the following functional equation in $\mathcal{A}_{2, {\mathbb Z} [t_1, \dots , t_m]}$ for FSSMPs$:$
	\begin{equation}
	\begin{split}
	\text{\rm \pounds}_{\mathcal{A}_2, \Bbbk}^{\cyr sh , \star}&(t_1, \dots , t_m)+\Bigl( \text{\rm \pounds}_{\mathcal{A}_2, \{1\}\sqcup \Bbbk}^{\cyr sh , \star}(1, t_1, \dots, t_m)-\text{\rm \pounds}_{\mathcal{A}_2, \mathbf{e}_1\oplus \Bbbk}^{\cyr sh , \star}(t_1, \dots, t_m) \Bigr)\mathbf{p}\\
	&=\text{\rm \pounds}_{\mathcal{A}_2, \{ 1 \}^w }^{\cyr sh , \star}(\{ 1 \}^{k_1-1}, 1-t_1 , \{ 1 \}^{k_2-1}, 1-t_2, \dots , \{ 1 \}^{k_m-1}, 1-t_m)\\
	&\hspace{7.5mm}- \text{\rm \pounds}_{\mathcal{A}_2, \{ 1 \}^w }^{\cyr sh , \star}(\{ 1 \}^{k_1-1}, 1-t_1 , \dots , \{ 1 \}^{k_{m-1}-1}, 1-t_{m-1}, \{ 1 \}^{k_m}).
	\end{split}
	\label{main functional equation}\end{equation}
	\label{MTA}\end{theorem}
\begin{proof}
	By Lemma \ref{vanish lemma} (\ref{A_2 vanish}), we have
	\begin{equation*}
	\begin{split}
	&\left(\sum_{p-1\geq n_1 \geq \cdots \geq n_m \geq 1}(-1)^{n_1}\binom{p-1}{n_1}\frac{t_1^{n_1-n_2}\cdots t_{m-1}^{n_{m-1}-n_m}t_m^{n_m}}{n_1^{k_1} \cdots n_m^{k_m}}\bmod{p^2}\right)_p \\
	= \ &\left(\sum_{p-1\geq n_1 \geq \cdots \geq n_m \geq 1}(1-pH_{n_1})\frac{t_1^{n_1-n_2}\cdots t_{m-1}^{n_{m-1}-n_m}t_m^{n_m}}{n_1^{k_1} \cdots n_m^{k_m}}\bmod{p^2}\right)_p \\
	= \ & \text{\rm \pounds}_{\mathcal{A}_2, \Bbbk}^{\cyr sh, \star}(t_1, \dots, t_m)-\mathbf{p}\left( \sum_{p-1 \geq n_1 \geq \cdots \geq n_m \geq 1}\frac{H_{n_1}t_1^{n_1-n_2}\cdots t_{m-1}^{n_{m-1}-n_m}t_m^{n_m}}{n_1^{k_1}\cdots n_m^{k_m}} \bmod{p^2}\right)_p.\\
	\end{split}
	\end{equation*}
	By substitutions $n_i \mapsto p-n_i$ and Lemma \ref{vanish lemma} (\ref{H_{p-n}}), we have
	\begin{equation*}
	\begin{split}
	&\left( \sum_{p-1 \geq n_1 \geq \cdots \geq n_m \geq 1}\frac{H_{n_1}t_1^{n_1-n_2}\cdots t_{m-1}^{n_{m-1}-n_m}t_m^{n_m}}{n_1^{k_1}\cdots n_m^{k_m}} \bmod{p}\right)_p\\
	= \ & \left( \sum_{p-1 \geq p-n_1 \geq \cdots \geq p-n_m \geq 1}\frac{H_{p-n_1}t_1^{(p-n_1)-(p-n_2)}\cdots t_{m-1}^{(p-n_{m-1})-(p-n_m)}t_m^{p-n_m}}{(p-n_1)^{k_1}\cdots (p-n_m)^{k_m}} \bmod{p} \right)_p\\
	= \ & (-1)^{\wt(\Bbbk)}\left( \sum_{p-1 \geq n_m \geq \cdots \geq n_1 \geq 1}\frac{(H_{n_1}-\frac{1}{n_1})t_1^{n_2-n_1}\cdots t_{m-1}^{n_m-n_{m-1}}t_m^{p-n_m}}{n_m^{k_m}\cdots n_1^{k_1}} \bmod{p} \right)_p\\
	= \ & (-1)^{\wt(\Bbbk)}\left( \sum_{p-1\geq n_m \geq \cdots \geq n_1 \geq n_0 \geq 1}\frac{t_1^{n_2-n_1}\cdots t_{m-1}^{n_m-n_{m-1}}t_m^{p-n_m}}{n_m^{k_m}\cdots n_1^{k_1}n_0} \right. \\
	&\hspace{20mm}\left. -\sum_{p-1 \geq n_m \geq \cdots \geq n_1 \geq 1}\frac{t_1^{n_2-n_1}\cdots t_{m-1}^{n_m-n_{m-1}}t_m^{p-n_m}}{n_m^{k_m}\cdots n_1^{k_1+1}}\bmod{p} \right)_p\\
	= \ & (-1)^{\wt(\Bbbk)}\left( \sum_{p-1\geq p-n_m \geq \cdots \geq p-n_1 \geq p-n_0 \geq 1}\frac{t_1^{(p-n_2)-(p-n_1)}\cdots t_{m-1}^{(p-n_m)-(p-n_{m-1})}t_m^{p-(p-n_m)}}{(p-n_m)^{k_m}\cdots (p-n_1)^{k_1}(p-n_0)} \right. \\
	&\left. -\sum_{p-1 \geq p-n_m \geq \cdots \geq p-n_1 \geq 1}\frac{t_1^{(p-n_2)-(p-n_1)}\cdots t_{m-1}^{(p-n_m)-(p-n_{m-1})}t_m^{p-(p-n_m)}}{(p-n_m)^{k_m}\cdots (p-n_1)^{k_1+1}}\bmod{p} \right)_p\\
= \ & -\left( \sum_{p-1\geq n_0 \geq n_1 \geq \cdots \geq n_m \geq 1}\frac{t_1^{n_1-n_2}\cdots t_{m-1}^{n_{m-1}-n_m}t_m^{n_m}}{n_0n_1^{k_1}\cdots n_m^{k_m}} \right.\\
&\hspace{15mm}\left.-\sum_{p-1\geq n_1 \geq \cdots \geq n_m \geq 1}\frac{t_1^{n_1-n_2}\cdots t_{m-1}^{n_{m-1}-n_m}t_m^{n_m}}{n_1^{k_1+1}\cdots n_m^{k_m}} \bmod{p} \right)_p\\
= \ &-\text{\rm \pounds}_{\mathcal{A}, \{1\}\sqcup \Bbbk}^{\cyr sh , \star}(1, t_1, \dots, t_m)+\text{\rm \pounds}_{\mathcal{A}, \mathbf{e}_1\oplus \Bbbk}^{\cyr sh , \star}(t_1, \dots, t_m).
	\end{split}
	\end{equation*}
	Therefore, we have the desired functional equation by Theorem {\ref{theoremA}}.
\end{proof}
When we substitute $1$ for some of the variables $t_1, \dots , t_{m-1}$ in (\ref{main functional equation}), we have the following functional equation:  
\begin{corollary}
	Let $\Bbbk = (k_1, \dots, k_m)$ be an index, $S = \{ i_1, \dots, i_h \}$ a subset of $\{ 1, \dots, m-1 \}$, and $\{ j_1, \dots, j_{h'} \}$ the complement of $\{ i_1, \dots , i_h \}$ with respect to $\{ 1, \dots , m-1 \}$. We define $\Bbbk_S := (k_1+\cdots +k_{i_1}, k_{i_1+1}+\cdots +k_{i_2}, \dots, k_{i_h+1}+\cdots +k_m)$ and $m'$ := $\dep (\Bbbk_S^{\vee})$. Then we have the following functional equation in $\mathcal{A}_{2, {\mathbb Z} [t_{j_1}, \dots, t_{j_{h'}}]}$ for FSSMPs$:$
	\begin{equation}
	\begin{split}
	&\left. \left\{ \text{\rm \pounds}_{\mathcal{A}_2, \Bbbk}^{\cyr sh , \star}(t_1, \dots , t_m)+\Bigl( \text{\rm \pounds}_{\mathcal{A}_2, \{1\}\sqcup \Bbbk}^{\cyr sh , \star}(1, t_1, \dots, t_m)-\text{\rm \pounds}_{\mathcal{A}_2, \mathbf{e}_1\oplus \Bbbk}^{\cyr sh , \star}(t_1, \dots, t_m) \Bigr)\mathbf{p} \right\} \right|_{t_{i_1}=\cdots =t_{i_h}=1}\\
	&= \ \text{\rm \pounds}_{\mathcal{A}_2, \Bbbk_S^{\vee} }^{\cyr sh , \star}(\{ 1 \}^{l_1}, 1-t_{j_1} , \{ 1 \}^{l_2}, 1-t_{j_2}, \dots , \{ 1 \}^{l_{h'}}, 1-t_{j_{h'}}, \{ 1 \}^{m'-M_{h'}-1}, 1-t_m)\\
	&\hspace{18mm}- \text{\rm \pounds}_{\mathcal{A}_2, \Bbbk_S^{\vee} }^{\cyr sh , \star}(\{ 1 \}^{l_1}, 1-t_{j_1} , \{ 1 \}^{l_2}, 1-t_{j_2}, \dots , \{ 1 \}^{l_{h'}}, 1-t_{j_{h'}}, \{ 1 \}^{m'-M_{h'}}),
	\end{split}
	\label{substitution functional equation}\end{equation}
	where $l_1= k_1+\cdots+k_{j_1}-j_1$, $l_2=k_{j_1+1}+\cdots+ k_{j_2}-j_2+j_1$, $\dots$, $l_{h'}=k_{j_{h'-1}+1}+\cdots+k_{j_{h'}}-j_{h'}+j_{h'-1}$, and $M_{h'} = k_1+\cdots +k_{j_{h'}}-j_{h'}+h'$. 
	\label{MTB}\end{corollary}
\begin{proof}
	This is obtained by combining the proof of Theorem \ref{MTA} and Corollary \ref{substitution corollary} .
\end{proof}
\begin{remark}
	In particular, we have the following functional equation in $\mathcal{A}_{2, {\mathbb Z} [t]}$ (cf. Corollary \ref{cor of thmA}):
	\begin{equation}
	\widetilde{\text{\rm \pounds}}_{\mathcal{A}_2, \Bbbk}^{\star}(t) + \bigl( \widetilde{\text{\rm \pounds}}_{\mathcal{A}_2, \{1\}\sqcup \Bbbk}^{\star}(t)-\widetilde{\text{\rm \pounds}}_{\mathcal{A}_2, \mathbf{e}_1\oplus \Bbbk}^{\star}(t) \bigr) \mathbf{p} = \widetilde{\text{\rm \pounds}}_{\mathcal{A}_2, \Bbbk^{\vee}}^{\star}(1-t)-\zeta_{\mathcal{A}_2}^{\star}(\Bbbk^{\vee}).
	\label{gen. of Hoffman's duality}\end{equation}
	Therefore, we also have the functional equation (\ref{fn eq 1}) in Introduction. The case $t=1$ gives the Hoffman duality (Proposition \ref{FMZV's properties} (\ref{Hoffman's duality for FMZV})) and its generalization in $\mathcal{A}_2$ (\cite[Theorem 2.11]{Z}):
	\begin{equation}
	\zeta_{\mathcal{A}_2}^{\star}(\Bbbk)+\bigl( \zeta_{\mathcal{A}_2}^{\star}(\{1\}\sqcup \Bbbk) - \zeta_{\mathcal{A}_2}^{\star}(\mathbf{e}_1\oplus \Bbbk) \bigr)\mathbf{p} = -\zeta_{\mathcal{A}_2}^{\star}(\Bbbk^{\vee}).
	\label{Hoffman-Zhao duality}\end{equation}
	\label{1-var remark}\end{remark}
\begin{theorem}
	Let $n$ be a positive integer and $\Bbbk = (k_1, \dots , k_m)$ an index. Then we have the following functional equation in $\mathcal{A}_{n, {\mathbb Z} [t_1, \dots, t_m]}:$
	\begin{equation}
	\sum_{j=0}^{m}(-1)^j\text{\rm \pounds}_{\mathcal{A}_n, (k_1, \dots, k_j)}^{\ast}(t_1, \dots, t_j)\text{\rm \pounds}_{\mathcal{A}_n, (k_m, \dots, k_{j+1})}^{\ast , \star}(t_m, \dots, t_{j+1})=0.
	\label{A_n formula}\end{equation}
	Here, we consider $\text{\rm \pounds}_{\mathcal{A}_n, (k_1, \dots, k_j)}^{\ast}(t_1, \dots, t_j)$ $($resp. $\text{\rm \pounds}_{\mathcal{A}_n, (k_m, \dots, k_{j+1})}^{\ast , \star}(t_m, \dots, t_{j+1}))$ as $1$ when $j=0$ $($resp. $j=m)$.
	\label{MTC}\end{theorem}
\begin{proof}
	This is an immediate consequence of Theorem \ref{A_n fneq}.
\end{proof}
\begin{corollary}
	Let $n$, $k$, and $m$ be positive integers and $\Bbbk = (k_1, \dots , k_m)$ an index. Then the following equalities hold$:$
	\begin{align}
	&\text{\rm \pounds}_{\mathcal{A}, \{ k \}^m}^{\ast}(\{ 1 \}^{i-1}, t, \{ 1 \}^{m-i})+(-1)^m\text{\rm \pounds}_{\mathcal{A}, \{ k \}^m}^{\ast, \star}(\{ 1 \}^{m-i}, t, \{ 1 \}^{i-1})=0,\label{Tauraso-Zhao}\\
	&(-1)^{m-1}\text{\rm \pounds}_{\mathcal{A}_n, \Bbbk}(t)=\widetilde{\text{\rm \pounds}}_{\mathcal{A}_n, \overline{\Bbbk}}^{\star}(t)+\sum_{j=1}^{m-1}(-1)^j\text{\rm \pounds}_{\mathcal{A}_n, (k_1, \dots, k_j)}(t)\zeta_{\mathcal{A}_n}^{\star}(k_m, \dots, k_{j+1}),\label{nonstar-star}\\
	&(-1)^{m-1}\widetilde{\text{\rm \pounds}}_{\mathcal{A}_n, \Bbbk}(t)=\text{\rm \pounds}_{\mathcal{A}_n, \overline{\Bbbk}}^{\star}(t)+\sum_{j=1}^{m-1}(-1)^j\zeta_{\mathcal{A}_n}(k_1, \dots, k_j)\text{\rm \pounds}_{\mathcal{A}_n, (k_m, \dots, k_{j+1})}^{\star}(t),\label{nonstar and star2}\\
	&\sum_{j=0}^m(-1)^j\zeta_{\mathcal{A}_n}(k_1, \dots , k_j)\zeta_{\mathcal{A}_n}^{\star}(k_m, \dots , k_{j+1}) = 0,\label{non-star and star}
	\end{align}
	Here, we consider $\zeta_{\mathcal{A}_n}^{\bullet}(\emptyset)$ as $1$ for $\bullet \in \{\emptyset, \star\}$.
	\label{zeta corollary}\end{corollary}
\begin{proof}
	We obtain the equality (\ref{Tauraso-Zhao}) by the substitution $t_1=\cdots = t_{i-1}=t_{i+1}=\cdots =t_m=1, t_i=t$ and Lemma \ref{FMZV's properties} (\ref{Zhou-Cai}). The equalities (\ref{nonstar-star}), (\ref{nonstar and star2}), and (\ref{non-star and star}) are clear. 
\end{proof}
\begin{remark}
	The equality (\ref{Tauraso-Zhao}) has been proved by Tauraso and J. Zhao (\cite[Lemma 5.9]{TZ}). 
	By considering the case $\Bbbk = (k_1, \{ 1 \}^{k_2-1})$ and $n=1$ in the equality (\ref{non-star and star}), we have
	\[
	\zeta_{\mathcal{A}}(k_1, \{ 1 \}^{k_2-1}) + (-1)^{k_2}\zeta_{\mathcal{A}}^{\star}(\{ 1 \}^{k_2-1}, k_1)=0
	\]
	since $\zeta_{\mathcal{A}}^{\star}(\{ 1 \}^k)=0$ for every positive integer $k$. Therefore, the equality (\ref{non-star and star}) is a generalization of Proposition \ref{FMZV's properties} (\ref{A-S for FMZV}) since
	\[
	\zeta_{\mathcal{A}}^{\star}(\{ 1 \}^{k_2-1}, k_1) = (-1)^{k_1+k_2-1}\zeta_{\mathcal{A}}^{\star}(k_1, \{ 1 \}^{k_2-1})
	\]
	by Proposition \ref{FMZV's properties} (\ref{reverse for FMZV}).
	The equality (\ref{non-star and star}) and its analogue for the usual multiple zeta values (the equality (\ref{usual non-star and star})) are consequences of the explicit formula of the antipode of the harmonic algebra or the Hopf algebra of quasi-symmetric functions (\cite[Theorem 3.2]{Ho2} or \cite[Theorem 3.1]{Ho}). See also \cite[Theorem3]{Zl}, \cite[Proposition 6]{IKOO}, \cite[Proposition 7,1]{Ka2}, and \cite[Proposition 3.7]{Y}.
	\label{antipode}\end{remark}
As an application of Remark \ref{1-var remark} (\ref{Hoffman-Zhao duality}) and Corollary \ref{zeta corollary} (\ref{non-star and star}), we give another proof of the following theorem which is a part of recent deep works by Kh.\ Hessami Pilehrood, T. Hessami Pilehrood, and Tauraso. The original proof is based on the identity \cite[Theorem 2.2]{PPT} which is different from our identities in Subsection \ref{subsec:Generalizations of Euler's identity and their corollaries}.
\begin{theorem}[{\cite[Theorem 4.3]{PPT}}]
	Let $k_1$ and $k_2$ be positive integers satisfying the condition that $k_1+k_2$ is even. Then we have
	\begin{equation}
	\zeta_{\mathcal{A}_2}(\{ 1 \}^{k_1-1}, 2 , \{ 1 \}^{k_2-1}) = \frac{1}{2}\left\{ 1 - (-1)^{k_2}\binom{k_1+k_2+1}{k_1+1} \right\} \frac{B_{\mathbf{p}-k_1-k_2-1}}{k_1+k_2+1}\mathbf{p},
	\label{PPT1}\end{equation}
	\begin{equation}
	\zeta_{\mathcal{A}_2}^{\star}(\{ 1 \}^{k_1-1}, 2 , \{ 1 \}^{k_2-1}) = \frac{1}{2}\left\{ 1 - (-1)^{k_2}\binom{k_1+k_2+1}{k_2+1} \right\} \frac{B_{\mathbf{p}-k_1-k_2-1}}{k_1+k_2+1}\mathbf{p}.
	\label{PPT2}\end{equation}
	\label{PPT theorem}\end{theorem}
\begin{remark}
	They also calculated $\zeta_{\mathcal{A}}^{\bullet}(\{ 2 \}^{k_1}, 3, \{ 2 \}^{k_2})$ and $\zeta_{\mathcal{A}}^{\bullet}(\{ 2 \}^{k_1}, 1 , \{ 2 \}^{k_2})$ where $k_1$ and $k_2$ are positive integers and $\bullet \in \{ \emptyset , \star \}$ (\cite[Theorem 4.1 and Theorem 4.2]{PPT}).
\end{remark}
\begin{proof}[Proof of Theorem $\ref{PPT theorem}$]
	Let $k_1$ and $k_2$ be positive integers such that $k_1+k_2$ is even. Let $w:=k_1+k_2+1$. First, we show the star case. By Remark \ref{1-var remark} (\ref{Hoffman-Zhao duality}), Proposition \ref{FMZV's properties} (\ref{dep2 for FMZV}), (\ref{dep2 A_2 star}), and (\ref{dep3 for FMZV}), we have
	{\small \begin{equation*}
	\begin{split}
	\zeta_{\mathcal{A}_2}^{\star}(\{1\}^{k_1-1}, 2, \{1\}^{k_2-1}) &= -\zeta_{\mathcal{A}_2}^{\star}(k_1, k_2)-\bigl( \zeta_{\mathcal{A}_2}^{\star}(1, k_1, k_2)-\zeta_{\mathcal{A}_2}^{\star}(k_1+1, k_2) \bigr)\mathbf{p}\\
	&= -\frac{1}{2}\left\{ (-1)^{k_2}k_1\binom{w}{k_2}-(-1)^{k_1}k_2\binom{w}{k_1}-w+1 \right\} \frac{B_{\mathbf{p}-w}}{w}\mathbf{p}\\
	&\hspace{4mm}-\Biggl( -\frac{1}{2} \left\{(-1)^{k_2}\binom{w}{k_2} +w\right\} \frac{B_{\mathbf{p}-w}}{w} -(-1)^{k_1+1}\binom{w}{k_1+1}\frac{B_{\mathbf{p}-w}}{w} \Biggr) \mathbf{p} \\
	&= \frac{1}{2}\left\{ 1-(-1)^{k_2}\binom{w}{k_2+1} \right\} \frac{B_{\mathbf{p}-w}}{w}\mathbf{p}.
	\end{split}
	\end{equation*}
}Let $\Bbbk = (\{1\}^{k_1-1}, 2, \{1\}^{k_2-1}) =: (l_1, \dots, l_{w-2})$. By Corollary \ref{zeta corollary} (\ref{non-star and star}), we have
	\begin{equation}
	\zeta_{\mathcal{A}_2}(\Bbbk)=\zeta_{\mathcal{A}_2}^{\star}(\overline{\Bbbk})+\sum_{j=1}^{w-3}(-1)^{j}\zeta_{\mathcal{A}_2}(l_1, \dots, l_j)\zeta_{\mathcal{A}_2}^{\star}(l_{w-2}, \dots, l_{j+1}).
	\label{41}\end{equation}
	We see that one of the following cases is satisfied for $j=1, \dots, w-3$: 
	
	\noindent (i) At least one of $\zeta_{\mathcal{A}_2}(l_1, \dots, l_j)$ and $\zeta_{\mathcal{A}_2}^{\star}(l_{w-2}, \dots, l_{j+1})$ is zero,
	
	\noindent (ii) Both of $\zeta_{\mathcal{A}_2}(l_1, \dots, l_j)$ and $\zeta_{\mathcal{A}_2}^{\star}(l_{w-2}, \dots, l_{j+1})$ belong to $\mathbf{p}\mathcal{A}_2$. 
	
	\noindent Therefore, the summation in the equality (\ref{41}) vanishes and we have
	\[
	\zeta_{\mathcal{A}_2}(\Bbbk)=\zeta_{\mathcal{A}_2}^{\star}(\overline{\Bbbk}).
	\]
	This completes the proof.
\end{proof}
\subsection{Functional equations for the index $\{ 1\}^m$}
\label{subsec:Functional equations for the index 1^m}
In this subsection, we argue about functional equations of FMPs of the index $\{1\}^m$.
\begin{lemma}
	Let $m$ be a positive integer. Then
	\begin{align}
	&\widetilde{\text{\rm \pounds}}_{\mathcal{A}, \{1\}^m}^{\star}(t) = \text{\rm \pounds}_{\mathcal{A}, m}(1-t),
	\label{1^m, 1}\\
	&\text{\rm \pounds}_{\mathcal{A}, \{ 1 \}^m}(t) = (-1)^{m-1}\text{\rm \pounds}_{\mathcal{A}, m}(1-t).
	\label{MT Fn-Eq}\end{align}
	\label{multi-FPL}\end{lemma}
\begin{proof}
	By Proposition \ref{FMZV's properties} (\ref{Zhou-Cai}), the cases $\Bbbk = \{ 1 \}^m$ in theorem \ref{introthm} (\ref{fn eq 1}) and (\ref{fn eq 2}) give the equalities (\ref{1^m, 1}) and (\ref{MT Fn-Eq}), respectively.
\end{proof}
By Lemma \ref{multi-FPL} and Proposition \ref{reverse prop} (\ref{reverse2}), we can express every FMP of the index $\{1\}^m$ by a FP. Therefore, we can obtain functional equations of FMPs of the index $\{1\}^m$ from functional equations of FPs. For example, we get distribution properties for FMPs of the index $\{1\}^m$ by the following result by Elbaz-Vincent and Gangl: 
\begin{proposition}[Elbaz-Vincent and Gangl {\cite[Proposition 5.7 (2)]{EG}}]
	Let $n$ be a {\em non-zero} integer and $m$ a positive integer. Let $\zeta_n$ be a primitive $|n|$-th root of unity. Then we have the following equality in $\mathcal{A}_{{\mathbb Z}[\zeta_n, t]}:$
	\begin{equation}
	\text{\rm \pounds}_{\mathcal{A}, m}(t^n) = n^{m-1}\sum_{k=0}^{|n|-1}\frac{1-t^{n\mathbf{p}}}{1-(\zeta_n^kt)^{\mathbf{p}}}\text{\rm \pounds}_{\mathcal{A}, m}(\zeta_n^kt).
	\label{dist0}\end{equation}
	\label{dist}\end{proposition}
\begin{theorem}[Distribution properties for FMPs of the index $\{1\}^m$]
	Let $n$ be an non-zero integer and $m$ a positive integer. Let $\zeta_n$ be a primitive $|n|$-th root of unity. Then the following equalities hold in $\mathcal{A}_{{\mathbb Z}[\zeta_n]\jump{t}}:$
	\begin{align}
	&\text{\rm \pounds}_{\mathcal{A}, \{1\}^m}(1-t^n) = n^{m-1}\sum_{k=0}^{|n|-1}\frac{1-t^{n\mathbf{p}}}{1-(\zeta_{n}^kt)^{\mathbf{p}}}\text{\rm \pounds}_{\mathcal{A}, \{1\}^m}(1-\zeta_n^kt),
	\label{dist1}\\
	&\widetilde{\text{\rm \pounds}}_{\mathcal{A}, \{1\}^m}\left( \frac{1}{1-t^n} \right) = n^{m-1}\sum_{k=0}^{|n|-1}\widetilde{\text{\rm \pounds}}_{\mathcal{A}, \{1\}^m}\left( \frac{1}{1-\zeta_n^kt} \right),\\
	&\text{\rm \pounds}_{\mathcal{A}, \{1\}^m}^{\star}\left( \frac{1}{1-t^n} \right) = n^{m-1}\sum_{k=0}^{|n|-1}\text{\rm \pounds}_{\mathcal{A}, \{1\}^m}^{\star}\left( \frac{1}{1-\zeta_n^kt} \right),
	\label{first dist}\\
	&\widetilde{\text{\rm \pounds}}_{\mathcal{A}, \{1\}^m}^{\star}(1-t^n) = n^{m-1}\sum_{k=0}^{|n|-1}\frac{1-t^{n\mathbf{p}}}{1-(\zeta_{n}^kt)^{\mathbf{p}}}\widetilde{\text{\rm \pounds}}_{\mathcal{A}, \{1\}^m}^{\star}(1-\zeta_n^kt).
	\label{dist4}
	\end{align}
	\label{dist thm}\end{theorem}
\begin{proof}
	These are obtained by Proposition \ref{dist}. Note that $(1-t^n)^{\mathbf{p}}=1-t^{n\mathbf{p}}$ and $(1-\zeta_n^kt)^{\mathbf{p}} = 1-(\zeta_n^kt)^{\mathbf{p}}$ in $\mathcal{A}_{{\mathbb Z}[\zeta_n, t]}$.
\end{proof}
\begin{corollary}
	Let $m$ be a positive integer. Then the following equalities hold in $\mathcal{A}_{{\mathbb Z}[t]}:$
	\begin{align}
	&\widetilde{\text{\rm \pounds}}_{\mathcal{A}, \{1\}^m}(t)=(-1)^{m-1}\widetilde{\text{\rm \pounds}}_{\mathcal{A}, \{1\}^m}(1-t),\label{ZS}\\
	&\text{\rm \pounds}_{\mathcal{A}, \{1\}^m}^{\star}(t) = (-1)^{m-1}\text{\rm \pounds}_{\mathcal{A}, \{1\}^m}^{\star}(1-t).
	\label{ZS star}\end{align}
	\label{ZS lemma}\end{corollary}
\begin{proof}
	Let $n=-1$ in Theorem \ref{dist thm}. Then we have the desired formulas by replacing $1/(1-t)$ with $t$.
\end{proof}
\begin{remark}
	Lemma \ref{multi-FPL} (\ref{MT Fn-Eq}) has been proved by Mattarei and Tauraso (\cite[The proof of Theorem 2.3]{T}, \cite[Lemma 3.2]{MT2}) and Lemma \ref{ZS lemma} (\ref{ZS}) has been proved by L. L. Zhao and Z. W. Sun (\cite[Theorem 1.2]{ZS}).
\end{remark}
\section{Special values of finite multiple polylogarithms}
\label{sec:Special values of finite multiple polylogarithms}
\subsection{Special values of F(S)MPs}
\label{subsec:Special values of F(S)MPs}
We calculate some special values of F(S)MPs in $\mathcal{A}$ and $\mathcal{A}_2$ by applying our main results. 
\begin{lemma}[Tauraso and J. Zhao \cite{TZ}]
	Let $m$ be an integer greater than $1$. Let $k_1$ and $k_2$ be positive integers such that $w:= k_1+k_2$ is odd. Then we have the following equalities$:$
	\begin{align}
	&\text{\rm \pounds}_{\mathcal{A}, m}(-1) = \frac{1-2^{m-1}}{2^{m-2}}\frac{B_{\mathbf{p}-m}}{m},
	\label{m, -1}\\
	&\text{\rm \pounds}_{\mathcal{A}, (k_1, k_2)}(-1)=\widetilde{\text{\rm \pounds}}_{\mathcal{A}, (k_1, k_2)}(-1) = \frac{2^{w-1}-1}{2^{w-1}}\frac{B_{\mathbf{p}-w}}{w},
	\label{k_1, k_2, -1, non-star}\\
	&\text{\rm \pounds}_{\mathcal{A}, (k_1, k_2)}^{\star}(-1) = \widetilde{\text{\rm \pounds}}_{\mathcal{A}, (k_1, k_2)}^{\star}(-1)=\frac{1-2^{w-1}}{2^{w-1}}\frac{B_{\mathbf{p}-w}}{w}.
	\label{k_1, k_2, -1, star}\end{align}
	\label{aux}\end{lemma}
\begin{proof}
	The equalities (\ref{m, -1}), (\ref{k_1, k_2, -1, non-star}), and (\ref{k_1, k_2, -1, star}) are \cite[Corollary 2.3]{TZ}, \cite[Theorem 3.1 (17)]{TZ}, and \cite[Theorem 3.1 (18)]{TZ}, respectively.
\end{proof}
\begin{proposition}
	Let $m$ be an integer greater than $1$. Let $k_1$ and $k_2$ be positive integers such that $w:= k_1+k_2$ is odd. Then we have the following equalities$:$
	\begin{align}
	&\text{\rm \pounds}_{\mathcal{A}_2, m}(-1)= \begin{cases} \frac{m(2^m-1)}{2^m}\widehat{B}_{\mathbf{p}-m-1}\mathbf{p}
	\label{A_2, m, -1} & \text{if $m$ is even},\\
	\frac{2^{m-1}-1}{2^{m-2}}(2\widehat{B}_{\mathbf{p}-m}-\widehat{B}_{2\mathbf{p}-m-1}) & \text{if $m$ is odd},\end{cases} \\
	&\text{\rm \pounds}_{\mathcal{A}_2, (k_1, k_2)}(-1) =  \frac{1-2^{w-1}}{2^{w-1}}(2\widehat{B}_{\mathbf{p}-w}-\widehat{B}_{2\mathbf{p}-w-1})+\frac{k_2(1-2^{k_1-1})}{2^{k_1-1}}\widehat{B}_{\mathbf{p}-k_1}\widehat{B}_{\mathbf{p}-k_2-1}\mathbf{p},\label{57}\\
	&\text{\rm \pounds}_{\mathcal{A}_2, (k_1, k_2)}^{\star}(-1) =  \frac{2^{w-1}-1}{2^{w-1}}(2\widehat{B}_{\mathbf{p}-w}-\widehat{B}_{2\mathbf{p}-w-1})+\frac{k_2(1-2^{k_1-1})}{2^{k_1-1}}\widehat{B}_{\mathbf{p}-k_1}\widehat{B}_{\mathbf{p}-k_2-1}\mathbf{p},\label{59}\\
	&\widetilde{\text{\rm \pounds}}_{\mathcal{A}_2, (k_1, k_2)}(-1) =  \frac{1-2^{w-1}}{2^{w-1}}(2\widehat{B}_{\mathbf{p}-w}-\widehat{B}_{2\mathbf{p}-w-1})+\frac{k_1(1-2^{k_2-1})}{2^{k_2-1}}\widehat{B}_{\mathbf{p}-k_1-1}\widehat{B}_{\mathbf{p}-k_2}\mathbf{p},\label{58}\\
	&\widetilde{\text{\rm \pounds}}_{\mathcal{A}_2, (k_1, k_2)}^{\star}(-1) =  \frac{2^{w-1}-1}{2^{w-1}}(2\widehat{B}_{\mathbf{p}-w}-\widehat{B}_{2\mathbf{p}-w-1})+\frac{k_1(1-2^{k_2-1})}{2^{k_2-1}}\widehat{B}_{\mathbf{p}-k_1-1}\widehat{B}_{\mathbf{p}-k_2}\mathbf{p}\label{60},
	\end{align}
	where we assume that $k_1$ $($resp. $k_2)$ is greater than $1$ in the equalities $(\ref{57})$ and $(\ref{59})$ $($resp. $(\ref{58})$ and $(\ref{60}))$.
	\label{A_2 aux}\end{proposition}
\begin{proof}
	The equality (\ref{A_2, m, -1}) is obtained by Z. H. Sun's results (\cite[Theorem 5.2 (b), Corollary 5.2 (a)]{ZHS1}) and the relation
	\begin{equation}
	\text{\rm \pounds}_{\mathcal{A}_n, m}(-1)=-\zeta_{\mathcal{A}_n}(m)+\frac{1}{2^{m-1}}\left( \sum_{k=1}^{\frac{p-1}{2}}\frac{1}{k^m}\bmod{p^n}\right)_p,
	\label{trivial formula}\end{equation}
	where $n$ is any positive integer. Tauraso and J. Zhao also proved the even case of the equality (\ref{A_2, m, -1}) (\cite[Corollary 2.3]{TZ}). Now, we consider the following relation:
	\[
	\text{\rm \pounds}_{\mathcal{A}_2, k_1}(-1)\zeta_{\mathcal{A}_2}(k_2) = \text{\rm \pounds}_{\mathcal{A}_2, (k_1, k_2)}(-1) + \widetilde{\text{\rm \pounds}}_{\mathcal{A}_2, (k_2, k_1)}(-1) + \text{\rm \pounds}_{\mathcal{A}_2, k_1+k_2}(-1).
	\]
	Since $k_1+k_2$ is odd, we have $\text{\rm \pounds}_{\mathcal{A}_2, (k_1, k_2)}(-1) = \widetilde{\text{\rm \pounds}}_{\mathcal{A}_2, (k_2, k_1)}(-1)$ by Proposition \ref{reverse prop} (\ref{reverse2}). Therefore, we obtain the equalities (\ref{57}) and (\ref{58}) by Proposition \ref{FMZV's properties} (\ref{Zhou-Cai}), Lemma \ref{aux} (\ref{m, -1}), and the equality (\ref{A_2, m, -1}). The proof of the equalities (\ref{59}) and (\ref{60}) is similar.
\end{proof}
\begin{proposition}
	Let $m$ be an integer greater than $1$. Let $k_1$ and $k_2$ be integers such that $w:= k_1+k_2$ is odd. Then we have the following equalities$:$
	\begin{align}
	&\text{\rm \pounds}_{\mathcal{A}, \{ 1 \}^m}(2) = \widetilde{\text{\rm \pounds}}_{\mathcal{A}, \{ 1 \}^m}^{\star}(2) = \frac{1-2^{m-1}}{2^{m-2}}\frac{B_{\mathbf{p}-m}}{m},
	\label{easy gen of Sun}\\
	&\widetilde{\text{\rm \pounds}}_{\mathcal{A}, \{ 1 \}^m}(1/2) = \text{\rm \pounds}_{\mathcal{A}, \{ 1 \}^m}^{\star}(1/2) = \frac{2^{m-1}-1}{2^{m-1}}\frac{B_{\mathbf{p}-m}}{m},
	\label{gen of Sun}\\
	&\text{\rm \pounds}_{\mathcal{A}, (\{ 1 \}^{k_1-1}, 2, \{ 1 \}^{k_2-1})}(2) =  \left\{ \frac{2^{w-1}-1}{2^{w-1}}-(-1)^{k_1}\binom{w}{k_1} \right\} \frac{B_{\mathbf{p}-w}}{w},
	\label{k_1, k_2, 2, non-star}\\
	&\widetilde{\text{\rm \pounds}}_{\mathcal{A}, (\{ 1 \}^{k_1-1}, 2, \{ 1 \}^{k_2-1})}^{\star}(2) = \left\{ \frac{1-2^{w-1}}{2^{w-1}}-(-1)^{k_1}\binom{w}{k_1} \right\} \frac{B_{\mathbf{p}-w}}{w}
	\label{k_1, k_2, 2},\\
	&\widetilde{\text{\rm \pounds}}_{\mathcal{A}, (\{ 1 \}^{k_1-1}, 2, \{ 1 \}^{k_2-1})}(1/2) = \frac{1}{2}\left\{ \frac{1-2^{w-1}}{2^{w-1}}-(-1)^{k_1}\binom{w}{k_1} \right\} \frac{B_{\mathbf{p}-w}}{w},\label{66}\\
	&\text{\rm \pounds}_{\mathcal{A}, (\{ 1 \}^{k_1-1}, 2, \{ 1 \}^{k_2-1})}^{\star}(1/2) = \frac{1}{2}\left\{ \frac{2^{w-1}-1}{2^{w-1}}-(-1)^{k_1}\binom{w}{k_1} \right\} \frac{B_{\mathbf{p}-w}}{w}.\label{k_1, k_2, 2, star}
	\end{align}
	\label{special values in A}\end{proposition}
\begin{proof}
	First, we prove the star cases. We use the functional equation (\ref{fn eq 1}) for an index $\Bbbk^{\vee}$:
	\begin{equation}
	\widetilde{\text{\rm \pounds}}_{\mathcal{A}, \Bbbk^{\vee}}^{\star}(t) = \widetilde{\text{\rm \pounds}}_{\mathcal{A}, \Bbbk}^{\star}(1-t) - \zeta_{\mathcal{A}}^{\star}(\Bbbk ).
	\label{dual of dual}\end{equation}
	Consider the case $t=2$ and $\Bbbk = m$ of the equality (\ref{dual of dual}) (or the equality (\ref{1^m, 1})). Then we obtain the star case of the equality (\ref{easy gen of Sun}) by Lemma \ref{aux} (\ref{m, -1}). Considering the case $\Bbbk = (k_1, k_2)$ of  the equality (\ref{dual of dual}), we obtain the equality (\ref{k_1, k_2, 2}) by Lemma \ref{aux} (\ref{k_1, k_2, -1, non-star}) and Proposition \ref{FMZV's properties} (\ref{dep2 for FMZV}). The star case of  the equality (\ref{gen of Sun}) and  the equality (\ref{k_1, k_2, 2, star}) are obtained by Proposition \ref{reverse prop}.
	
	Next, we prove the non-star cases by Corollary \ref{zeta corollary} (\ref{nonstar-star}) for $\mathcal{A}$. By considering the case $t=2$ and $\Bbbk = \{ 1 \}^m$ (i.e. the case $t=2$ of the equality (\ref{MT Fn-Eq})), we have the non-star case of  the equality (\ref{easy gen of Sun}). We consider the case $\Bbbk = (\{ 1 \}^{k_1-1}, 2, \{ 1 \}^{k_2-1})$ satisfying the condition that $k_1+k_2$ is odd. The summation for $j=k_1, \dots , w-2$ of the right hand side of (\ref{nonstar-star}) vanishes since $\zeta_{\mathcal{A}}^{\star}(\{ 1 \}^{w-1-j})=0$. We suppose that $j$ is an element of $\{ 0, \dots , k_1-1 \}$. If $j$ is odd, we have $\zeta_{\mathcal{A}}^{\star}(\{ 1 \}^{k_2-1}, 2, \{ 1 \}^{k_1-j-1}) = 0$ by Theorem \ref{PPT theorem} and if $j$ is even, we have $\text{\rm \pounds}_{\mathcal{A}, \{ 1 \}^j}(2) = 0$ by  the equality (\ref{easy gen of Sun}). Hence, we see that the summation in the right hand side of (\ref{nonstar-star}) vanishes and we have the equality (\ref{k_1, k_2, 2, non-star}). The non-star case of the equality (\ref{gen of Sun}) and the equality (\ref{66}) are obtained by Proposition \ref{reverse prop}.
\end{proof}
\begin{remark}
	Z. W. Sun proved that $\text{\rm \pounds}_{\mathcal{A}, \{ 1 \}^2}^{\star}(1/2) = 0$ (see \cite[Theorem 1.1]{ZWS}). The proof is based on some technical calculations. The case $(k_1, k_2) = (1, 2)$ or $(k_1, k_2)=(2, 1)$ of Proposition \ref{special values in A} (\ref{k_1, k_2, 2}) and (\ref{k_1, k_2, 2, non-star}) have already been obtained  by Me\v{s}trovi\'c \cite[Theorem 1.1, Corollary 1.2]{M} and by Tauraso and J. Zhao \cite[Proposition 7.1]{TZ}.
\end{remark}
\begin{definition}
	Let $n$ be a positive integer and $a$ a non-zero rational number. We define the element $q_{\mathbf{p}}(a)$ of $\mathcal{A}_n$ to be $(q_p(a) \bmod{p^n})_p$ where $q_p(a)$ is the Fermat quotient, that is, 
	\[
	q_p(a) = \frac{a^{p-1}-1}{p}
	\]
	for a prime number $p$. 
	\label{def of FQ}\end{definition}
By Fermat's little theorem, $q_{\mathbf{p}}(a)$ is well-defined as an element of $\mathcal{A}_n$. Under the hypothesis that $abc$-conjecture is true, we see that $q_{\mathbf{p}}(a)$ is non-zero. See \cite{Si}.
\begin{theorem}
	Let $m$ be a positive even number. Then we have the following equalities in $\mathcal{A}_2:$
	\begin{align}
	&\text{\rm \pounds}_{\mathcal{A}_2, \{ 1 \}^m}(2)=-\widetilde{\text{\rm \pounds}}_{\mathcal{A}_2, \{ 1 \}^m}^{\star}(2) = \left( \frac{m+1}{2^m}-m-2 \right) \frac{B_{\mathbf{p}-m-1}}{m+1}\mathbf{p},
	\label{A_2, m, 2}\\
	&\widetilde{\text{\rm \pounds}}_{\mathcal{A}_2, \{ 1 \}^m}(1/2) =-\text{\rm \pounds}_{\mathcal{A}_2, \{ 1 \}^m}^{\star} (1/2) = \frac{1-2^{m+1}}{2^{m+1}}\frac{B_{\mathbf{p}-m-1}}{m+1}\mathbf{p}.
	\label{A_2, m, 1/2}
	\end{align}
	\label{A_2 theorem}\end{theorem}
\begin{proof}
	First, we prove the star cases. By the functional equation Remark \ref{1-var remark} (\ref{gen. of Hoffman's duality}), we have
	\[
	\widetilde{\text{\rm \pounds}}_{\mathcal{A}_2, \{ 1 \}^m}^{\star}(2) - \zeta_{\mathcal{A}_2}^{\star}(\{ 1 \}^m ) = \text{\rm \pounds}_{\mathcal{A}_2, m}(-1) + \bigl(\widetilde{\text{\rm \pounds}}_{\mathcal{A}_2, (1, m)}^{\star}(-1)-\text{\rm \pounds}_{\mathcal{A}_2, m+1}(-1) \bigr)\mathbf{p}.
	\]
	Therefore, by combining Proposition \ref{FMZV's properties} (\ref{Zhou-Cai}), Lemma \ref{aux} (\ref{m, -1}), (\ref{k_1, k_2, -1, non-star}) and Proposition \ref{A_2 aux} (\ref{A_2, m, -1}), we have
	\begin{equation*}
	\widetilde{\text{\rm \pounds}}_{\mathcal{A}_2, \{ 1 \}^m}^{\star} (2) -\frac{B_{\mathbf{p}-m-1}}{m+1}\mathbf{p}= \frac{m(2^m-1)}{2^m}\frac{B_{\mathbf{p}-m-1}}{m+1}\mathbf{p}+ \left(\frac{1-2^m}{2^m}\frac{B_{\mathbf{p}-m-1}}{m+1}-\frac{1-2^m}{2^{m-1}}\frac{B_{\mathbf{p}-m-1}}{m+1} \right)\mathbf{p}
	\end{equation*}
	or
	\begin{equation}
	\widetilde{\text{\rm \pounds}}_{\mathcal{A}_2, \{1\}^m}^{\star}(2)= \left( m+2 -\frac{m+1}{2^m} \right) \frac{B_{\mathbf{p}-m-1}}{m+1}\mathbf{p}.
	\label{firstA2}\end{equation}
	By the equality (\ref{reverse A_2}), we have
	\[
	\text{\rm \pounds}_{\mathcal{A}_2, \{ 1 \}^m}^{\star}(1/2) = \frac{1}{2^{\mathbf{p}}} \left( \widetilde{\text{\rm \pounds}}_{\mathcal{A}_2, \{ 1 \}^m}^{\star}(2) + \mathbf{p}\sum_{i=1}^m \widetilde{\text{\rm \pounds}}_{\mathcal{A}_2, (\{ 1 \}^{i-1}, 2, \{ 1 \}^{m-i})}^{\star}(2) \right) .
	\]
	Hence, by combining the equality (\ref{firstA2}) and Proposition \ref{special values in A} (\ref{k_1, k_2, 2}), we have
	{\small \begin{equation*}
	\begin{split}
	\text{\rm \pounds}_{\mathcal{A}_2, \{ 1 \}^m}^{\star}(1/2) &= \frac{1}{2}\left\{ \left( m+2-\frac{m+1}{2^m} \right) \frac{B_{\mathbf{p}-m-1}}{m+1}\mathbf{p} + \sum_{i=1}^m\left\{ \frac{1-2^m}{2^m}- (-1)^i \binom{m+1}{i} \right\}\frac{B_{\mathbf{p}-m-1}}{m+1}\mathbf{p}\right\} \\
	&= \frac{2^{m+1}-1}{2^{m+1}}\frac{B_{\mathbf{p}-m-1}}{m+1}\mathbf{p},
	\end{split}
	\end{equation*}
}since $m$ is even and $\sum_{i=1}^m (-1)^i\binom{m+1}{i}=0$. Note that the equality $2^{\mathbf{p}}=2(1+q_{\mathbf{p}}(2)\mathbf{p})$ holds in $\mathcal{A}_2$ and $\text{\rm \pounds}_{\mathcal{A}, \{1\}^m}^{\star}(1/2)=0$.
	
	Next, we prove the non-star cases. By Corollary \ref{zeta corollary} (\ref{nonstar-star}) for $\mathcal{A}_2$, we have
	\begin{equation}
	\text{\rm \pounds}_{\mathcal{A}_2, \{1\}^m}(2)=-\widetilde{\text{\rm \pounds}}_{\mathcal{A}_2, \{1\}^m}^{\star}(2)+\sum_{j=1}^{m-1}(-1)^{j-1}\text{\rm \pounds}_{\mathcal{A}_2, \{1\}^j}(2)\zeta_{\mathcal{A}_2}^{\star}(\{1\}^{m-j}).
	\label{with sum}\end{equation}
	Since $\zeta_{\mathcal{A}_2}^{\star}(\{1\}^{m-j})$ is contained in $\mathbf{p}\mathcal{A}_2$, we have
	\[
	\text{\rm \pounds}_{\mathcal{A}_2, \{ 1 \}^j}(2)\zeta_{\mathcal{A}_2}^{\star}(\{1\}^{m-j}) = \text{(a certain rational number)} \times B_{\mathbf{p}-m+j-1}B_{\mathbf{p}-j}\mathbf{p}
	\]
	for any $j=1, \dots , m-1$ by Proposition \ref{FMZV's properties} (\ref{Zhou-Cai}) and Proposition \ref{special values in A} (\ref{easy gen of Sun}). If $j$ is odd, we have $B_{\mathbf{p}-m+j-1}=0$ and if $j$ is even, we have $B_{\mathbf{p}-j}=0$ because $B_{2n+1}=0$ for any positive integer $n$. Therefore, the summation in the right hand side of (\ref{with sum}) vanishes and we have 
	\begin{equation}
	\text{\rm \pounds}_{\mathcal{A}_2, \{ 1 \}^m}(2)=-\widetilde{\text{\rm \pounds}}_{\mathcal{A}_2, \{1\}^m}^{\star}(2)=\left( \frac{m+1}{2^m}-m-2 \right) \frac{B_{\mathbf{p}-m-1}}{m+1}\mathbf{p}.
	\label{secondA2}\end{equation}
	by Proposition \ref{special values in A} (\ref{easy gen of Sun}) and (\ref{k_1, k_2, 2, non-star}), Proposition \ref{FMZV's properties} (\ref{Zhou-Cai}), and Proposition \ref{A_2 aux} (\ref{A_2, m, -1}). By the equality (\ref{reverse A_2}), the equality (\ref{secondA2}), and Proposition \ref{special values in A} (\ref{k_1, k_2, 2, non-star}), we also have
	{\small \begin{equation*}
	\begin{split}
	\widetilde{\text{\rm \pounds}}_{\mathcal{A}_2, \{ 1 \}^m} (1/2) &= \frac{1}{2^{\mathbf{p}}}\bigl( \text{\rm \pounds}_{\mathcal{A}_2, \{ 1 \}^m}(2) + \mathbf{p}\sum_{i=1}^m\text{\rm \pounds}_{\mathcal{A}_2, (\{ 1 \}^{i-1}, 2, \{ 1 \}^{m-i})}(2) \bigr) \\
	&= \frac{1}{2} \left\{ \left( \frac{m+1}{2^m}-m-2 \right) \frac{B_{\mathbf{p}-m-1}}{m+1}\mathbf{p}+\sum_{i=1}^m\left\{ \frac{2^m-1}{2^m} -(-1)^i\binom{m+1}{i}\right\} \frac{B_{\mathbf{p}-m-1}}{m+1}\mathbf{p}\right\} \\
	&= \frac{1-2^{m+1}}{2^{m+1}}\frac{B_{p-m-1}}{m+1}\mathbf{p}.
	\end{split}
	\end{equation*}}
\end{proof}
\begin{remark}
	The cases $m=2$ of Theorem \ref{A_2 theorem} have already been given by Z. W. Sun and L. L. Zhao \cite{SZ}, Me\v{s}trovi\'c \cite{M}, and Tauraso and J. Zhao \cite{TZ}. Indeed, the non-star case of the equality (\ref{A_2, m, 2}) is \cite[Proposition 7.1 (78)]{TZ} and the star case of the equality (\ref{A_2, m, 2}) which is equivalent to Proposition \ref{Fermat quotient} (\ref{A_2, 1^2, 2}) below is \cite[Theorem 1.1 (1)]{M} or \cite[Proposition 7.1(77)]{TZ}. The star case of Theorem \ref{A_2 theorem} (\ref{A_2, m, 1/2}) was conjectured by Z. W. Sun \cite[Conjecture 1.1]{Z} and proved by Z. W. Sun and L. L. Zhao \cite{SZ}. Me\v{s}trovi\'c gave another proof of Sun's conjecture in \cite{M} and our proof of the equality (\ref{A_2, m, 1/2}) is similar to his proof.
\end{remark}
Now, we recall the following results for the finite polylogarithms obtained by Z. H. Sun \cite{ZHS1, ZHS2}, Dilcher and Skula \cite{DS}, and Me\v{s}trovi\'c \cite{M2}:
\begin{lemma}
	The following equalities hold$:$
	\begin{align}
	&\text{\rm \pounds}_{\mathcal{A}_3, 1}(-1) = -2q_{\mathbf{p}}(2)+q_{\mathbf{p}}(2)^2\mathbf{p}-\left( \frac{2}{3}q_{\mathbf{p}}(2)^3+\frac{1}{4}B_{\mathbf{p}-3}\right) \mathbf{p}^2,
	\label{FP1}\\
	&\text{\rm \pounds}_{\mathcal{A}_3, 1}(2) = -2q_{\mathbf{p}}(2)-\frac{7}{12}B_{\mathbf{p}-3}\mathbf{p}^2,
	\label{A_3, 2}\\
	&\text{\rm \pounds}_{\mathcal{A}_2, 2}(2) = -q_{\mathbf{p}}(2)^2+\left( \frac{2}{3}q_{\mathbf{p}}(2)^3+\frac{7}{6}B_{\mathbf{p}-3} \right) \mathbf{p},
	\label{A_2, 2, 2}\\
	&\text{\rm \pounds}_{\mathcal{A}, 3}(2) = -\frac{1}{3}q_{\mathbf{p}}(2)^3-\frac{7}{24}B_{\mathbf{p}-3}, 
	\label{A, 3, 2}\\
	&\text{\rm \pounds}_{\mathcal{A}_3, 1}(1/2) = q_{\mathbf{p}}(2)-\frac{1}{2}q_{\mathbf{p}}(2)^2\mathbf{p}+\left( \frac{1}{3}q_{\mathbf{p}}(2)^3-\frac{7}{48}B_{\mathbf{p}-3} \right) \mathbf{p}^2,
	\label{A_2, 1, 1/2}\\
	&\text{\rm \pounds}_{\mathcal{A}_2, 2}(1/2) = -\frac{1}{2}q_{\mathbf{p}}(2)^2+\left( \frac{1}{2}q_{\mathbf{p}}(2)^3+\frac{7}{
		24}B_{\mathbf{p}-3} \right) \mathbf{p},
	\label{A, 2, 1/2}\\
	&\text{\rm \pounds}_{\mathcal{A}, 3}(1/2) = \frac{1}{6}q_{\mathbf{p}}(2)^3+\frac{7}{48}B_{\mathbf{p}-3}.
	\label{A, 3, 1/2}
	\end{align}
	\label{FP values}\end{lemma}
\begin{proof}
	The equality (\ref{FP1}) is obtained by \cite[Theorem 5.2 (c)]{ZHS1} and the equality (\ref{trivial formula}). The equalities (\ref{A_3, 2}) and (\ref{A_2, 2, 2}) are \cite[Theorem 4.1 (i)]{ZHS2} and \cite[Theorem 4.1 (ii)]{ZHS2}, respectively. The equalities (\ref{A, 3, 2}), (\ref{A_2, 1, 1/2}), and (\ref{A, 2, 1/2}) are essentially due to Dilcher and Skula \cite{DS} (see \cite[Remark 4.1]{ZHS2}). The equality (\ref{A_2, 1, 1/2}) is also shown by Me\v{s}trovi\'c \cite{M2}. The equality (\ref{A, 3, 1/2}) is obtained by the equality (\ref{A_2, 2, 2}) and Proposition \ref{reverse prop}.
\end{proof}
We obtain the following special values for F(S)MP by the above lemma:
\begin{proposition}
	Let $\bullet \in \{ \emptyset , \star \}$. Then the following equalities hold$:$
	\begin{align}
	&\text{\rm \pounds}_{\mathcal{A}_2, \{ 1 \}^2}(-1) = -\widetilde{\text{\rm \pounds}}_{\mathcal{A}_2, \{ 1 \}^2}^{\star}(-1)=q_{\mathbf{p}}(2)^2-\left( q_{\mathbf{p}}(2)^3+\frac{13}{24}B_{\mathbf{p}-3} \right) \mathbf{p},
	\label{A_2, 1^2, -1}\\
	&\text{\rm \pounds}_{\mathcal{A}_2, \{ 1 \}^2}^{\star}(-1) = -\widetilde{\text{\rm \pounds}}_{\mathcal{A}_2, \{ 1 \}^2}(-1) = q_{\mathbf{p}}(2)^2-\left( q_{\mathbf{p}}(2)^3+\frac{1}{24}B_{\mathbf{p}-3} \right) \mathbf{p},
	\label{}\\
	&\text{\rm \pounds}_{\mathcal{A}_2, \{ 1\}^2}^{\star}(2) = -\widetilde{\text{\rm \pounds}}_{\mathcal{A}_2, \{ 1 \}^2}(2) = -q_{\mathbf{p}}(2)^2+\left( \frac{2}{3}q_{\mathbf{p}}(2)^3+\frac{1}{12}B_{\mathbf{p}-3} \right) \mathbf{p},
	\label{A_2, 1^2, 2}\\
	&\text{\rm \pounds}_{\mathcal{A}_2, \{ 1 \}^2} (1/2) = -\widetilde{\text{\rm \pounds}}_{\mathcal{A}_2, \{ 1 \}^2}^{\star} (1/2) = \frac{1}{2}q_{\mathbf{p}}(2)^2-\frac{1}{2}q_{\mathbf{p}}(2)^3\mathbf{p},
	\label{A, 1^2, 1/2}\\
	&\text{\rm \pounds}_{\mathcal{A}, (1, 2)}^{\star}(2) = -\widetilde{\text{\rm \pounds}}_{\mathcal{A}, (2, 1)}(2) = -\frac{1}{3}q_{\mathbf{p}}(2)^3-\frac{25}{24}B_{\mathbf{p}-3},
	\label{A, (1, 2), 2}\\
	&\text{\rm \pounds}_{\mathcal{A}, (1, 2)}(1/2) = -\widetilde{\text{\rm \pounds}}_{\mathcal{A}, (2, 1)}^{\star} (1/2) = -\frac{1}{6}q_{\mathbf{p}}(2)^3-\frac{25}{48}B_{\mathbf{p}-3},\\
	&\text{\rm \pounds}_{\mathcal{A}, (2, 1)}^{\star}(2) = - \widetilde{\text{\rm \pounds}}_{\mathcal{A}, (1, 2)}(2) = -\frac{1}{3}q_{\mathbf{p}}(2)^3+\frac{23}{24}B_{\mathbf{p}-3},
	\label{A, (2, 1), 2}
		\end{align}
		\begin{align}
	&\text{\rm \pounds}_{\mathcal{A}, (2, 1)} (1/2) = -\widetilde{\text{\rm \pounds}}_{\mathcal{A}, (1, 2)}^{\star} (1/2) = -\frac{1}{6}q_{\mathbf{p}}(2)^3+\frac{23}{48}B_{\mathbf{p}-3},\\
	&\text{\rm \pounds}_{\mathcal{A}, \{ 1 \}^3}^{\bullet}(-1) = \widetilde{\text{\rm \pounds}}_{\mathcal{A}, \{ 1\}^3}^{\bullet}(-1) = \text{\rm \pounds}_{\mathcal{A}, \{ 1 \}^3}^{\star}(2) = \widetilde{\text{\rm \pounds}}_{\mathcal{A}, \{ 1 \}^3}(2) =-\frac{1}{3}q_{\mathbf{p}}(2)^3-\frac{7}{24}B_{\mathbf{p}-3},
	\label{A, 1^3, -1}\\
	&\text{\rm \pounds}_{\mathcal{A}, \{ 1 \}^3}(1/2) = \widetilde{\text{\rm \pounds}}_{\mathcal{A}, \{ 1 \}^3}^{\star} (1/2) = \frac{1}{6}q_{\mathbf{p}}(2)^3+\frac{7}{48}B_{\mathbf{p}-3}.
	\label{A, 1^3, 1/2}
	\end{align}
	\label{Fermat quotient}\end{proposition}
\begin{proof}
	We can calculate $\widetilde{\text{\rm \pounds}}_{\mathcal{A}_2, \{1\}^2}^{\star}(-1)$, $\text{\rm \pounds}_{\mathcal{A}_2, \{1\}^2}^{\star}(2)$, $\text{\rm \pounds}_{\mathcal{A}, (1, 2)}^{\star}(2)$, $\text{\rm \pounds}_{\mathcal{A}, (2, 1)}^{\star}(2)$, $\widetilde{\text{\rm \pounds}}_{\mathcal{A}, \{1\}^3}^{\star}(-1)$, and\\
	$\text{\rm \pounds}_{\mathcal{A}, \{1\}^3}^{\star}(2)$ by the equalities 
	\begin{align*}
	&\widetilde{\text{\rm \pounds}}_{\mathcal{A}_2, \{ 1 \}^2}^{\star}(-1)= \zeta_{\mathcal{A}_2}^{\star}(\{ 1 \}^2) + \text{\rm \pounds}_{\mathcal{A}_2, 2}(2)+\bigl(\widetilde{\text{\rm \pounds}}_{\mathcal{A}_2, (1, 2)}^{\star}(2)-\text{\rm \pounds}_{\mathcal{A}_2, 3}(2)\bigr)\mathbf{p},\\
	&\text{\rm \pounds}_{\mathcal{A}_2, \{1\}^2}^{\star}(2) = \text{\rm \pounds}_{\mathcal{A}_2, \{1\}^2}(2)+\text{\rm \pounds}_{\mathcal{A}_2, 2}(2), \ \text{\rm \pounds}_{\mathcal{A}, (1, 2)}^{\star}(2)=\text{\rm \pounds}_{\mathcal{A}, (1, 2)}(2)+\text{\rm \pounds}_{\mathcal{A}, 3}(2),\\
	&\text{\rm \pounds}_{\mathcal{A}, (2, 1)}^{\star}(2) = \text{\rm \pounds}_{\mathcal{A}, (2, 1)}(2)+\text{\rm \pounds}_{\mathcal{A}, 3}(2), \ \widetilde{\text{\rm \pounds}}_{\mathcal{A}, \{1\}^3}^{\star}(-1) = \text{\rm \pounds}_{\mathcal{A}, 3}(2), \ \text{\rm \pounds}_{\mathcal{A}, \{1\}^3}^{\star}(2)=\text{\rm \pounds}_{\mathcal{A}, \{1\}^3}^{\star}(-1),
	\end{align*}
	respectively. Here, we have used Remark \ref{1-var remark} (\ref{gen. of Hoffman's duality}), Lemma \ref{multi-FPL} (\ref{1^m, 1}), and Corollary \ref{ZS lemma} (\ref{ZS star}). All other values obtained by Proposition \ref{reverse prop} and Corollary \ref{zeta corollary} (\ref{nonstar-star}) and (\ref{nonstar and star2}).
\end{proof}
\begin{remark}
	Note that all of the values that appear in the above proposition essentially have been given by Me\v{s}trovi\'{c} \cite[Theorem 1.1]{M} and Tauraso and J. Zhao \cite[Proposition 7.1]{TZ}. We have determined all values of the form $\overline{\text{\rm \pounds}}_{\mathcal{A}_n, \Bbbk}^{\bullet}(r)$ for $- \in \{ \emptyset , \sim \}$, $\bullet \in \{ \emptyset , \star \}$, and $r \in \{ -1, 2^{\pm 1} \}$ when $n+\wt(\Bbbk) \leq 4$ by Lemma \ref{aux}, Proposition \ref{A_2 aux}, Proposition \ref{special values in A}, Theorem \ref{A_2 theorem}, Lemma \ref{FP values}, and Proposition \ref{Fermat quotient}.
	\label{all determine}\end{remark}
Furthermore, we have the following some special values of FMPs of weight $4$.
\begin{proposition}
	Let $\bullet \in \{ \emptyset, \star \}$. Then the following equalities hold$:$
	\begin{align}
	&\text{\rm \pounds}_{\mathcal{A}, (1, 3)}^{\bullet}(-1) = -\widetilde{\text{\rm \pounds}}_{\mathcal{A}, (3, 1)}^{\bullet}(-1) = \frac{1}{2}q_{\mathbf{p}}(2)B_{\mathbf{p}-3},\label{91}\\
	&\text{\rm \pounds}_{\mathcal{A}, (2, 1, 1)}(2) = \widetilde{\text{\rm \pounds}}_{\mathcal{A}, (1, 1, 2)}^{\star}(2)= -\frac{1}{2}q_{\mathbf{p}}(2)B_{\mathbf{p}-3},\\
	&\text{\rm \pounds}_{\mathcal{A}, (2, 1, 1)}^{\star}(1/2) = \widetilde{\text{\rm \pounds}}_{\mathcal{A}, (1, 1, 2)}(1/2) = -\frac{1}{4}q_{\mathbf{p}}(2)B_{\mathbf{p}-3}.
	\label{four}\end{align}
	\label{four prop}\end{proposition}
\begin{proof}
	By \cite[Proposition 6.1 (55)]{TZ}, we have $\widetilde{\text{\rm \pounds}}_{\mathcal{A}, (3, 1)}(-1)=-\frac{1}{2}q_{\mathbf{p}}(2)B_{\mathbf{p}-3}$. All of the other values are obtained by the functional equations and Proposition \ref{reverse prop}.
\end{proof}
Next, we calculate some special values of FH(S)MPs. The following two lemmas are due to Chamberland and Dilcher \cite{CD}, Tauraso and J. Zhao \cite{TZ}, and  Kh.\ Hessami Pilehrood, T. Hessami Pilehrood, and Tauraso \cite{PPT}:
\begin{lemma}
	Let $m$, $k_1$, $k_2$, and $k_3$ be positive integers, $w=k_1+k_2+k_3$, and $\bullet \in \{ \emptyset , \star \}$. Then the following equalities hold$:$
	\begin{equation}
	\text{\rm \pounds}_{\mathcal{A}, (k_1, k_2)}^{\ast, \bullet}(-1, -1) = (-1)^{k_1}\frac{1-2^{k_1+k_2-1}}{2^{k_1+k_2-1}}\binom{k_1+k_2}{k_1}\frac{B_{\mathbf{p}-k_1-k_2}}{k_1+k_2}
	\label{odd}\end{equation}
	if $k_1+k_2$ is odd,
	\begin{equation}
	\text{\rm \pounds}_{\mathcal{A}, (k_1, k_2)}^{\ast , \bullet}(-1, -1) = \frac{(2^{k_1-1}-1)(2^{k_2-1}-1)}{2^{k_1+k_2-3}k_1k_2}B_{\mathbf{p}-k_1}B_{\mathbf{p}-k_2}\hspace{8mm}
	\label{even bullet}\end{equation}
	if $k_1+k_2$ is even and $k_1, k_2 \geq 2$,
	\begin{equation}
	\text{\rm \pounds}_{\mathcal{A}, (m, 1)}^{\ast, \bullet}(-1, -1) = \text{\rm \pounds}_{\mathcal{A}, (1, m)}^{\ast, \bullet}(-1, -1) = \frac{2^{m-1}-1}{2^{m-2}m}q_{\mathbf{p}}(2)B_{\mathbf{p}-m}
	\label{B to q}\end{equation}
	if $m$ is odd and $m \geq 3$,
	{\footnotesize \begin{equation}
	\text{\rm \pounds}_{\mathcal{A}, (k_1, k_2, k_3)}^{\ast}(-1, -1, 1) = -\text{\rm \pounds}_{\mathcal{A}, (k_1, k_2, k_3)}^{\ast, \star}(-1, -1, 1) = \frac{1}{2}\left\{ (-1)^{k_3}\binom{w}{k_3}-\frac{1-2^{w-1}}{2^{w-1}}\binom{w}{k_1} \right\} \frac{B_{\mathbf{p}-w}}{w}
	\label{-1, -1, 1}\end{equation}
}if $k_1$ is even and $k_2+k_3$ odd,
	{\footnotesize \begin{equation}
	-\text{\rm \pounds}_{\mathcal{A}, (k_1, k_2, k_3)}^{\ast}(1, -1, -1) = \text{\rm \pounds}_{\mathcal{A}, (k_1, k_2, k_3)}^{\ast, \star}(1, -1, -1) = \frac{1}{2}\left\{ (-1)^{k_1}\binom{w}{k_1}-\frac{1-2^{w-1}}{2^{w-1}}\binom{w}{k_3} \right\} \frac{B_{\mathbf{p}-w}}{w}
	\label{1, -1, -1}\end{equation}
}if $k_1+k_2$ is odd and $k_3$ even,
	{\footnotesize \begin{equation}
	\text{\rm \pounds}_{\mathcal{A}, (k_1, k_2, k_3)}^{\ast}(-1, 1, -1) = - \text{\rm \pounds}_{\mathcal{A}, (k_1, k_2, k_3)}^{\ast, \star}(-1, 1, -1) = \frac{1-2^{w-1}}{2^{w}}\left\{ \binom{w}{k_3}-\binom{w}{k_1}\right\} \frac{B_{\mathbf{p}-w}}{w}\hspace{13mm}
	\label{-1, 1, -1}\end{equation}
}if $k_1$ is even, $k_2$ odd, and $k_3$ even,
	\begin{align}
	&\text{\rm \pounds}_{\mathcal{A}, \{ 1 \}^3}^{\ast}(1, -1, -1) = -\text{\rm \pounds}_{\mathcal{A}, \{ 1 \}^3}^{\ast}(-1, -1, 1) = q_{\mathbf{p}}(2)^3+\frac{7}{8}B_{\mathbf{p}-3},
	\label{1, -1, -1 and -1, -1, 1}\\
	&\text{\rm \pounds}_{\mathcal{A}, \{ 1 \}^3}^{\ast, \star}(1, -1, -1) = -\text{\rm \pounds}_{\mathcal{A}, \{ 1 \}^3}^{\ast, \star}(-1, -1, 1) = q_{\mathbf{p}}(2)^3-\frac{7}{8}B_{\mathbf{p}-3},
	\label{1, -1, -1 and -1, -1, 1 star}\\
	&\text{\rm \pounds}_{\mathcal{A}, \{ 1 \}^3}^{\ast, \bullet}(-1, 1, -1) = 0,
	\label{-1, 1, -1=0}\\
	&\text{\rm \pounds}_{\mathcal{A}, \{ 1 \}^3}^{\ast, \bullet}(-1, -1, -1) = -2\text{\rm \pounds}_{\mathcal{A}, \{ 1 \}^3}^{\ast, \bullet}(1, -1, 1) = -\frac{4}{3}q_{\mathbf{p}}(2)^3-\frac{1}{6}B_{\mathbf{p}-3}.
	\label{1, -1, 1}\end{align}
	\label{harmonic alt}\end{lemma}
\begin{proof}
	The non-star case of the equality (\ref{odd}) is \cite[Theorem 3.1 (15)]{TZ}.
	The equality (\ref{even bullet}) and (\ref{B to q}) are \cite[Theorem 3.1 (20)]{TZ}.
	Now, suppose that $w$ is odd. By \cite[Theorem 4.1]{TZ}, we have
	{\footnotesize \begin{equation}
	2\text{\rm \pounds}_{\mathcal{A}, (k_1, k_2, k_3)}^{\ast}(-1, -1, 1)= \zeta_{\mathcal{A}}(k_3, k_1+k_2)
	+\text{\rm \pounds}_{\mathcal{A}, (k_2+k_3, k_1)}^{\ast}(-1, -1)-\text{\rm \pounds}_{\mathcal{A}, k_1}(-1)\widetilde{\text{\rm \pounds}}_{\mathcal{A}, (k_3, k_2)}(-1)
	\label{dep3rel1},
	\end{equation}
}\begin{equation}
	\begin{split}
	2\text{\rm \pounds}_{\mathcal{A}, (k_1, k_2, k_3)}^{\ast}(-1, 1, -1) &= -\text{\rm \pounds}_{\mathcal{A}, k_1}(-1)\text{\rm \pounds}_{\mathcal{A}, (k_3, k_2)}(-1)-\widetilde{\text{\rm \pounds}}_{\mathcal{A}, (k_2, k_1)}(-1)\text{\rm \pounds}_{\mathcal{A}, k_3}(-1)\\
	& \hspace{5mm}+\text{\rm \pounds}_{\mathcal{A}, (k_3, k_1+k_2)}^{\ast}(-1, -1)+\text{\rm \pounds}_{\mathcal{A}, (k_2+k_3, k_1)}^{\ast}(-1, -1).
	\end{split}
	\label{dep3rel2}\end{equation}
	If $k_1$ is even, then we have $\text{\rm \pounds}_{\mathcal{A}, k_1}(-1)=0$ by Lemma \ref{aux} (\ref{m, -1}). Therefore, by the equality (\ref{dep3rel1}), we can calculate $\text{\rm \pounds}_{\mathcal{A}, (k_1, k_2, k_3)}^{\ast}(-1, -1, 1)$, Proposition \ref{FMZV's properties} (\ref{dep2 for FMZV}), and the equality (\ref{odd}).
	If $k_1$ and $k_3$ are even, then we have $\text{\rm \pounds}_{\mathcal{A}, k_1}(-1)=\text{\rm \pounds}_{\mathcal{A}, k_3}(-1)=0$ by Lemma \ref{aux} (\ref{m, -1}). Therefore, we can calculate $\text{\rm \pounds}_{\mathcal{A}, (k_1, k_2, k_3)}^{\ast}(-1, 1, -1)$ by the equalities (\ref{dep3rel2}) and (\ref{odd}). 
	The non-star case of the equality (\ref{1, -1, -1}) is obtained by Proposition \ref{reverse prop} and the equality (\ref{-1, -1, 1}) and the non-star cases of the equalities (\ref{1, -1, -1 and -1, -1, 1}), (\ref{-1, 1, -1=0}), and (\ref{1, -1, 1}) are obtained by \cite[Proposition 7.6]{TZ}. 
	All star cases are obtained by Theorem \ref{MTC} (\ref{A_n formula}).
	Note that \cite[Theorem 3.1 (16)]{TZ} which is the corresponding formula to the star case of the equality (\ref{odd}) is incorrect.
\end{proof}
\begin{lemma}
	Let $k_1$ and $k_2$ be positive even integers. Then we have
	{\small \begin{align}
	&\text{\rm \pounds}_{\mathcal{A}_2, (k_1, k_2)}^{\ast}(-1, -1) = \left\{ \frac{(k_2-k_1)(2^{k_1+k_2}-1)}{2^{k_1+k_2+1}(k_1+k_2+2)}\binom{k_1+k_2+2}{k_1+1}-\frac{k_1+k_2}{2}\right\}\frac{B_{\mathbf{p}-k_1-k_2-1}}{k_1+k_2+1}\mathbf{p},
	\label{new Zhao}\\
	&\text{\rm \pounds}_{\mathcal{A}_2, (k_1, k_2)}^{\ast, \star}(-1, -1) = \left\{ \frac{(k_2-k_1)(2^{k_1+k_2}-1)}{2^{k_1+k_2+1}(k_1+k_2+2)}\binom{k_1+k_2+2}{k_1+1}+\frac{k_1+k_2}{2}\right\}\frac{B_{\mathbf{p}-k_1-k_2-1}}{k_1+k_2+1}\mathbf{p},
	\label{new Zhao star}
	\end{align}
}\begin{align}
	&\text{\rm \pounds}_{\mathcal{A}_2, \{ 1 \}^2}^{\ast}(-1, -1) = 2q_{\mathbf{p}}(2)^2-\left( 2q_{\mathbf{p}}(2)^3+\frac{1}{3}B_{\mathbf{p}-3}\right) \mathbf{p},
	\label{k_1=k_2=1}\\
	&\text{\rm \pounds}_{\mathcal{A}_2, \{ 1 \}^2}^{\ast, \star}(-1, -1) = 2q_{\mathbf{p}}(2)^2-\left( 2q_{\mathbf{p}}(2)^3-\frac{1}{3}B_{\mathbf{p}-3}\right) \mathbf{p},\label{k_1=k_2=1, star}\\
	&\text{\rm \pounds}_{\mathcal{A}_2, \{1\}^3}^{\ast}(-1, -1, -1) = -\frac{4}{3}q_{\mathbf{p}}(2)^3+\widehat{B}_{\mathbf{p}-3}-\frac{1}{2}\widehat{B}_{2\mathbf{p}-4}+2\left( q_{\mathbf{p}}(2)^4-q_{\mathbf{p}}(2)\widehat{B}_{\mathbf{p}-3} \right)\mathbf{p},\label{109}\\
	&\text{\rm \pounds}_{\mathcal{A}_2, \{1\}^3}^{\ast, \star}(-1, -1, -1) = -\frac{4}{3}q_{\mathbf{p}}(2)^3+\widehat{B}_{\mathbf{p}-3}-\frac{1}{2}\widehat{B}_{2\mathbf{p}-4}+2\left( q_{\mathbf{p}}(2)^4+q_{\mathbf{p}}(2)\widehat{B}_{\mathbf{p}-3} \right)\mathbf{p}\label{1102}.
	\end{align}
	\label{A_2 alt}\end{lemma}
\begin{proof}
	The equalities (\ref{new Zhao}), (\ref{k_1=k_2=1}), and (\ref{109}) are \cite[Lemma 3.1]{PPT}, \cite[Proposition 7.3 (100)]{TZ}, and \cite[Proposition 7.6 (117)]{TZ}, respectively. The equalities (\ref{new Zhao star}), (\ref{k_1=k_2=1, star}), and (\ref{1102}) are obtained by the relations
	\begin{align*}
	&\text{\rm \pounds}_{\mathcal{A}_2, (k_1, k_2)}^{\ast, \star}(-1, -1) = \text{\rm \pounds}_{\mathcal{A}_2, (k_1, k_2)}^{\ast}(-1, -1) + \zeta_{\mathcal{A}_2}(k_1+k_2),\\
	&\text{\rm \pounds}_{\mathcal{A}_2, \{ 1 \}^2}^{\ast, \star}(-1, -1) = \text{\rm \pounds}_{\mathcal{A}_2, \{ 1 \}^2}^{\ast}(-1, -1) + \zeta_{\mathcal{A}_2}(2),
	\end{align*}
and	
\[
	\text{\rm \pounds}_{\mathcal{A}_2, \{1\}^3}^{\ast, \star}(-1, -1, -1) = \text{\rm \pounds}_{\mathcal{A}_2, \{1\}^3}^{\ast}(-1, -1, -1)+\widetilde{\text{\rm \pounds}}_{\mathcal{A}_2, (2,1)}(-1)+\text{\rm \pounds}_{\mathcal{A}_2, (1,2)}(-1)+\text{\rm \pounds}_{\mathcal{A}_2, 3}(-1),
	\]
	respectively. Here, note that 
	\[
	\widetilde{\text{\rm \pounds}}_{\mathcal{A}_2, (2,1)}(-1)+\text{\rm \pounds}_{\mathcal{A}_2, (1,2)}(-1) = -\frac{3}{2}\left( 2\widehat{B}_{\mathbf{p}-3}-\widehat{B}_{2\mathbf{p}-4} \right) - \frac{4}{3}q_{\mathbf{p}}(2)B_{\mathbf{p}-3}\mathbf{p}
	\]
	holds by \cite[Proposition 7.3 (105) and (106)]{TZ}.
\end{proof}
\begin{theorem}
	Let $m, k_1, k_2$, and $k_3$ be positive integers and $\bullet \in \{ \emptyset, \star\}$. Let $w=k_1+k_2$ and $w'=k_1+k_2+k_3$. Then we have the following equalities$:$
	\begin{equation}
	\text{\rm \pounds}_{\mathcal{A}, \{1\}^w}^{\ast}(\{1\}^{k_1-1}, 1/2, 2, \{1\}^{k_2-1})=(-1)^{k_1}\frac{1-2^{w-1}}{2^{w-1}}\binom{w}{k_1}\frac{B_{\mathbf{p}-w}}{w}
	\label{MV1'}\end{equation}
	if $w$ is odd,
	\begin{equation}
	\text{\rm \pounds}_{\mathcal{A}, \{ 1 \}^{w}}^{\ast , \star}(\{ 1 \}^{k_1-1}, 2, 1/2, \{ 1 \}^{k_2-1}) 
	=
	(-1)^{k_1}\frac{2^{w-1}-1}{2^{w-1}}\binom{w}{k_1}\frac{B_{\mathbf{p}-w}}{w}
	\label{MV1}\end{equation}
	if $w$ is odd,
	\begin{equation}
	\text{\rm \pounds}_{\mathcal{A}, \{1\}^w}^{\ast}(\{1\}^{k_1-1}, 1/2, 2, \{1\}^{k_2-1}) = 0
	\label{MV2'}\end{equation}
	if $w$ is even,
	\begin{equation}
	\text{\rm \pounds}_{\mathcal{A}, \{ 1 \}^{w}}^{\ast , \star}(\{ 1 \}^{k_1-1}, 2, 1/2, \{ 1 \}^{k_2-1}) 
	=
	-\frac{(2^{k_1-1}-1)(2^{k_2-1}-1)}{2^{w-3}k_1k_2}B_{\mathbf{p}-k_1}B_{\mathbf{p}-k_2}
	\label{MV1.5}\end{equation}
	if $w$ is even and $k_1, k_2 \geq 2$,
	\begin{equation}
	\text{\rm \pounds}_{\mathcal{A}, \{1\}^{m+1}}^{\ast, \star}(\{1\}^{m-1}, 2, 1/2)=\text{\rm \pounds}_{\mathcal{A}, \{1\}^{m+1}}^{\ast, \star}(2, 1/2, \{1\}^{m-1})=\frac{1-2^{m-1}}{2^{m-2}m}q_{\mathbf{p}}(2)B_{\mathbf{p}-m}
	\label{B to q 2}\end{equation}
	if $m$ is odd and $m \geq 3$,
	\begin{equation}
	\text{\rm \pounds}_{\mathcal{A}, \{1\}^w}^{\ast}(2, \{1\}^{k_1-2}, 1/2, 2, \{1\}^{k_2-1})=\frac{2^{w-1}-1}{2^{w-1}}\left\{\binom{w}{k_1}-1 \right\}\frac{B_{\mathbf{p}-w}}{w}
	\label{MV2.5}\end{equation}
	if $k_1\geq 3$ is odd and $k_2$ is even,
	\begin{equation}
	\text{\rm \pounds}_{\mathcal{A}, \{1\}^m}^{\ast}(2, 1/2, 2, \{1\}^{m-3})=\frac{(1-2^{m-1})(m^2-m+2)}{2^m}\frac{B_{\mathbf{p}-m}}{m}
	\label{MV2.7}\end{equation}
	if $m\geq 3$ is odd, 
	\begin{equation}
	\text{\rm \pounds}_{\mathcal{A}, \{ 1 \}^{w}}^{\ast , \star}( \{ 1 \}^{k_1-1}, 2, 1/2, \{ 1 \}^{k_2-2}, 2) = \frac{1-2^{w-1}}{2^{w-1}}\left\{ 1+(-1)^{k_1-1}\binom{w}{k_1}\right\} \frac{B_{\mathbf{p}-w}}{w}
	\label{MV2}\end{equation}
	if $w\geq 3$ is odd and $k_2 \geq 2$,
	\begin{equation}
	\text{\rm \pounds}_{\mathcal{A}, \{1\}^w}^{\ast}(\{1\}^{k_1-1}, 1/2, 2, \{1\}^{k_2-2}, 1/2)=\frac{1-2^{w-1}}{2^w}\left\{ \binom{w}{k_1}-1\right\} \frac{B_{\mathbf{p}-w}}{w}
	\label{reverse mult1}\end{equation}
	if $k_1$ is even and $k_2\geq 3$ is odd,
	\begin{equation}
	\text{\rm \pounds}_{\mathcal{A}, \{1\}^m}^{\ast}(\{1\}^{m-3}, 1/2, 2, 1/2) = \frac{(2^{m-1}-1)(m^2-m+2)}{2^{m+1}}\frac{B_{\mathbf{p}-m}}{m}
	\label{reverse mult2}\end{equation}
	if $m$ is odd and $m \geq 3$,
	\begin{equation}
	\text{\rm \pounds}_{\mathcal{A}, \{ 1 \}^{w}}^{\ast , \star}(1/2, \{ 1 \}^{k_1-2}, 2, 1/2, \{ 1 \}^{k_2-1})= \frac{2^{w-1}-1}{2^{w}}\left\{ 1+(-1)^{k_1}\binom{w}{k_1}\right\} \frac{B_{\mathbf{p}-w}}{w}
	\label{MV3}\end{equation}
	if $w\geq 3$ is odd and $k_1 \geq 2$,
	\begin{equation}
	\text{\rm \pounds}_{\mathcal{A}, \{1\}^m}^{\ast}(1, 2, \{1\}^{m-2})=-2\text{\rm \pounds}_{\mathcal{A}, \{1\}^m}^{\ast}(\{1\}^{m-2}, 1/2, 1) = \frac{(2^{m-1}-1)(m-1)}{2^{m-1}}\frac{B_{\mathbf{p}-m}}{m}
	\label{MV4'}\end{equation}
	if $m\geq 3$ is odd,
	\begin{equation}
	\text{\rm \pounds}_{\mathcal{A}, \{ 1 \}^m}^{\ast , \star}(\{ 1 \}^{m-2}, 2, 1) = -2\text{\rm \pounds}_{\mathcal{A}, \{ 1 \}^m}^{\ast , \star}( 1, 1/2, \{ 1 \}^{m-2}) =\frac{(2^{m-1}-1)(m-1)}{2^{m-1}}\frac{B_{\mathbf{p}-m}}{m}
	\label{MV4}\end{equation}
	if $m\geq 3$ is odd,
	{\scriptsize \begin{equation}
		\text{\rm \pounds}_{\mathcal{A}, (\{ 1 \}^{k_1+k_2-1}, 2, \{ 1 \}^{k_3-1})}^{\ast}(\{ 1 \}^{k_1-1}, 1/2, 2, \{ 1 \}^{k_2+k_3-2})
		=\frac{1}{2}\left\{ (-1)^{k_3}\binom{w'}{k_3}-\frac{1-2^{w'-1}}{2^{w'-1}}\binom{w'}{k_1}\right\} \frac{B_{\mathbf{p}-w'}}{w'}
		\label{MV5'''}\end{equation}
	}if $k_1$ is even and $k_2+k_3$ odd,
	{\scriptsize \begin{equation}
		\text{\rm \pounds}_{\mathcal{A}, (\{ 1 \}^{k_1+k_2-1}, 2, \{ 1 \}^{k_3-1})}^{\ast, \star}(\{ 1 \}^{k_1-1}, 2, 1/2, \{ 1 \}^{k_2+k_3-2})
		=\frac{1}{2}\left\{ (-1)^{k_3}\binom{w'}{k_3}-\frac{1-2^{w'-1}}{2^{w'-1}}\binom{w'}{k_1}\right\} \frac{B_{\mathbf{p}-w'}}{w'}
		\label{MV5}\end{equation}
	}if $k_1$ is even and $k_2+k_3$ odd,
	{\scriptsize \begin{equation}
		\text{\rm \pounds}_{\mathcal{A}, (\{ 1 \}^{k_1-1}, 2, \{ 1\}^{k_2+k_3-1})}^{\ast}(\{ 1 \}^{k_1+k_2-2}, 1/2, 2, \{ 1 \}^{k_3-1})
		=-\frac{1}{2}\left\{ (-1)^{k_1}\binom{w'}{k_1}-\frac{1-2^{w'-1}}{2^{w'-1}}\binom{w'}{k_3}\right\} \frac{B_{\mathbf{p}-w'}}{w'}
		\label{MV5''}\end{equation}
	}if $k_1+k_2$ odd and $k_3$ even,
	{\scriptsize \begin{equation}
		\text{\rm \pounds}_{\mathcal{A}, (\{ 1 \}^{k_1-1}, 2, \{ 1\}^{k_2+k_3-1})}^{\ast, \star}(\{ 1 \}^{k_1+k_2-2}, 2, 1/2, \{ 1 \}^{k_3-1})
		=-\frac{1}{2}\left\{ (-1)^{k_1}\binom{w'}{k_1}-\frac{1-2^{w'-1}}{2^{w'-1}}\binom{w'}{k_3}\right\} \frac{B_{\mathbf{p}-w'}}{w'}
		\label{MV5'}\end{equation}
	}if $k_1+k_2$ odd and $k_3$ even,
	{\small \begin{equation}
		\text{\rm \pounds}_{\mathcal{A}, \{1\}^{w'}}^{\ast}(\{1\}^{k_1-1}, 1/2, 2, \{1\}^{k_2-2}, 1/2, 2, \{1\}^{k_3-1})=\frac{1-2^{w'-1}}{2^{w'}}\left\{\binom{w'}{k_1}-\binom{w'}{k_3}\right\}\frac{B_{\mathbf{p}-w'}}{w'}
		\label{MV6'}\end{equation}	
	}if $k_1$ is even, $k_2$ odd, $k_3$ even, and $k_2>1$,
	{\small \begin{equation}
	\text{\rm \pounds}_{\mathcal{A}, \{ 1 \}^{w'}}^{\ast, \star}(\{ 1 \}^{k_1-1}, 2, 1/2, \{ 1 \}^{k_2-2}, 2, 1/2, \{ 1 \}^{k_3-1}) = \frac{1-2^{w'-1}}{2^{w'}}\left\{ \binom{w'}{k_3}-\binom{w'}{k_1} \right\} \frac{B_{\mathbf{p}-w'}}{w'}
	\label{MV6}\end{equation}
}if $k_1$ is even, $k_2$ odd, $k_3$ even, and $k_2>1$,
	\begin{equation}
	\text{\rm \pounds}_{\mathcal{A}, \{1\}^{w+1}}^{\ast}(\{1\}^{k_1-1}, 1/2, 1, 2, \{1\}^{k_2-1}) = \frac{1-2^w}{2^{w+1}}\left\{ \binom{w+1}{k_1}-\binom{w+1}{k_2}\right\}\frac{B_{\mathbf{p}-w-1}}{w+1}
	\label{MV7'}\end{equation}
	if $k_1$ and $k_2$ are even,
	\begin{equation}
	\text{\rm \pounds}_{\mathcal{A}, \{1\}^{w+1}}^{\ast, \star}(\{1\}^{k_1-1}, 2, 1, 1/2, \{1\}^{k_2-1}) = \frac{1-2^w}{2^{w+1}}\left\{ \binom{w+1}{k_2}-\binom{w+1}{k_1}\right\}\frac{B_{\mathbf{p}-w-1}}{w+1}
	\label{MV7}\end{equation}
	if $k_1$ and $k_2$ are even,
	\begin{align}
	&\text{\rm \pounds}_{\mathcal{A}, (1, 2)}^{\ast, \bullet}(2, 1/2) = -\text{\rm \pounds}_{\mathcal{A}, (2, 1)}^{\ast, \bullet}(2, 1/2) = q_{\mathbf{p}}(2)^3-\frac{7}{8}B_{\mathbf{p}-3},
	\label{(1,2) 2, 1/2}\\
	&\text{\rm \pounds}_{\mathcal{A}, (1, 2)}^{\ast, \bullet}(1/2, 2) = -\text{\rm \pounds}_{\mathcal{A}, (2, 1)}^{\ast, \bullet}(1/2, 2) = -\frac{7}{8}B_{\mathbf{p}-3},
	\label{(1, 2) 1/2, 2}\\
	&\text{\rm \pounds}_{\mathcal{A}, \{ 1 \}^3}^{\ast, \bullet}(2, 1, 1/2) = 0.
	\label{2, 1, 1/2=0}\end{align}
	\label{FHMP special values}\end{theorem}
\begin{proof}
	First, we prove the star-cases. By Theorem \ref{MTA}, we have
	{\footnotesize \begin{equation}
	\text{\rm \pounds}_{\mathcal{A}, (k_1, k_2)}^{\cyr sh , \star}(s, t) = \text{\rm \pounds}_{\mathcal{A}, \{ 1 \}^{k_1+k_2}}^{\cyr sh , \star}(\{ 1 \}^{k_1-1}, 1-s, \{ 1 \}^{k_2-1}, 1-t)-\text{\rm \pounds}_{\mathcal{A}, \{ 1 \}^{k_1+k_2}}^{\cyr sh , \star}(\{ 1\}^{k_1-1}, 1-s, \{ 1 \}^{k_2}),
	\label{s,t}\end{equation}
}where $s$ and $t$ are indeterminates. If we substitute $-1$ (resp. $1$) for $s$ (resp. $t$) in  the equality (\ref{s,t}), then we see that
	\[
	(\text{L. H. S. of (\ref{s,t})}) = \text{\rm \pounds}_{\mathcal{A}, (k_1, k_2)}^{\cyr sh , \star}(-1, 1) = \text{\rm \pounds}_{\mathcal{A}, (k_1, k_2)}^{\ast , \star}(-1, -1)
	\]
	and
	{\small \[
	(\text{R. H. S. of (\ref{s,t})}) = - \text{\rm \pounds}_{\mathcal{A}, \{ 1 \}^{k_1+k_2}}^{\cyr sh , \star}(\{ 1 \}^{k_1-1}, 2, \{ 1\}^{k_2}) = - \text{\rm \pounds}_{\mathcal{A}, \{ 1 \}^{k_1+k_2}}^{\ast , \star}(\{ 1 \}^{k_1-1}, 2, 1/2, \{ 1\}^{k_2-1}) .
	\]
}Therefore, we obtain the equality (\ref{MV1}), (\ref{MV1.5}), and (\ref{B to q 2}) by Lemma \ref{harmonic alt} (\ref{odd}), (\ref{even bullet}), and (\ref{B to q}), respectively. Next, if we substitute $-1$ for $s$ and $t$ in the equality (\ref{s,t}), then we see that
	\[
	(\text{L. H. S. of (\ref{s,t})}) = \text{\rm \pounds}_{\mathcal{A}, (k_1, k_2)}^{\cyr sh , \star}(-1, -1) = \text{\rm \pounds}_{\mathcal{A}, (k_1, k_2)}^{\star}(-1)
	\]
	and
	\[
	(\text{R. H. S. of (\ref{s,t})}) = \text{\rm \pounds}_{\mathcal{A}, \{ 1 \}^{k_1+k_2}}^{\cyr sh , \star}(\{ 1\}^{k_1-1}, 2, \{ 1 \}^{k_2-1}, 2)- \text{\rm \pounds}_{\mathcal{A}, \{ 1\}^{k_1+k_2}}^{\cyr sh , \star}(\{ 1 \}^{k_1-1}, 2, \{ 1\}^{k_2}).
	\]
	Therefore, by combining Lemma \ref{aux} (\ref{k_1, k_2, -1, non-star}) and the equality (\ref{MV1}) which has obtained just before, we have the explicit value of $\text{\rm \pounds}_{\mathcal{A}, \{ 1 \}^{k_1+k_2}}^{\cyr sh , \star}(\{ 1\}^{k_1-1}, 2, \{ 1 \}^{k_2-1}, 2)$ if $k_1+k_2$ is odd. By translating FSSMP into FHSMP, we have the equality (\ref{MV2}) and the first value of (\ref{MV4}). The equality (\ref{MV3}) and the rest of (\ref{MV4}) are obtained by Proposition \ref{reverse prop}. By Corollary \ref{MTB}, we have the following equality:

	\[
	\text{\rm \pounds}_{\mathcal{A}, (k_1, k_2, k_3)}^{\cyr sh, \star}(-1, 1, 1) = -\text{\rm \pounds}_{\mathcal{A}, (\{ 1 \}^{k_1+k_2-1}, 2, \{ 1 \}^{k_3-1})}^{\cyr sh, \star}(\{ 1 \}^{k_1-1}, 2, \{ 1 \}^{k_2+k_3-1}).
	\]
	After translating FSSMPs into FHSMPs, we have the equality (\ref{MV5}) by Lemma \ref{harmonic alt} (\ref{-1, -1, 1}) when $k_1$ is even and $k_2+k_3$ odd and we have the explicit value of $\text{\rm \pounds}_{\mathcal{A}, (1, 2)}^{\ast, \star}(2, 1/2)$ by Lemma \ref{harmonic alt} (\ref{1, -1, -1 and -1, -1, 1 star}) when $k_1=k_2=k_3=1$. The equality (\ref{MV5'}) and the explicit value of $\text{\rm \pounds}_{\mathcal{A}, (2, 1)}^{\ast, \star}(2, 1/2)$ are obtained by Proposition \ref{reverse prop}. The star cases of the equality (\ref{(1, 2) 1/2, 2}) are obtained by the following relations:
	\[
	\text{\rm \pounds}_{\mathcal{A}, 2}(2)\text{\rm \pounds}_{\mathcal{A}, 1}(1/2) = \text{\rm \pounds}_{\mathcal{A}, (2, 1)}^{\ast, \star}(2, 1/2)+\text{\rm \pounds}_{\mathcal{A}, (1, 2)}^{\ast, \star}(1/2, 2)-\zeta_{\mathcal{A}}(3)
	\]
	and
	\[
	\text{\rm \pounds}_{\mathcal{A}, 1}(2)\text{\rm \pounds}_{\mathcal{A}, 2}(1/2) = \text{\rm \pounds}_{\mathcal{A}, (1, 2)}^{\ast, \star}(2, 1/2)+\text{\rm \pounds}_{\mathcal{A}, (2, 1)}^{\ast, \star}(1/2, 2)-\zeta_{\mathcal{A}}(3).
	\]
	By Theorem \ref{MTA}, we have the following equality:
	\[
	\text{\rm \pounds}_{\mathcal{A}, (k_1, k_2, k_3)}^{\cyr sh, \star}(-1, -1, 1) = -\text{\rm \pounds}_{\mathcal{A}, \{ 1 \}^{w'}}^{\cyr sh, \star}(\{ 1 \}^{k_1-1}, 2, \{ 1 \}^{k_2-1}, 2, \{ 1 \}^{k_3}).
	\]
	After translating FSSMPs into FHSMPs, we have the equalities (\ref{MV6}) and (\ref{MV7}) by Lemma \ref{harmonic alt} (\ref{-1, 1, -1}) when $k_1$ is even, $k_2$ odd, and $k_3$ even and we have the explicit value of $\text{\rm \pounds}_{\mathcal{A}, \{ 1 \}^3}^{\ast, \star}(2, 1, 1/2)$ by Lemma \ref{harmonic alt} (\ref{-1, 1, -1=0}) when $k_1=k_2=k_3=1$. 
	
	Next, we prove non-star cases. By Theorem \ref{MTC} (\ref{A_n formula}), we have
	\begin{equation}
	\begin{split}
	&(-1)^{w}\text{\rm \pounds}_{\mathcal{A}, \{1\}^w}^{\ast}(\{1\}^{k_1-1}, 1/2, 2, \{1\}^{k_2-1})\\
	&=-\text{\rm \pounds}_{\mathcal{A}, \{1\}^w}^{\ast, \star}(\{1\}^{k_2-1}, 2, 1/2, \{1\}^{k_1-1})-(-1)^{k_1}\widetilde{\text{\rm \pounds}}_{\mathcal{A}, \{1\}^{k_1}}(1/2)\widetilde{\text{\rm \pounds}}_{\mathcal{A}, \{1\}^{k_2}}^{\star}(2).
	\end{split}
	\end{equation}
	Suppose $k_1, k_2 \geq 2$. By Proposition \ref{special values in A} (\ref{easy gen of Sun}) and (\ref{gen of Sun}), we have
	\[
	\widetilde{\text{\rm \pounds}}_{\mathcal{A}, \{1\}^{k_1}}(1/2)\widetilde{\text{\rm \pounds}}_{\mathcal{A}, \{1\}^{k_2}}^{\star}(2)= -\frac{(2^{k_1-1}-1)(2^{k_2-1}-1)}{2^{w-3}k_1k_2}B_{\mathbf{p}-k_1}B_{\mathbf{p}-k_2}.
	\]
	If $w$ is odd, then 
	\[
	\text{\rm \pounds}_{\mathcal{A}, \{1\}^w}^{\ast}(\{1\}^{k_1-1}, 1/2, 2, \{1\}^{k_2-1})=\text{\rm \pounds}_{\mathcal{A}, \{1\}^w}^{\ast, \star}(\{1\}^{k_2-1}, 2, 1/2, \{1\}^{k_1-1})
	\]
	holds and if $w$ is even, then
	\[
	\begin{split}
	&\text{\rm \pounds}_{\mathcal{A}, \{1\}^w}^{\ast}(\{1\}^{k_1-1}, 1/2, 2, \{1\}^{k_2-1})\\
	&=\frac{(2^{k_1-1}-1)(2^{k_2-1}-1)}{2^{w-3}k_1k_2}B_{\mathbf{p}-k_1}B_{\mathbf{p}-k_2}+(-1)^{k_1}\frac{(2^{k_1-1}-1)(2^{k_2-1}-1)}{2^{w-3}k_1k_2}B_{\mathbf{p}-k_1}B_{\mathbf{p}-k_2}=0,
	\end{split}
	\]
	holds by the equality (\ref{MV1.5}). The case $k_1=1$ and $k_2=1$ are similar. Therefore we obtain the equalities (\ref{MV1'}) and (\ref{MV2'}). Suppose that $w$ is odd. By Theorem \ref{MTC} (\ref{A_n formula}), we have
	\begin{equation}
	\begin{split}
	&\text{\rm \pounds}_{\mathcal{A}, \{1\}^w}^{\ast}(2, \{1\}^{k_1-2}, 1/2, 2, \{1\}^{k_2-1})=\text{\rm \pounds}_{\mathcal{A}, \{1\}^w}^{\ast, \star}(\{1\}^{k_2-1}, 2, 1/2, \{1\}^{k_1-2}, 2)\\
	&\hspace{23.5mm}+\sum_{j=1}^{k_1-1}(-1)^j\text{\rm \pounds}_{\mathcal{A}, \{1\}^j}^{\ast}(2)\text{\rm \pounds}_{\mathcal{A}, \{1\}^{w-j}}^{\ast, \star}(\{1\}^{k_2-1}, 2, 1/2, \{1\}^{k_1-j-1})\\
	&\hspace{23.5mm}+(-1)^{k_1}\text{\rm \pounds}_{\mathcal{A}, \{1\}^{k_1}}^{\ast}(2, \{1\}^{k_1-2}, 1/2)\widetilde{\text{\rm \pounds}}_{\mathcal{A}, \{1\}^{k_2}}^{\star}(2).
	\end{split}
	\label{universe}\end{equation}
	Suppose that $k_2$ is even. $\widetilde{\text{\rm \pounds}}_{\mathcal{A}, \{1\}^{k_2}}^{\star}(2)=0$ holds by Proposition \ref{special values in A} (\ref{easy gen of Sun}). Furthermore, if $j$ is even, $\text{\rm \pounds}_{\mathcal{A}, \{1\}^j}(2)=0$ holds and if $j$ is odd, $\text{\rm \pounds}_{\mathcal{A}, \{1\}^{w-j}}^{\ast, \star}(\{1\}^{k_2-1}, 2, 1/2, \{1\}^{k_1-j-1})=0$ holds by the equality (\ref{MV1.5}). Hence, 
	\[
	\text{\rm \pounds}_{\mathcal{A}, \{1\}^w}^{\ast}(2, \{1\}^{k_1-2}, 1/2, 2, \{1\}^{k_2-1})=\text{\rm \pounds}_{\mathcal{A}, \{1\}^w}^{\ast, \star}(\{1\}^{k_2-1}, 2, 1/2, \{1\}^{k_1-2}, 2)
	\]
	holds and we have the equality (\ref{MV2.5}) by the equality (\ref{MV2}). If $m \geq 5$ is odd, we have
	\begin{equation*}
	\begin{split}
	&\text{\rm \pounds}_{\mathcal{A}, \{1\}^m}^{\ast}(2, 1/2, 2, \{1\}^{m-3})
	-\text{\rm \pounds}_{\mathcal{A}, \{1\}^m}^{\ast, \star}(\{1\}^{m-3}, 2, 1/2, 2)\\
	&=-\text{\rm \pounds}_{\mathcal{A}, 1}(2)\text{\rm \pounds}_{\mathcal{A}, \{1\}^{m-1}}^{\ast, \star}(\{1\}^{m-3}, 2, 1/2)+\text{\rm \pounds}_{\mathcal{A}, \{1\}^2}^{\ast}(2, 1/2)\widetilde{\text{\rm \pounds}}_{\mathcal{A}, \{1\}^{m-2}}^{\star}(2)\\
	&=-(-2q_{\mathbf{p}}(2))\frac{1-2^{m-3}}{2^{m-4}(m-2)}q_{\mathbf{p}}(2)B_{\mathbf{p}-m+2}+(-2q_{\mathbf{p}}(2)^2)\frac{1-2^{m-3}}{2^{m-4}}\frac{B_{\mathbf{p}-m+2}}{m-2}\\
	&=0
	\end{split}
	\end{equation*}
	by the equality (\ref{universe}), Lemma \ref{FP values} (\ref{A_3, 2}), the equality (\ref{B to q 2}), Theorem \ref{A_2 multiple thm} (\ref{110}) below, and Proposition \ref{special values in A} (\ref{easy gen of Sun}). The case $m=3$ is similar. Therefore, we have the equality (\ref{MV2.7}). The equalities (\ref{reverse mult1}) and (\ref{reverse mult2}) are obtained by Proposition \ref{reverse prop}. The equalities (\ref{MV4'}) are obtained by Corollary \ref{zeta corollary} (\ref{Tauraso-Zhao}) and the equality (\ref{MV4}).  Since the equality (\ref{MV5'''}) is obtained by Proposition \ref{reverse prop}, we prove the equality (\ref{MV5''}). Suppose that $k_1+k_2$ is odd and $k_3$ even. By Theorem \ref{MTC} (\ref{A_n formula}), we have
	{\scriptsize \begin{equation*}
	\begin{split}
	\text{\rm \pounds}_{\mathcal{A}, (\{1\}^{k_1-1}, 2, \{1\}^{k_2+k_3-1})}^{\ast}(\{1\}^{k_1+k_2-2}, 1/2, 2, \{1\}^{k_3-1})
	&=-\text{\rm \pounds}_{\mathcal{A}, (\{1\}^{k_2+k_3-1}, 2, \{1\}^{k_1-1})}^{\ast, \star}(\{1\}^{k_3-1}, 2, 1/2, \{1\}^{k_1+k_2-2})\\
	&\hspace{3.6mm}-(-1)^{k_1+k_2-1}\widetilde{\text{\rm \pounds}}_{\mathcal{A}, (\{1\}^{k_1-1}, 2, \{1\}^{k_2-1})}(1/2)\widetilde{\text{\rm \pounds}}_{\mathcal{A}, \{1\}^{k_3}}^{\star}(2).
	\end{split}
	\end{equation*}}
	Since $\widetilde{\text{\rm \pounds}}_{\mathcal{A}, \{1\}^{k_3}}^{\star}(2)=0$ by Proposition \ref{special values in A} (\ref{easy gen of Sun}), we have the equality (\ref{MV5''}) by the equality (\ref{MV5}). Suppose that $k_1$ is even, $k_2$ odd, $k_3$ even, and $k_2>1$. By Theorem \ref{MTC} (\ref{A_n formula}), we have
	{\footnotesize \begin{equation*}
	\begin{split}
	&\text{\rm \pounds}_{\mathcal{A}, \{1\}^{w'}}^{\ast}(\{1\}^{k_1-1}, 1/2, 2, \{1\}^{k_2-2}, 1/2, 2, \{1\}^{k_3-1})
	=\text{\rm \pounds}_{\mathcal{A}, \{1\}^{w'}}^{\ast, \star}(\{1\}^{k_3-1}, 2, 1/2, \{1\}^{k_2-2}, 2, 1/2, \{1\}^{k_1-1})\\
	&\hspace{9.5mm}+(-1)^{k_1}\widetilde{\text{\rm \pounds}}_{\mathcal{A}, \{1\}^{k_1}}(1/2)\text{\rm \pounds}_{\mathcal{A}, \{1\}^{k_2+k_3}}^{\ast, \star}(\{1\}^{k_3-1}, 2, 1/2, \{1\}^{k_2-2}, 2)\\
	&\hspace{9.5mm}+\sum_{j=0}^{k_2-2}(-1)^{k_1+j+1}\text{\rm \pounds}_{\mathcal{A}, \{1\}^{k_1+j+1}}^{\ast}(\{1\}^{k_1-1}, 1/2, 2, \{1\}^j)\text{\rm \pounds}_{\mathcal{A}, \{1\}^{k_2+k_3-j-1}}^{\ast, \star}(\{1\}^{k_3-1}, 2, 1/2, \{1\}^{k_2-j-2})\\
	&\hspace{9.5mm}+(-1)^{k_1+k_2}\text{\rm \pounds}_{\mathcal{A}, \{1\}^{k_1+k_2}}^{\ast}(\{1\}^{k_1-1}, 1/2, 2, \{1\}^{k_2-2}, 1/2)\widetilde{\text{\rm \pounds}}_{\mathcal{A}, \{1\}^{k_3}}^{\star}(2).
	\end{split}
	\end{equation*}
}For $j = 0, \dots, k_2-2$, if $j$ is odd, $\text{\rm \pounds}_{\mathcal{A}, \{1\}^{k_1+j+1}}^{\ast}(\{1\}^{k_1-1}, 1/2, 2, \{1\}^j)=0$ holds by the equality (\ref{MV2'}) and if $j$ is even, 
	\[
	\text{\rm \pounds}_{\mathcal{A}, \{1\}^{k_2+k_3-j-1}}^{\ast, \star}(\{1\}^{k_3-1}, 2, 1/2, \{1\}^{k_2-j-2}) = \text{(a certain element of $\mathcal{A}$)} \times B_{\mathbf{p}-k_3} =0
	\]
	holds by the equality (\ref{MV1.5}). Furthermore, $\widetilde{\text{\rm \pounds}}_{\mathcal{A}, \{1\}^{k_1}}(1/2)=\widetilde{\text{\rm \pounds}}_{\mathcal{A}, \{1\}^{k_3}}^{\star}(2)=0$ holds by Proposition \ref{special values in A} (\ref{easy gen of Sun}) and (\ref{gen of Sun}). Therefore we obtain the equality (\ref{MV6'}) by the equality (\ref{MV6}). Similarly, we see that the equality (\ref{MV7'}) holds. The non-star cases of the equalities (\ref{(1,2) 2, 1/2}), (\ref{(1, 2) 1/2, 2}), and (\ref{2, 1, 1/2=0}) are also obtained by the star cases of them.
\end{proof}
\begin{remark}
	The case $m=3$ of the equalities (\ref{MV4'})
	\begin{equation}
	\text{\rm \pounds}_{\mathcal{A}, \{ 1 \}^3}^{\ast}(1, 2, 1) = -2\text{\rm \pounds}_{\mathcal{A}, \{ 1\}^3}^{\ast}(1, 1/2, 1) =\frac{1}{2}B_{\mathbf{p}-3}
	\label{1^3 non-star}
	\end{equation}
	also have been obtained by Tauraso and J. Zhao \cite[Proposition 7.1 (85)]{TZ}.
	\label{TZ remark}\end{remark}
\begin{theorem}
	Let $k_1$ and $k_2$ be positive even integers and $w=k_1+k_2+1$. Then we have
	\begin{align}
	&\text{\rm \pounds}_{\mathcal{A}_2, \{1\}^{w-1}}^{\ast}(\{1\}^{k_1-1}, 1/2, 2, \{1\}^{k_2-1}) = - \frac{1}{2}\left\{1+\frac{2^{w-1}-1}{2^{w-1}}\binom{w}{k_2} \right\}\frac{B_{\mathbf{p}-w}}{w}\mathbf{p},\label{A_2 mult nonstar}\\
	&\text{\rm \pounds}_{\mathcal{A}_2, \{1\}^{w-1}}^{\ast, \star}(\{1\}^{k_1-1}, 2, 1/2, \{1\}^{k_2-1}) = \frac{1}{2}\left\{ 1+\frac{2^{w-1}-1}{2^{w-1}}\binom{w}{k_1} \right\} \frac{B_{\mathbf{p}-w}}{w}\mathbf{p},
	\label{A_2 multiple}\\
	&\text{\rm \pounds}_{\mathcal{A}_2, \{1\}^2}^{\ast}(2, 1/2) = -2q_{\mathbf{p}}(2)^2+\left( q_{\mathbf{p}}(2)^3-\frac{7}{8}B_{\mathbf{p}-3} \right) \mathbf{p},
	\label{110}\\
	&\text{\rm \pounds}_{\mathcal{A}_2, \{1\}^2}^{\ast, \star}(2, 1/2) = -2q_{\mathbf{p}}(2)^2+\left( q_{\mathbf{p}}(2)^3-\frac{5}{24}B_{\mathbf{p}-3} \right) \mathbf{p},
	\label{111}\\
	&\text{\rm \pounds}_{\mathcal{A}_2, \{1\}^2}^{\ast}(1/2, 2) = \frac{5}{24}B_{\mathbf{p}-3} \mathbf{p},
	\label{112}\\
	&\text{\rm \pounds}_{\mathcal{A}_2, \{1\}^2}^{\ast, \star}(1/2, 2) = \frac{7}{8}B_{\mathbf{p}-3} \mathbf{p}.
	\label{113}\
	\end{align}
	\label{A_2 multiple thm}\end{theorem}
\begin{proof}
	By Theorem \ref{MTA}, we have
	\begin{equation*}
	\begin{split}
	&\text{\rm \pounds}_{\mathcal{A}_2, (k_1, k_2)}^{\ast, \star}(-1, -1) + \Bigl(\text{\rm \pounds}_{\mathcal{A}_2, (1, k_1, k_2)}^{\ast, \star}(1, -1, -1)-\text{\rm \pounds}_{\mathcal{A}_2, (k_1+1, k_2)}^{\ast, \star}(-1, -1) \Bigr) \mathbf{p}\\ 
	&=-\text{\rm \pounds}_{\mathcal{A}_2, \{1\}^{k_1+k_2}}^{\ast, \star}(\{1\}^{k_1-1}, 2, 1/2, \{1\}^{k_2-1}).
	\end{split}
	\end{equation*}
	Therefore the equality (\ref{A_2 multiple}) is obtained by Lemma \ref{harmonic alt} (\ref{odd}), (\ref{1, -1, -1}), and Lemma \ref{A_2 alt} (\ref{new Zhao star}). By Theorem \ref{MTC} (\ref{A_n formula}), we have
	\begin{equation*}
	\begin{split}
	&\text{\rm \pounds}_{\mathcal{A}_2, \{1\}^{w-1}}^{\ast}(\{1\}^{k_1-1}, 1/2, 2, \{1\}^{k_2-1})
	=-\text{\rm \pounds}_{\mathcal{A}_2, \{1\}^{w-1}}^{\ast, \star}(\{1\}^{k_2-1}, 2, 1/2, \{1\}^{k_1-1})\\
	&\hspace{35mm}-\sum_{j=1}^{k_1-1}(-1)^j\zeta_{\mathcal{A}_2}(\{1\}^j)\text{\rm \pounds}_{\mathcal{A}_2, \{1\}^{w-j-1}}^{\ast, \star}(\{1\}^{k_2-1}, 2, 1/2, \{1\}^{k_1-j-1})\\
	&\hspace{35mm}-(-1)^{k_1}\widetilde{\text{\rm \pounds}}_{\mathcal{A}_2, \{1\}^{k_1}}(1/2)\widetilde{\text{\rm \pounds}}_{\mathcal{A}_2, \{1\}^{k_2}}^{\star}(2)\\
	&\hspace{35mm}-\sum_{i=0}^{k_2-2}(-1)^{k_1+i+1}\text{\rm \pounds}_{\mathcal{A}_2, \{1\}^{k_1+i+1}}^{\ast}(\{1\}^{k_1-1}, 1/2, 2, \{1\}^i)\zeta_{\mathcal{A}_2}^{\star}(\{1\}^{k_2-i-1}).
	\end{split}
	\end{equation*}
	For $j=1, \dots, k_1-1$, if $j$ is odd, then $\zeta_{\mathcal{A}_2}(\{1\}^j)=0$ and if $j$ is even, then \[
	\zeta_{\mathcal{A}_2}(\{1\}^j)\text{\rm \pounds}_{\mathcal{A}_2, \{1\}^{w-j-1}}^{\ast, \star}(\{1\}^{k_2-1}, 2, 1/2, \{1\}^{k_1-j-1})=(\text{a certain element of $\mathcal{A}_2$}) \times B_{\mathbf{p}-k_2}=0
	\] by Proposition \ref{FMZV's properties} (\ref{Zhou-Cai}) and Theorem \ref{FHMP special values} (\ref{MV1.5}). By Theorem \ref{A_2 theorem}, we have
	\[
	\widetilde{\text{\rm \pounds}}_{\mathcal{A}_2, \{1\}^{k_1}}(1/2)\widetilde{\text{\rm \pounds}}_{\mathcal{A}_2, \{1\}^{k_2}}^{\star}(2)=0.
	\]
	For $i=0, \dots, k_2-2$, if $i$ is even, then $\zeta_{\mathcal{A}_2}^{\star}(\{1\}^{k_2-i-1})=0$ holds and if $i$ is odd, then 
	\[
	\text{\rm \pounds}_{\mathcal{A}_2, \{1\}^{k_1+i+1}}^{\ast}(\{1\}^{k_1-1}, 1/2, 2, \{1\}^i)\zeta_{\mathcal{A}_2}^{\star}(\{1\}^{k_2-i-1})=0
	\]
	by Proposition \ref{FMZV's properties} (\ref{Zhou-Cai}) and Theorem \ref{FHMP special values} (\ref{MV2'}). Therefore, we have the equality (\ref{A_2 mult nonstar}) by the equality (\ref{A_2 multiple}).
	We also have the equality (\ref{111}) by applying Theorem \ref{MTA}. The equalities (\ref{110}), (\ref{112}), and (\ref{113}) are obtained by the following relations:
	\begin{align*}
	&\text{\rm \pounds}_{\mathcal{A}_2, \{1\}^2}^{\ast, \star}(2, 1/2) = \text{\rm \pounds}_{\mathcal{A}_2, \{1\}^2}^{\ast}(2, 1/2) + \zeta_{\mathcal{A}_2}(2),\\
	&\text{\rm \pounds}_{\mathcal{A}_2, 1}(2)\text{\rm \pounds}_{\mathcal{A}_2, 1}(1/2) = \text{\rm \pounds}_{\mathcal{A}_2, \{1\}^2}^{\ast}(2, 1/2)+\text{\rm \pounds}_{\mathcal{A}_2, \{1\}^2}^{\ast}(1/2, 2)+\zeta_{\mathcal{A}_2}(2),\\
	&\text{\rm \pounds}_{\mathcal{A}_2, 1}(2)\text{\rm \pounds}_{\mathcal{A}_2, 1}(1/2) = \text{\rm \pounds}_{\mathcal{A}_2, \{1\}^2}^{\ast, \star}(2, 1/2)+\text{\rm \pounds}_{\mathcal{A}_2, \{1\}^2}^{\ast, \star}(1/2, 2)-\zeta_{\mathcal{A}_2}(2).
	\end{align*}
\end{proof}
\subsection{Relation between Ono-Yamamoto's FMPs and our FMPs}
\label{subsec:Relation between Ono-Yamamoto's FMPs and our FMPs}
Ono and Yamamoto gave another definition of finite multiple polylogarithms in the paper \cite{OY}. Their purpose is to establish the shuffle relation of FMPs (\cite[Theorem 3.6]{OY}). In this subsection, we give the relation between Ono-Yamamoto's FMPs and our FMPs. Furthermore, we calculate some special values of Ono-Yamamoto's FMPs.
\begin{definition}
	Let $\Bbbk = (k_1, \dots , k_m)$ be an index. Then {\em Ono-Yamamoto's finite multiple polylogarithm} $\li_{\Bbbk}(t) \in \mathcal{A}_{{\mathbb Z} [t]}$ is defined by
	\[
	\li_{\Bbbk}(t) := \Biggl( \sideset{}{'}\sum\limits_{0<l_1, \dots , l_m<p} \frac{t^{l_1+\cdots +l_m}}{l_1^{k_1}(l_1+l_2)^{k_2}\cdots (l_1+\cdots +l_m)^{k_m}} \Biggr)_p,
	\]
	where the summation $\sum'$ runs over only fractions whose denominators are prime to $p$. 
\end{definition}
By the substitution $l_i \mapsto p-l_i$, we have
\begin{equation}
\li_{\Bbbk}(r) = (-1)^{\wt (\Bbbk )} r^{\dep (\Bbbk )}\li_{\Bbbk}(r^{-1})
\label{reverse for OY}\end{equation}
for any non-zero rational number $r$.

We prepare the following notations to discuss the relation between Ono-Yamamoto's FMPs and our FMPs (cf. \cite[Section 2]{OY}):
\[
[l]:= \{ 1, \dots , l \} ,
\]
\[
\Phi_{m, l} := \{ \phi \colon [m] \to [l] : \text{surjective} \mid \phi (a) \neq \phi (a+1) \ \text{for all} \ a \in [m-1] \} ,
\]
\[
m_{\phi}:= l \ \text{when} \ \phi \in \Phi_{m, l},
\]
\[
\Phi_m := \bigsqcup_{l=1}^m \Phi_{m, l}, \ \delta_{\phi}(i) := \# \{ a \in [i-1] \mid \phi (a) > \phi (a+1) \} \ \text{for $\phi$ in $\Phi_m$},
\]
\[
\beta \colon \Phi_m \to [m] \ \text{is defined by} \ \beta (\phi) := \delta_{\phi}(m)+1, \ \Phi_m^i := \beta^{-1}(i), 
\]
where $i, l$, and $m$ are positive integers.
\begin{proposition}
	Let $\Bbbk = (k_1, \dots , k_m)$ be an index. Then we have
	\begin{equation}
	\li_{\Bbbk}(t) = \sum_{i=1}^m t^{(i-1)\mathbf{p}} \sum_{\phi \in \Phi_{m}^i} \text{\rm \pounds}_{\mathcal{A}, (\sum_{\phi (j)=m_{\phi}}k_j, \dots , \sum_{\phi(j) = 1}k_j)}^{\ast}(\{ 1 \}^{m_{\phi}-\phi (m)}, t , \{ 1 \}^{\phi (m)-1}).
	\label{OYSS}\end{equation}
	\label{OY-SS}\end{proposition}
\begin{proof}
	We omit the proof because it is completely the same as the proof of \cite[Proposition 2.4]{OY}.
\end{proof}
\begin{corollary}
	Let $k_1, k_2$, and $k_3$ be positive integers. Then we have
	\begin{equation}
	\li_{(k_1, k_2)}(t) = \text{\rm \pounds}_{\mathcal{A}, (k_2, k_1)}(t)+t^{\mathbf{p}}\widetilde{\text{\rm \pounds}}_{\mathcal{A}, (k_1, k_2)}(t),
	\label{OY2}\end{equation}
	\begin{equation}
	\begin{split}
	\li_{(k_1, k_2, k_3)}(t) &= \text{\rm \pounds}_{\mathcal{A}, (k_3, k_2, k_1)}(t)+t^{\mathbf{p}} \text{\rm \pounds}_{\mathcal{A}, (k_3, k_1, k_2)}(t) + t^{\mathbf{p}}\widetilde{\text{\rm \pounds}}_{\mathcal{A}, (k_2, k_1, k_3)}(t)\\
	&\hspace{4mm} +t^{\mathbf{p}}\text{\rm \pounds}_{\mathcal{A}, (k_2, k_3, k_1)}^{\ast}(1, t, 1) + t^{\mathbf{p}}\text{\rm \pounds}_{\mathcal{A}, (k_1, k_3, k_2)}^{\ast}(1, t, 1)\\
	&\hspace{4mm} + t^{\mathbf{p}}\text{\rm \pounds}_{\mathcal{A}, (k_1+k_3, k_2)}(t) + t^{\mathbf{p}}\widetilde{\text{\rm \pounds}}_{\mathcal{A}, (k_2, k_1+k_3)}(t) +t^{2\mathbf{p}}\widetilde{\text{\rm \pounds}}_{\mathcal{A}, (k_1, k_2, k_3)}(t) .
	\end{split}
	\label{OY3}\end{equation}
	\label{OY k_1, k_2}\end{corollary}
In general, it is difficult to calculate each term in the right hand side of (\ref{OYSS}) for $1<i<m$. Therefore it seems to hard to calculate special values of Ono-Yamamoto's FMPs. However, we can evaluate the following values:
\begin{proposition}
	Let $\Bbbk = (k_1, \dots , k_m)$ be an index. Then we have
	\begin{align}
	&\li_{\Bbbk}(1)=0,
	\label{OYsp1}\\
	&\li_{\{ 1 \}^2}(-1) = \li_{\{ 1 \}^2}(2) = 2q_{\mathbf{p}}(2)^2, \ \li_{\{ 1 \}^2}(1/2) = \frac{1}{2}q_{\mathbf{p}}(2)^2,
	\label{OY 1^2}\\
	&\li_{(1, 2)}(-1)=\li_{(2,1)}(-1)=0,
	\label{OY (1,2) -1}\\
	&\li_{(1, 2)}(2) = \frac{2}{3}q_{\mathbf{P}}(2)^3-\frac{2}{3}B_{\mathbf{p}-3}, \ \li_{(2, 1)}(2) = \frac{2}{3}q_{\mathbf{p}}(2)^3+\frac{4}{3}B_{\mathbf{p}-3},
	\label{OY (1,2) 2}\\
	&\li_{(1,2)}(1/2) = -\frac{1}{6}q_{\mathbf{p}}(2)^3+\frac{1}{6}B_{\mathbf{p}-3}, \ \li_{(2, 1)}(1/2) = -\frac{1}{6}q_{\mathbf{p}}(2)^3-\frac{1}{3}B_{\mathbf{p}-3},
	\label{OY (1,2) 1/2}\\
	&\li_{\{ 1 \}^3}(-1) = \li_{\{ 1 \}^3}(2) = -\frac{4}{3}q_{\mathbf{p}}(2)^3-\frac{2}{3}B_{\mathbf{p}-3}, \ \li_{\{ 1 \}^3}(1/2) = \frac{1}{6}q_{\mathbf{p}}(2)^3+\frac{1}{12}B_{\mathbf{p}-3}.
	\label{OY 1^3}\end{align}
	\label{OY special values}\end{proposition}
\begin{proof}
	If $m=1$, we have $\li_{\Bbbk}(1)=\zeta_{\mathcal{A}}(\Bbbk )=0$. We assume that $m$ is greater than or equal to $2$. Let $l$ be one of $2, \dots , m$ and $S_l$ the $l$-th symmetric group. We define an equivalence relation on $\Phi_{m, l}$ as follows: $\phi \sim \phi'$ holds for $\phi , \phi' \in \Phi_{m, l}$ if and only if there exists $\sigma \in S_l$ such that $\phi = \sigma \circ \phi'$ holds. We take and fix a system of representatives $\{ \phi_{l, 1}, \dots \phi_{l, i_l} \}$ of the quotient set $\Phi_{m, l}/S_l$ where $i_l$ is the cardinality of $\Phi_{m, l}/S_l$. Then, by Proposition \ref{OY-SS}, we have
	\begin{equation*}
	\begin{split}
	\li_{\Bbbk}(1) &= \sum_{\phi \in \Phi_m}\zeta_{\mathcal{A}}\Bigl( \sum_{\phi (j)=m_{\phi}}k_j, \dots , \sum_{\phi (j) = 1}k_j \Bigr)
	=\sum_{l=2}^m\sum_{\phi \in \Phi_{m, l}}\zeta_{\mathcal{A}}\Bigl( \sum_{\phi (j)=l}k_j, \dots , \sum_{\phi (j) = 1}k_j \Bigr) \\
	&=\sum_{l=2}^m\sum_{s=1}^{i_l}\left( \sum_{\sigma \in S_l}\zeta_{\mathcal{A}} \Biggl( \sigma \Bigl( \sum_{\phi_{l, s} (j)=l}k_j, \dots , \sum_{\phi_{l, s} (j) = 1}k_j \Bigr) \Biggr) \right) .
	\end{split}
	\end{equation*}
	We see that this is zero by Proposition \ref{FMZV's properties} (\ref{sym sum}).
	
	The equalities (\ref{OY 1^2}), (\ref{OY (1,2) -1}), (\ref{OY (1,2) 2}), (\ref{OY (1,2) 1/2}), and (\ref{OY 1^3}) are obtained by Corollary \ref{OY k_1, k_2}, Lemma \ref{aux}, Proposition \ref{special values in A}, Proposition \ref{Fermat quotient}, Lemma \ref{harmonic alt} (\ref{1, -1, 1}), and Remark \ref{TZ remark} (\ref{1^3 non-star}).
\end{proof}
\appendix
\section{An elementary proof of the congruences between Bernoulli numbers and generalized Bernoulli numbers}
\label{sec:An elementary proof of the congruences between Bernoulli numbers and generalized Bernoulli numbers}
Let $N$ be a positive integer and $\chi$ a Dirichlet character modulo $N$. The generalized Bernoulli numbers $B_{n, \chi}$ are defined by
\[
\sum_{a=1}^N\frac{\chi (a)te^{at}}{e^{Nt}-1} = \sum_{n=0}^{\infty}B_{n, \chi}\frac{t^n}{n!}.
\]
We use the following formula for non-trivial character $\chi$:
\begin{equation}
B_{1, \chi} = \frac{1}{N}\sum_{a=1}^N\chi (a)a.
\label{B_1}\end{equation}
In this appendix, we give an elementary proof of the following classical congruence:
\begin{proposition}
	Let $k$ be an integer greater than or equal to $2$ and $p$ a prime greater than $k+1$. Let $\omega_p$ be the Teichm\"uller character. Then 
	\begin{equation}
	B_{1, \omega_p^{-k}} \equiv -\frac{B_{p-k}}{k} \pmod{p}.
	\end{equation}
	\label{classical congruence}\end{proposition}
For an integer $m$ and a positive integer $n$, we define $B_{1, \omega_{\mathbf{p}}^m}$ to be $(B_{1, \omega_p^m} \bmod{p^n}) \in \mathcal{A}_n$. Then we have $B_{1, \omega_{\mathbf{p}}^{-k}} = \widehat{B}_{\mathbf{p}-k}$ in $\mathcal{A}$ by the above proposition.\footnote[1]{See Example \ref{example} about the notation $\widehat{B}_m$. The element $\widehat{B}_{\mathbf{p}-k}$ is considered as a finite analogue of Riemann zeta value $\zeta (k)$. See \cite{K}.} Therefore we can express some FMZVs or some special values of FMPs by the generalized Bernoulli numbers, e.g. 
\[
\zeta_{\mathcal{A}}(k_1, k_2) = (-1)^{k_1-1}\binom{k_1+k_2}{k_1}B_{1, \omega_{\mathbf{p}}^{-k_1-k_2}}
\]
where $k_1$ and $k_2$ are positive integers. 

We can prove Proposition \ref{classical congruence} by using the Kubota-Leopoldt $p$-adic $L$-function (cf. \cite[Corollary 5.15]{Wa}), however, we give a proof of stronger results by elementary calculations. 
\begin{theorem}
	Let $k$ be an integer greater than or equal to $2$. If $k$ is odd, then
	\begin{equation}
	\begin{split}
	B_{1, \omega_{\mathbf{p}}^{-k}} &= \frac{k^2-k+6}{2}\widehat{B}_{\mathbf{p}-k}-(k^2-2k+3)\widehat{B}_{2\mathbf{p}-k-1}+\frac{k^2-3k+2}{2}\widehat{B}_{3\mathbf{p}-k-2}\\
	& \hspace{5mm}+k(\widehat{B}_{\mathbf{p}-k}-\widehat{B}_{2\mathbf{p}-k-1})\mathbf{p}-k\widehat{B}_{\mathbf{p}-k-2} \mathbf{p}^2 \ \  \text{in} \ \mathcal{A}_3
	\end{split}
	\label{A_3 app odd}\end{equation}
	and if $k$ is even, then
	\begin{equation}
	B_{1, \omega_{\mathbf{p}}^{-k}} = -3k(\widehat{B}_{\mathbf{p}-k-1}-\widehat{B}_{2\mathbf{p}-k-2})\mathbf{p} \ \text{in} \ \ \mathcal{A}_3.
	\label{A_3 app even}\end{equation}
	\label{appB thm}\end{theorem}
\begin{corollary}
	Let $k$ be an integer greater than or equal to $2$. Then
	\begin{align}
	&B_{1, \omega_{\mathbf{p}}^{-k}} = (k+2)\widehat{B}_{\mathbf{p}-k}-(k+1)\widehat{B}_{2\mathbf{p}-k-1} \ \ \text{in} \ \mathcal{A}_2,
	\label{A_2 app}\\
	&B_{1, \omega_{\mathbf{p}}^{-k}} = \widehat{B}_{\mathbf{p}-k} \ \  \text{in} \ \mathcal{A}.
	\label{A app}\end{align}
	\label{app cor}\end{corollary}
\begin{proof}
	We can prove these by the same method as the proof of Theorem \ref{appB thm} or by Kummer's congruences below.
\end{proof}
We consider congruences in the $p$-adic integer ring $\mathbb{Z}_p$. First, we recall Kummer's congruences in the sense of Z. H. Sun.
\begin{proposition}[Z. H. Sun \cite{ZHS1}]
	Let $p$ be an odd prime number. Let $m$ and $l$ be positive integers satisfying $l \not \equiv 0 \pmod{p-1}$. Then
	\begin{equation}
	\widehat{B}_{m(p-1)+l}\equiv m\widehat{B}_{p-1+l}-(m-1)(1-p^{l-1})\widehat{B}_{l} \pmod{p^2}.
	\label{Kummer-Sun}\end{equation} 
\end{proposition}
This congruence is a generalization of Kummer's congruence, that is, $\widehat{B}_n \equiv \widehat{B}_m \pmod{p}$ when $n \equiv m \pmod{p-1}$.

From now on, we fix an integer $k$ greater than or equal to $2$ and an odd prime $p$ satisfying $p> k+4$.
\begin{lemma}
	Let $p$ and $k$ be as above. Then
	{\footnotesize \begin{align*}
	&\sum_{a=1}^{p-1}a^{1-k} \equiv
	\begin{cases}
	-(k-1)(3\widehat{B}_{p-k}-3\widehat{B}_{2p-k-1}+\widehat{B}_{3p-k-2})p-\binom{k+1}{3}\widehat{B}_{p-k-2}p^3 \pmod{p^4} & \text{if $k$ is odd,}\\
	\binom{k}{2}(2\widehat{B}_{p-k-1}-\widehat{B}_{2p-k-2})p^2 \pmod{p^4} & \text{if $k$ is even,}
	\end{cases}\\
	&\sum_{a=1}^{p-1}a^{p-k} \equiv
	\begin{cases}
	-k\widehat{B}_{p-k}p+\widehat{B}_{p-k}p^2-\binom{k+2}{3}\widehat{B}_{p-k-2}p^3 \pmod{p^4} & \text{if $k$ is odd,}\\
	\binom{k+1}{2}\widehat{B}_{p-k-1}p^2-\frac{2k+1}{2}\widehat{B}_{p-k-1}p^3 \pmod{p^4} & \text{if $k$ is even,}
	\end{cases}\\
	&\sum_{a=1}^{p-1}a^{2p-k-1} \equiv
	\begin{cases}
	-(k+1)\widehat{B}_{2p-k-1}p+2\widehat{B}_{2p-k-1}p^2-\binom{k+3}{3}\widehat{B}_{p-k-2}p^3 \pmod{p^4} & \text{if $k$ is odd,}\\
	\binom{k+2}{2}\widehat{B}_{2p-k-2}p^2-(2k+3)\widehat{B}_{p-k-1}p^3 \pmod{p^4} & \text{if $k$ is even,}
	\end{cases}\\
	&\sum_{a=1}^{p-1}a^{3p-k-2} \equiv 
	\begin{cases}
	-(k+2)\widehat{B}_{3p-k-2}p+3(2\widehat{B}_{2p-k-1}-\widehat{B}_{p-k})p^2-\binom{k+4}{3}\widehat{B}_{p-k-2}p^3 \pmod{p^4} & \text{if $k$ is odd,}\\
	\binom{k+3}{2}(2\widehat{B}_{2p-k-2}-\widehat{B}_{p-k-1})p^2 -\frac{3}{2}(2k+5)\widehat{B}_{p-k-1}p^3 \pmod{p^4} & \text{if $k$ is even.}
	\end{cases}
	\end{align*}}
	\label{app lemma}\end{lemma}
\begin{proof}
	The first congruence is deduced by Proposition \ref{FMZV's properties} (\ref{RZA4}). Let $m$ be an integer greater than $3$. Then, by Faulhaber's formula, we have
	\begin{equation*}
	\begin{split}
	\sum_{a=1}^{p-1}a^m &= \frac{1}{m+1}\sum_{i=1}^{m+1}\binom{m+1}{i}B_{m+1-i}p^i\\
	&\equiv pB_m+\frac{p^2}{2}mB_{m-1}+\frac{p^3}{6}m(m-1)B_{m-2}+\frac{p^4}{24}m(m-1)(m-2)B_{m-3} \pmod{p^4}
	\end{split}
	\end{equation*}
	since $\frac{1}{m+1}\binom{m+1}{i}B_{m+1-i}p^i=p^4\binom{m}{i-1}pB_{m+1-i}\frac{p^{i-5}}{i}$ and $pB_{m+1-i}, p^{i-5}/i$ are $p$-adic integers for $i \geq 5$. If $m=p-k, 2p-k-1$, or $3p-k-2$, then $\frac{p^4}{24}m(m-1)(m-2)B_{m-3}$ also vanishes mod $p^4$ since $B_{m-3}$ is a $p$-adic integer by the von Staudt-Clausen theorem. Since $B_m=0$ if $m\geq 3$ is odd, we have
	\begin{align*}
	&\sum_{a=1}^{p-1}a^{p-k} \equiv
	\begin{cases}
	pB_{p-k}+\frac{p^3}{6}k(k+1)B_{p-k-2} \pmod{p^4} & \text{if $k$ is odd},\\
	\frac{p^2}{2}(p-k)B_{p-k-1} \pmod{p^4} & \text{if $k$ is even},
	\end{cases}\\
	&\sum_{a=1}^{p-1}a^{2p-k-1} \equiv
	\begin{cases}
	pB_{2p-k-1}+\frac{p^3}{6}(k+1)(k+2)B_{2p-k-3} \pmod{p^4} & \text{if $k$ is odd},\\
	\frac{p^2}{2}(2p-k-1)B_{2p-k-2} \pmod{p^4} & \text{if $k$ is even},
	\end{cases}\\
	&\sum_{a=1}^{p-1}a^{3p-k-2} \equiv
	\begin{cases}
	pB_{3p-k-2}+\frac{p^3}{6}(k+2)(k+3)B_{3p-k-4} \pmod{p^4} & \text{if $k$ is odd},\\
	\frac{p^2}{2}(3p-k-2)B_{3p-k-3} \pmod{p^4} & \text{if $k$ is even}.
	\end{cases}
	\end{align*}
	We obtain the desired formulas by changing each $B_m$ to $\widehat{B}_m$ and by Kummer's congruence (\ref{Kummer-Sun}), e.g. 
	\[
	\widehat{B}_{3p-k-\varepsilon -2} \equiv 2\widehat{B}_{2p-k-\varepsilon -1}-\widehat{B}_{p-k-\varepsilon} \pmod{p^2}
	\]
	where $\varepsilon \in \{ 0, 1 \}$.
\end{proof}
\begin{proof}[Proof of Theorem $\ref{appB thm}$]
	We put $\omega := \omega_p$. For $a \in \mathbb{Z}_p^{\times}$, we define $\langle a \rangle$ to be $\omega (a)^{-1}a$. Then $\langle a \rangle \in 1+p\mathbb{Z}_p$. Let \[
	\log_p \colon \mathbb Z_p^\times\rightarrow p\mathbb Z_p 
	\]
	be {\em a $p$-adic logarithm function} and let $\exp$ be {\em the $p$-adic exponential function}, which is the right inverse of $\log_p$. Then we have
	\begin{equation}
	\langle a \rangle = \exp ( \log_p \langle a \rangle ) = \exp (\log_p a) = \exp \left( \frac{1}{p-1}\log_p (a^{p-1}) \right) .
	\label{<a>}\end{equation}
	Let $\alpha := \frac{1}{p-1}\log_p (a^{p-1}) = \frac{1}{p-1}\log_p (1+pq_p(a)) \in p\mathbb{Z}_p$ where $q_p(a)$ is the Fermat quotient (See Definition \ref{def of FQ}). Since we can check easily that $\alpha^n/n! \equiv 0 \pmod{p^4}$ if $n \geq 4$, we have
	\begin{equation}
	\langle a \rangle \equiv 1+\alpha +\frac{\alpha^2}{2}+\frac{\alpha^3}{6} \pmod{p^4}.
	\label{exp}\end{equation} 
	By definition of a $p$-adic logarithm, we have
	\[
	\log_p (1+pq_p(a)) = \sum_{n=1}^{\infty} \frac{(-1)^{n-1}}{n}p^nq_p(a)^n \equiv pq_p(a)-\frac{1}{2}p^2q_p(a)^2+\frac{1}{3}p^3q_p(a)^3 \pmod{p^4}.
	\]
	Therefore, by the congruence (\ref{exp}), we have
	{\footnotesize \begin{equation*}
	\begin{split}
	\langle a \rangle &\equiv 1+\frac{1}{p-1}\left( pq_p(a)-\frac{1}{2}p^2q_p(a)^2+\frac{1}{3}p^3q_p(a)^3 \right) +\frac{1}{2}\left( \frac{1}{p-1} \right)^2 \left( pq_p(a)-\frac{1}{2}p^2q_p(a)^2+\frac{1}{3}p^3q_p(a)^3 \right)^2 \\
	&\hspace{8mm}+\frac{1}{6}\left( \frac{1}{p-1} \right)^3 \left( pq_p(a)-\frac{1}{2}p^2q_p(a)^2+\frac{1}{3}p^3q_p(a)^3 \right)^3 \\
	&\equiv 1-q_p(a)p-(q_p(a)-q_p(a)^2)p^2-\left( q_p(a)-\frac{3}{2}q_p(a)^2+q_p(a)^3 \right) p^3 \pmod{p^4}.
	\end{split}
	\end{equation*}
}Note that the last congruence is obtained by $\frac{1}{p-1}=-\sum_{i=0}^{\infty}p^i$, $\left( \frac{1}{p-1}\right)^2 = \sum_{i=1}^{\infty}ip^{i-1}$, and $\left( \frac{1}{p-1} \right)^3 = -\sum_{i=2}^{\infty}i(i-1)p^{i-2}$. By the binomial expansion, we have
	\begin{equation*}
	\begin{split}
	\langle a \rangle^k &\equiv 1+k\left\{ -q_p(a)p-(q_p(a)-q_p(a)^2)p^2-\left( q_p(a)-\frac{3}{2}q_p(a)^2+q_p(a)^3 \right) p^3\right\} \\
	&\hspace{7mm}+\binom{k}{2}\left\{ -q_p(a)p-(q_p(a)-q_p(a)^2)p^2-\left( q_p(a)-\frac{3}{2}q_p(a)^2+q_p(a)^3 \right) p^3\right\}^2 \\
	&\hspace{7mm}+\binom{k}{3}\left\{ -q_p(a)p-(q_p(a)-q_p(a)^2)p^2-\left( q_p(a)-\frac{3}{2}q_p(a)^2+q_p(a)^3 \right) p^3\right\}^3 \\
	&\equiv 1-kq_p(a)p-\left\{ kq_p(a)-\binom{k+1}{2}q_p(a)^2 \right\} p^2 \\
	&\hspace{7mm}- \left\{ kq_p(a)-\left( k^2+\frac{k}{2} \right) q_p(a)^2+\frac{k^3+3k^2-4k}{6}q_p(a)^3 \right\} p^3\pmod{p^4}
	\end{split}
	\end{equation*}
	Hence, by the equality (\ref{B_1}), we have
	{\footnotesize \begin{equation*}
	\begin{split}
	&pB_{1, \omega^{-k}} = \sum_{a=1}^{p-1}\omega^{-k}(a)a = \sum_{a=1}^{p-1}a^{1-k}\langle a \rangle^k \\
	&\equiv \sum_{a=1}^{p-1}a^{1-k}-kp\sum_{a=1}^{p-1}a^{1-k}q_p(a)-\left\{ k\sum_{a=1}^{p-1}a^{1-k}q_p(a)-\binom{k+1}{2}\sum_{a=1}^{p-1}a^{1-k}q_p(a)^2 \right\}p^2 \\
	&\hspace{5mm} -\left\{ k\sum_{a=1}^{p-1}a^{1-k}q_p(a)-\left( k^2+\frac{k}{2} \right) \sum_{a=1}^{p-1}a^{1-k}q_p(a)^2+\frac{k^3+3k^2-4k}{6}\sum_{a=1}^{p-1}a^{1-k}q_p(a)^3 \right\}p^3 \pmod{p^4}.
	\end{split}
	\end{equation*}
}Since
	\begin{align*}
	&p\sum_{a=1}^{p-1}a^{1-k}q_p(a) = -\sum_{a=1}^{p-1}a^{1-k}+\sum_{a=1}^{p-1}a^{p-k},\\
	&p^2\sum_{a=1}^{p-1}a^{1-k}q_p(a)^2 = \sum_{a=1}^{p-1}a^{1-k}-2\sum_{a=1}^{p-1}a^{p-k}+\sum_{a=1}^{p-1}a^{2p-k-1},\\
	&p^3\sum_{a=1}^{p-1}a^{1-k}q_p(a)^3 =-\sum_{a=1}^{p-1}a^{1-k}+3\sum_{a=1}^{p-1}a^{p-k}-3\sum_{a=1}^{p-1}a^{2p-k-1}+\sum_{a=1}^{p-1}a^{3p-k-2},\end{align*}
	we have
	{\footnotesize \begin{equation*}
	\begin{split}
	pB_{1, \omega^{-k}} &\equiv \left\{ \frac{k^3+6k^2+5k+6}{6}+\left( k^2+\frac{3}{2}k\right) p +kp^2 \right\} \sum_{a=1}^{p-1}a^{1-k} -\left\{ \frac{k^3+5k^2}{2}+(k^2+k)p+kp^2 \right\} \sum_{a=1}^{p-1}a^{p-k} \\
	&\hspace{5mm}+\left\{ \frac{k^3+4k^2-3k}{2}+\left( k^2+\frac{k}{2} \right) p \right\} \sum_{a=1}^{p-1}a^{2p-k-1} -\frac{k^3+3k^2-4k}{6}\sum_{a=1}^{p-1}a^{3p-k-1} \pmod{p^4}.
	\end{split}
	\end{equation*}
}If $k$ is odd, by Lemma \ref{app lemma}, we have
	{\footnotesize \begin{equation*}
	\begin{split}
	pB_{1, \omega^{-k}} &\equiv \left\{ \frac{k^3+6k^2+5k+6}{6}+\left( k^2+\frac{3}{2}k\right) p +kp^2 \right\} \left\{ -(k-1)(3\widehat{B}_{p-k}-3\widehat{B}_{2p-k-1}+\widehat{B}_{3p-k-2})p  \right.\\
	&\hspace{5mm}\left. -\binom{k+1}{3}\widehat{B}_{p-k-2}p^3 \right\}\\
	&\hspace{5mm} -\left\{ \frac{k^3+5k^2}{2}+(k^2+k)p+kp^2 \right\} \left\{ -k\widehat{B}_{p-k}p+\widehat{B}_{p-k}p^2-\binom{k+2}{3}\widehat{B}_{p-k-2}p^3\right\} \\
	&\hspace{5mm}+\left\{ \frac{k^3+4k^2-3k}{2}+\left( k^2+\frac{k}{2} \right) p \right\} \left\{ -(k+1)\widehat{B}_{2p-k-1}p+2\widehat{B}_{2p-k-1}p^2-\binom{k+3}{3}\widehat{B}_{p-k-2}p^3\right\} \\
	&\hspace{5mm}-\frac{k^3+3k^2-4k}{6}\left\{ -(k+2)\widehat{B}_{3p-k-2}p+3(2\widehat{B}_{2p-k-1}-\widehat{B}_{p-k})p^2-\binom{k+4}{3}\widehat{B}_{p-k-2}p^3\right\}\\
	&\equiv \left\{ \frac{k^2-k+6}{2}\widehat{B}_{p-k}-(k^2-2k+3)\widehat{B}_{2p-k-1}+\frac{k^2-3k+2}{2}\widehat{B}_{3p-k-2} \right\} p \\
	&\hspace{5mm}+ k(\widehat{B}_{p-k}-\widehat{B}_{2p-k-1})p^2-k\widehat{B}_{p-k-2}p^3 \pmod{p^4}.
	\end{split}
	\end{equation*}
}If $k$ is even, then
	\begin{equation*}
	\begin{split}
	pB_{1, \omega^{-k}} &\equiv \left\{ \frac{k^3+6k^2+5k+6}{6}+\left( k^2+\frac{3}{2}k\right) p +kp^2 \right\} \left\{ \binom{k}{2}(2\widehat{B}_{p-k-1}-\widehat{B}_{2p-k-2})p^2 \right\}\\
	&\hspace{5mm} -\left\{ \frac{k^3+5k^2}{2}+(k^2+k)p+kp^2 \right\} \left\{ \binom{k+1}{2}\widehat{B}_{p-k-1}p^2-\frac{2k+1}{2}\widehat{B}_{p-k-1}p^3\right\} \\
	\end{split}
	\end{equation*}
	\begin{equation*}
	\begin{split}
	&\hspace{5mm}+\left\{ \frac{k^3+4k^2-3k}{2}+\left( k^2+\frac{k}{2} \right) p \right\} \left\{ \binom{k+2}{2}\widehat{B}_{2p-k-2}p^2-(2k+3)\widehat{B}_{p-k-1}p^3\right\} \\
	&\hspace{5mm}-\frac{k^3+3k^2-4k}{6}\left\{ \binom{k+3}{2}(2\widehat{B}_{2p-k-2}-\widehat{B}_{p-k-1})p^2 -\frac{3}{2}(2k+5)\widehat{B}_{p-k-1}p^3\right\}\\
	&\equiv -3k(\widehat{B}_{p-k-1}-\widehat{B}_{2p-k-2})p^2 \pmod{p^4}.
	\end{split}
	\end{equation*}
	This completes the proof.
\end{proof}
\section{Table of sufficient conditions for congruences}
\label{sec:Table of sufficient conditions for congruences}
In this appendix, we give a sufficient condition that the congruence obtained as each $p$-component of the special value of a FMP holds. First, we list the equalities whose $p$-component congruences hold for all prime numbers: {\footnotesize Theorem \ref{introthm} (\ref{fn eq 1}) and (\ref{fn eq 2}), Proposition \ref{FMZV's properties} (\ref{reverse for FMZV}) and (\ref{Hoffman's duality for FMZV}), Proposition \ref{reverse prop} (\ref{reverse A_2}), (\ref{reverse1}), and (\ref{reverse2}), Theorem \ref{MTA} (\ref{main functional equation}), Corollary \ref{MTB} (\ref{substitution functional equation}), Remark \ref{1-var remark} (\ref{gen. of Hoffman's duality}) and (\ref{Hoffman-Zhao duality}), Theorem \ref{MTC} (\ref{A_n formula}), Corollary \ref{zeta corollary} (\ref{nonstar-star}), (\ref{nonstar and star2}), and (\ref{non-star and star}), Proposition \ref{dist} (\ref{dist0}), Proposition \ref{A_2 aux} (\ref{trivial formula}), Subsection \ref{subsec:Relation between Ono-Yamamoto's FMPs and our FMPs} (\ref{reverse for OY}), Proposition \ref{OY-SS} (\ref{OYSS}), and Corollary \ref{OY k_1, k_2} (\ref{OY2}) and (\ref{OY3})}. Here, we understand non-star sum (e.g. the $p$-component of $\zeta_{\mathcal{A}}(\Bbbk)$) as $0$ if $p \leq \dep (\Bbbk)$.  
\begin{table}[htb]
	\caption{Lemma \ref{vanish lemma}, Proposition \ref{FMZV's properties}, Corollary \ref{zeta corollary}, Theorem \ref{PPT theorem}, Lemma \ref{multi-FPL}, Theorem \ref{dist thm}, Corollary \ref{ZS lemma}, Lemma \ref{aux}, Lemma \ref{A_2 aux}, Proposition \ref{special values in A}, Theorem \ref{A_2 theorem}, Lemma \ref{FP values}, Proposition \ref{Fermat quotient}, Proposition \ref{four prop}, Lemma \ref{harmonic alt}, Lemma \ref{A_2 alt}, Theorem \ref{FHMP special values}, Remark \ref{TZ remark}, Theorem \ref{A_2 multiple thm}, Proposition \ref{OY special values}, Theorem \ref{appB thm}, and Corollary \ref{app cor}}
	\begin{tabular}{|c||c|c|c|c|}
		\hline
		number &(\ref{A_2 vanish})&(\ref{Zhou-Cai})&(\ref{RZA3})&(\ref{RZA4})\\ \hline
		condition &$p>n$&$p>mk+2$&$p>k+3$&$p>k+4$\\ \hline \hline
		number &(\ref{dep2 for FMZV})&(\ref{dep2 A_2}) and (\ref{dep2 A_2 star})&(\ref{dep3 for FMZV})&(\ref{sym sum})\\ \hline
		condition &$p>k_1+k_2-1$&$p>w+1$ &$p>w'$&$p>\wt (\Bbbk)+1$\\ \hline \hline
		number &(\ref{A-S for FMZV})&(\ref{Tauraso-Zhao})&(\ref{PPT1}) and (\ref{PPT2})&(\ref{1^m, 1}), (\ref{MT Fn-Eq}), (\ref{dist1}) - (\ref{m, -1})\\ \hline
		condition &$p>k_2$&$p>mk+1$&$p>k_1+k_2+1$&$p>m+1$\\ \hline \hline
		number &(\ref{k_1, k_2, -1, non-star}) and (\ref{k_1, k_2, -1, star})&(\ref{A_2, m, -1})&(\ref{57}) - (\ref{60})&(\ref{easy gen of Sun}) and (\ref{gen of Sun})\\ \hline
		condition &$p>w+1 $&$p>m+2$&$p>w+2$&$p>m+1$\\ \hline \hline
		number &(\ref{k_1, k_2, 2, non-star}) - (\ref{k_1, k_2, 2, star})&(\ref{A_2, m, 2}) and (\ref{A_2, m, 1/2})&(\ref{FP1}) - (\ref{A, 1^3, 1/2})&(\ref{91}) - (\ref{four})\\ \hline
		condition &$p>w+1$&$p>m+2$&$p>3$&$p>5$\\ \hline \hline
		number &(\ref{odd}) and (\ref{even bullet})&(\ref{B to q})&(\ref{-1, -1, 1}) - (\ref{-1, 1, -1})&(\ref{1, -1, -1 and -1, -1, 1}) - (\ref{1, -1, 1})\\ \hline
		condition  &$p>k_1+k_2+1$&$p>m+2$&$p>w+1$&$p>3$\\ \hline \hline
		number &(\ref{new Zhao}) and (\ref{new Zhao star})&(\ref{k_1=k_2=1}) and (\ref{k_1=k_2=1, star})&(\ref{109}) and (\ref{1102})&$\begin{array}{c}(\ref{MV1'}) \text{-} (\ref{MV1.5}), (\ref{MV2.5}), \\(\ref{MV2}), (\ref{reverse mult1}), (\ref{MV3})\end{array}$\\ \hline
		condition &$p>k_1+k_2+2$&$p>3$&$p>5$&$p>w+1$\\ \hline \hline
		number &(\ref{B to q 2})&$\begin{array}{l}(\ref{MV2.7}), (\ref{reverse mult2}), \\ (\ref{MV4'}), (\ref{MV4})\end{array}$&(\ref{MV5'''}) - (\ref{MV6})&(\ref{MV7'}) and (\ref{MV7})\\ \hline
		condition &$p>m+2$&$p>m+1$& $p>w'+1$&$p>w+2$\\ \hline \hline
		number &(\ref{(1,2) 2, 1/2}) - (\ref{2, 1, 1/2=0}), (\ref{1^3 non-star})&(\ref{A_2 mult nonstar}) and (\ref{A_2 multiple})&(\ref{110}) - (\ref{113})&(\ref{OYsp1})\\ \hline 
		condition &$p>3$&$p>w+1$&$p>3$&$p>\wt (\Bbbk)+1$\\ \hline \hline
		number  &(\ref{OY 1^2}) - (\ref{OY 1^3})&(\ref{A_3 app odd}) and (\ref{A_3 app even})&(\ref{A_2 app})&(\ref{A app})\\ \hline
		condition &$p>3$&$p>k+4$&$p>k+3$&$p>k+2$\\ \hline
	\end{tabular}
\end{table}

For example, (\ref{Zhou-Cai}) in the above table means that 
\[
\zeta_{p-1}^{\star}(\{ k \}^m) := \sum_{p-1\geq n_1\geq \cdots \geq n_m\geq 1}\frac{1}{n_1^k\cdots n_m^k} \equiv k\frac{B_{p-mk-1}}{mk+1}p \pmod{p^2}
\]
holds for any positive integers $m, k$ and any prime $p$ satisfying $p> mk+2$. The mod $p$ version
\[
\zeta_{p-1}^{\star}(\{k\}^m) \equiv 0 \pmod{p}
\]
holds for $p>mk+1$ (cf. \cite[Theorem 4.4]{Ho}).

\section*{Acknowledgements}
The authors would like to thank Prof. Masanobu Kaneko for helpful comments, correcting several historical
mistakes, and giving them his manuscript for the finite multiple zeta values.
They also thank their advisor Prof.\ Tadashi Ochiai and Sho Ogaki for reading the manuscript carefully.

\bigskip

\address{\sc
	Department of Mathematics, Graduate School of Science Osaka University Toyonaka, Osaka 560-0043 Japan
}\\
{\it E-mail address}: {k-sakugawa@cr.math.sci.osaka-u.ac.jp}

\vspace{5mm}
\noindent

\address{\sc
	Department of Mathematics, Graduate School of Science Osaka University Toyonaka, Osaka 560-0043 Japan
}\\
{\it E-mail address}: {shinchan.prime@gmail.com}

\end{document}